\documentclass[12pt]{article}
\usepackage{amsfonts}
\usepackage{amsmath}
\usepackage{mathrsfs}
\usepackage{amssymb}
\usepackage{color,amscd,array,graphicx,mathdots,bm,amsfonts,epsfig,amsmath}
\usepackage{amsthm}
\usepackage{enumerate}
\usepackage{geometry}
\usepackage{caption}
\usepackage{subfigure}
\usepackage[T1]{fontenc}
\usepackage{authblk}
\usepackage{blindtext}

\newtheorem{theorem}{Theorem}[section]

\newtheorem{lemma}[theorem]{Lemma}
\newtheorem{example}[theorem]{Example}
\newtheorem{proposition}[theorem]{Proposition}
\newtheorem{remark}[theorem]{Remark}
\newtheorem{hypothesis}{Hypothesis}

\usepackage{graphicx,color,float} 

\usepackage[colorlinks=true,citecolor=blue]{hyperref}

\makeatletter
\newcommand{\figcaption}{\def\@captype{figure}\caption}
\newcommand{\tabcaption}{\def\@captype{table}\caption}
\makeatother

\def\ga{{\gamma}}
\def\om{{\omega}}

\let\emptyset\varnothing

\def\by{\mathbf{y}}
\def\bz{\mathbf{z}}
\def\bx{\mathbf{x}}

\def\bt{\mathbf{t}}
\def\bs{\mathbf{s}}
\def\br{\mathbf{r}}

\newfam\msbmfam\font\tenmsbm=msbm10\textfont
\msbmfam=\tenmsbm\font\sevenmsbm=msbm7
\scriptfont\msbmfam=\sevenmsbm

\def\EE{\mathbb E}\def\PP{\mathbb P}\def\NN{\mathbb N}
\def\RR{\mathbb R}
\def\PP{\mathbb P}\def\sP{{\mathscr P}}
\def\cK{{\cal K}}
\def\cF{{\cal F}}\def\cH{{\cal H}}
\def\cI{{\cal I}}\def\cM{{\cal M}}\def\cJ{{\cal J}}
\def\cS{{\cal S}}

\def\1{\mathbf 1}

\def\<{\left<}\def\>{\right>}

\def\({\left(}\def\){\right)}

\usepackage{graphicx} 
\usepackage{epstopdf}
\usepackage{float}

\providecommand{\keywords}[1]
{
  \textbf{\textit{Keywords:\ }} #1
}

\numberwithin{equation}{section}

\begin{document}

\title{On mean-field super-Brownian motions}

\author[a]{Yaozhong Hu \thanks{Supported by an NSERC Discovery grant and a startup fund from University of Alberta at Edmonton. Email: yaozhong@ualberta.ca}}
\author[a]{Michael A. Kouritzin \thanks{Supported by an NSERC Discovery grant. Email: michaelk@ualberta.ca}}
\author[b]{Panqiu Xia\thanks{Email: px@math.ku.dk}}
\author[c]{Jiayu Zheng \thanks{Supported by NSFC grant 11901598.  Email: jyzheng@smbu.edu.cn}}
\affil[a]{Department of Mathematical and Statistical Sciences, University of Alberta, Edmonton, AB, T6G 2G1, Canada}
\affil[b]{Department of Mathematical Sciences,
University of Copenhagen, Copenhagen \O, 2100, Denmark}
\affil[c]{Faculty of Computational Mathematics and Cybernetics, Shenzhen MSU-BIT University, Shenzhen, Guangdong, 518172, China}

\date{}

\maketitle

\begin{abstract}
The mean-field stochastic partial differential equation (SPDE) corresponding to a mean-field super-Brownian motion (sBm) is obtained and studied. 
In this mean-field sBm,  the branching-particle lifetime is allowed to depend upon the   probability  distribution of the sBm itself, producing an SPDE whose space-time 
white noise coefficient has, in addition to the typical sBm square root,  
an extra factor that is a function of the  probability law  of the density of the 
mean-field sBm.
This novel mean-field SPDE is thus motivated by population models where 
things like overcrowding and isolation can affect growth.
A two step approximation method is employed to show the existence for this SPDE
under general conditions. 
Then, mild moment conditions are imposed to get uniqueness. 
Finally, smoothness of the SPDE solution is established under a further simplifying condition.
\end{abstract}


\section{Introduction}
The classical mean-field theory was widely used in statistical mechanics to study e.g. the derivation of Boltzmann or Vlasov equations in the kinetic gas theory.  This theory has also been applied in quantum mechanics, quantum chemistry and so forth. 
In the late 2000's, Larsy and Lions (see \cite{jjm-07-lasry-lions} and references therein) generalized this theory to approximate the Nash equilibrium with a large number of players that can be described as a system of exchangeable stochastic differential equations (SDEs). 
Moreover, there is a series of more recent papers focusing on mean-field  
backward SDEs (see \cite{amo-11-andersson-djehiche,ap-09-buckdahn-djehiche-li-peng,ap-17-buckdahn-li-peng-rainer} etc.). 
Still, it is natural to extend the mean-field theory for SDEs to  infinite dimension 
and, in particular, consider mean-field stochastic partial differential equations 
(SPDEs).
Yet, hitherto little has been done. 
To the authors' best knowledge, there are only a very limited number of works 
on mean-field SPDEs (see \cite{jota-18-dumitrescu-oksendal-sulem,cam-19-tang-meng-wang}). 

The mean-field SPDE we study arises from a 
Dawson-Watanabe-style, high-density branching-process limit with some naturally 
modified branching mechanism.  
Suppose that there is a population of particles, each performing Brownian motions 
on $\mathbb R$, with exponentially distributed lifetimes.  
At the end of each individual particle's life, it gives birth to a number of offsprings according to the Dawson-Watanabe branching mechanism.  
Then, it is well-known (c.f. Perkins \cite{springer-02-perkins})  that the empirical measures $X_t^n$ of this  Dawson-Watanabe branching particle system converges to 
the super-Brownian motion (sBm), described by a measure-valued stochastic partial differential equation.  
Furthermore, the one dimensional sBm $X_t$, 
considered as an $\cM_F(\RR)$-valued process, has a Lebesgue density $X(t,x)$ 
for all $t\in \RR_+ $ almost surely. 
Thus from the random field point of view, one can write (c.f. Xiong \cite{ws-13-xiong})  the sBm as the unique weak random field solution to the following SPDE
\[
\frac{\partial}{\partial t}X_t(x)=\frac{1}{2}\Delta X_t(x)+ \sqrt{\gamma X_t(x)}\dot{W}(t,x)\,,
\]
where $\gamma>0$ is the branching rate, $\Delta$ denotes the Laplacian operator in space and $\dot W=
\frac{\partial ^2}{\partial t\partial x}W$ is the space-time white noise on $\RR_+\times \RR$
(i.e. $W$ is the Brownian sheet).

Suppose now in the finite particle prelimit that each individual branching particle's lifetime is affected
by the entire population (perhaps through overcrowding or isolation) 
so it dies and branches accordingly.     
Then, in the limit, the branching rate $\ga$ (or more precisely the particle lifetime) depends upon 
the probability distribution of the population. 
This leads us to consider the 
following mean-field sBm, whose Lebesgue density $X(t,x)$ satisfies the following SPDE
\begin{align}\label{sbm}
\frac{\partial}{\partial t} X_t(x) =\frac{1}{2} \Delta X_t(x) + \sigma\big(t,x,\PP_{X_t(x)}\big)\sqrt{X_t(x)} \dot{W}(t, x),
\end{align}
where $\PP_{X_t(x)}$ is the probability law of  the real valued random variable $X_t(x)$.

Superprocesses or branching processes have been widely applied in natural sciences. The use of branching processes to approximate large-scale networks is one of those successful examples (c.f. \cite{bath-14-eckhoff,cam-16-hofstad}), while large-scale networks are employed to model real-world problems such as the spread of diseases (c.f. Strogatz \cite{nature-01-strogatz}), the evolution of complex biochemical reaction systems (c.f. \cite{springer-15-anderson-kurtz,springer-19-feinberg}), and so on. This approximation is based on the fact that the global structures of a complex network are determined by their local properties and the network behaves locally like a tree structure (see Eckhoff \cite{bath-14-eckhoff}). A typical example is that the homogeneous Erd\H{o}s-R\'{e}nyi random graph can be 
adequately approximated by the  Poisson-Galton–Watson process (c.f. Van Der Hofstad \cite{cam-16-hofstad}).  Additionally, if one takes the spatial movement into consideration, the associated branching particle system could be a good replacement of the branching process in the approximation of corresponding networks with spatial motion. In fact, the sBm can be understood as a type of scaling limit of the reaction network $\{S\to 2S, S\to \emptyset\}$ with the same reaction rates and where the molecules of species $S$ move as independent Brownian motions (c.f. \cite[etc.]{ap-94-blount,aap-15-pfaffelhuber-popovic,arxiv-20-popovic-veber} for other types of scaling of reaction networks with spatial motions). Assuming the reaction rates depend also on the distribution of species in the system, the corresponding scaling limit should satisfy a mean-field sBm of the form \eqref{sbm}. Other than the sBm, the scaled $\Lambda$-Fleming-Viot branching system converges to a stochastic Fisher-KPP equation (c.f. \cite{ejp-10-barton-etheridge-veber,ejp-20-etheridge-veber-yu}), which describes the population evolution of competing species. We are interested in knowing whether the techniques used in this paper could be potentially applied to derive a mean-field stochastic Fisher-KPP equation or other SPDEs arising from the large-scale networks with distribution dependent coefficients and
consequently to establish the existence, uniqueness and regularity results of the solution.

 On the other hand, one may also obtain this equation \eqref{sbm} from the average of weakly interacting sBm's (cf. Overbeck \cite{ap-96-overbeck}). Let $\mathbf{X}^N=(X^1,\dots, X^N)$ be an $N$-type sBm's that is the solution to the martingale problem: for any $\phi=\phi^1\otimes \dots \otimes \phi^N$ with $\phi^i\in \cS(\RR)$ the Schwartz space of functions on $\RR$ for all $i=1,\dots, N$, the process
\[
\mathbf{M}^N_t(\phi)=\mathbf{X}^N_t(\phi)-\mathbf{X}^N_0(\phi)-\int_0^t\mathbf{X}^N_s(\Delta \phi)ds
\]
 is a continuous square integrable martingale with quadratic variation
 \[
\langle\mathbf{M}^N(\phi)\rangle_t=\sum_{j=1}^N\int_0^t\int_{\RR}\sigma\Big(s,x,\frac{1}{N}\sum_{i=1}^N\delta_{X^i_s(x)}\Big)^2(\phi^j(x))^2X^j_s(dx)ds.
 \]
Then, as $N\to \infty$,  $\frac{1}{\sqrt{N}}\sum_{i=1}^NX^i$ shall heuristically converges to a random filed satisfying equation \eqref{sbm} with $\sigma (t,x,\mu)=\sigma(t,x,\EE (X_{\mu}))$ where $X_{\mu}$ denotes a random variable of distribution $\mu$. Some related results can be found in e.g. Overbeck \cite{pre-94-overbeck}. We are not going to justify this limit in the present paper.

We shall focus on the existence, uniqueness and regularity of the solution to 
equation \eqref{sbm}. 
The first difficulty that we encounter is that there exists no readily-applicable, 
fully-developed theory 
on the Fokker-Planck-Kolmogorov equation associated with \eqref{sbm}.
So, we cannot follow the approach used in finite dimensional case (c.f. \cite{amo-11-andersson-djehiche,ap-09-buckdahn-djehiche-li-peng,ap-17-buckdahn-li-peng-rainer}) to study the existence and uniqueness of solutions to the associated Fokker-Planck-Kolmogorov equation first, and then to solve the mean field equation. 

Nevertheless, the anticipation of solutions to \eqref{sbm} is well justified.
Due to the appearance of the branching character (the $\sqrt{X_t(x)} $ factor in front of the noise), it is natural to use a branching particle system to approximate this equation.  
Assuming that such approximation is done and some high-density limit exists, 
one presumably obtains that every limit point $X= X_t(  dx,  \om)$ is an $\cM_F(\RR)$-valued Markov process. 
 One should be careful, that this limiting process $X$ is different from the stochastic McKean-Vlasov equation as in Dawson and Vaillancourt \cite{nodea-95-dawson-vaillancourt}.  Indeed, the
noise coefficient in \eqref{sbm} is a function of the probability law $\PP_{X_t}$ of the solution $X_t$ as a finite random measure. In comparison, coefficients in \cite{nodea-95-dawson-vaillancourt} as functions of finite measures depend on the random measure $X_t$ itself.   
Notice that $\PP_{X_t}$ is a probability measure on the   space  of finite measures
$  \cM_F(\RR)$. 
If we want to show that $X_t$ satisfies    equation \eqref{sbm} in certain sense, we need to verify the absolute continuity of $X_t$ with respect to the Lebesgue measure for all $t>0$ almost surely, namely, the existence of $X_t(x,\om)  $ such that
$X_t( dx, \om)=X_t(x,\om) dx$. 
This (random) measure $X_t( dx, \om) $ may or may not have such a Lebesgue density.
The classical methods to check absolute continuity are
 based on the moment duality or Laplace functional and require
an explicit form of the corresponding martingale problem, whereas the 
presentation of the martingale problem for our limit $X$ depends on 
$\sigma(t,x,\PP_{X_t(x)})$, which is not well-defined without the 
absolute continuity of $X_t$. 
This dilemma is one of the main difficulties in studying solutions to \eqref{sbm}.
Further, even if the absolute continuity is established so 
$X_t( dx, \om)=X_t(x,\om) dx$, the law (now as a measure of $\RR$) of $X_t(x)$  is not a continuous functional of $\PP_{X_t(x)}$ with respect to the Wasserstein metric. 
Thus, $\sigma(t,x,\PP_{X_t(x)})$ has some intrinsic singularity with respect to the probability measure $\PP_{X_t(x)}$, which will force us to use non-standard
methods.

To overcome these difficulties in the context of existence, we apply a 
two-step approximation (see e.g.  Ji et al. \cite{arxiv-21-ji-xiong-yang}). 
Let $\sP(\RR_+)$ denote the collection of all Borel probability measures on $\RR_+$  equipped with the weak topology,  and let $\cM(\RR;\sP(\RR_+))$ be the collections of measurable functions on $\RR$ with values in $\sP(\RR_+)$. 
In the first step, we fix $\delta>0$, and prove the existence of the pair $(X^{\delta},Y^{\delta})$ that solves the equation
\begin{align}\label{sbm1}
\begin{cases}
\displaystyle\frac{\partial}{\partial t} X^{\delta}_t(x) =\frac{1}{2} \Delta X^{\delta}_t(x) + \widetilde{\sigma}_{ \delta}(t,x, \PP_{Y^{\delta}_t})\sqrt{X^{\delta}_t(x)} \dot{W}(t, x),\\
\displaystyle Y_t^{\delta}(x)=\int_{\RR}p_{\delta}(x-y)X^{\delta}_t(dy),
\end{cases}
\end{align}
with a non-random initial condition $X_0\in \cM_F(\RR)$, where $p_{\delta}(x)=\frac{1}{\sqrt{2\pi\delta}}e^{-\frac{x^2}{2\delta}}$ denotes the heat kernel,   $\PP_{Y^{\delta}_t}=\PP_{Y^{\delta}_t(\cdot)}$ is understood as an element in $\cM(\RR;\sP(\RR_+))$, and $\widetilde{\sigma}_{ \delta}:\RR_+\times \RR\times \cM(\RR;\sP(\RR_+))\to\RR_+$ is given by
\begin{align}\label{def_tsgm}
\widetilde{\sigma}_{ \delta}(t,x,\Gamma)=\int_{\RR}dy p_{\delta}(x-y)\sigma\big(t,y,\Gamma(y)\big).
\end{align} 
In the next step, we prove the tightness of $\{X^{\delta}\}_{\delta>0}$
and $\{Y^{\delta}\}_{\delta>0}$ in the space $C([0,T]\times \RR;\RR)$  for any $T>0$. 
Then, we can find a random field limit point in distribution as $\delta \downarrow 0$. 
This will prove the existence of equation \eqref{sbm}, once it is shown that $X$ satisfies the equivalent martingale problem (MP): 
for all $\phi\in \cS(\RR)$,
\begin{align}\label{mp1}
M_t(\phi)=\langle X_t, \phi\rangle-\langle X_0,\phi \rangle-\frac{1}{2}\int_0^t\langle X_s,\phi\rangle ds
\end{align}
is a square integrable martingale with quadratic variation 
\begin{align}\label{mp2}
\langle M(\phi) \rangle_t=\int_0^t\int_{\RR}\sigma(s,x,\PP_{X_s(x)})^2\phi(x)^2X_s(dx)ds.
\end{align}

The uniqueness problem for equation \eqref{sbm} is much more involved. 
Overbeck \cite{ap-96-overbeck} appears relevant to this problem. 
However, this work requires  (e.g. \cite[Proposition 3.3]{ap-96-overbeck}) that $\sigma(t,x,\PP_{X_t(x)})$ is differentiable in time and twice differentiable in space, with all derivatives being uniformly bounded.   Suppose that $\sigma$ is differentiable in the third argument in certain sense and consider using the chain rule. Then, applying Overbeck's result, one still needs to define and verify the differentiability of $\PP_{X_t(x)}$ in $x$, which seems challenging without more artificial assumptions.

Instead, motivated by the fact that a distribution is often uniquely determined by its moments, we impose the condition that $\sigma(t,x,\PP_{X_t(x)})$ depends on the moments of $X_t(x)$. 
Firstly, we find an ``almost'' explicit moment formula for  $X_t(x)$ under mild hypotheses that ensures the existence of $X$ assuming $\sigma(s, x, \PP_{X_s(x)})$ is known. 
Using this formula, under the moment conditions, we can show the uniqueness of moments of any solution $X_t(x)$ to \eqref{sbm}. Then, the weak uniqueness of solutions to equation \eqref{sbm} can be proved by studying its (unique) log-Laplace equation. 

After we establish the existence and uniqueness of solutions to \eqref{sbm}, we also study the regularity of the moments of the solution to \eqref{sbm}. As the diffusion coefficient involves a square root that is not Lipschitz, the Picard iteration fails to get a convergent sequence in $L^2(\Omega \times [0,T]\times \RR)$. However, using the Picard iteration for the moments, one may get a convergent sequence in $C([0,T]\times \RR)$. This allows us to get the time and spatial regularity of the moments of the solution to \eqref{sbm}.   On the other hand, the regularity of the moments also implies the differentiability of $\sigma(t,x,\PP_{X_t(x)})$ in both $t$ and $x$. One may obtain the uniqueness by Overbeck's theorem.

Inspired by the regularity of moments,   our result may potentially be extended to higher dimensions in the following way. Let $X_t$ denote  the sBm in $\RR^d$ with $d\geq 2$. Then, $X_t$ does not have an almost sure Lebesgue density (c.f. Dawson and Hochberg 
\cite{ap-dawson-hochberg}), namely $X_t(dx)/dx$ is not a real-valued random field. But we can instead consider $\EE (X_t(dx))/dx$, which is the first moment of the ``density'' of $X_t$ if we formally exchange the order of differentiation and integration by Fubini theorem. 
Recall the fact that in one dimensional case, the density of the sBm's is only $1/2-\epsilon$ H\"{o}lder continuous for any $\epsilon\in (0,1/2)$ in space, but the moment is differentiable. Hence, it is reasonable to expect that in general $\EE (X_t(dx))/dx$, the Radon-Nikodym derivative of the moment of the distribution function of the sBm with respect to the Lebesgue measure, exists as a real-valued (deterministic) function on $\RR^d$, although  $ X_t(dx)/dx$ does not exist as a real-valued random field itself. Suppose now that $\sigma$ depends on the moment of the density of the sBm in terms of $\EE (X_t(dx))/dx$. Then, the corresponding mean-field martingale problem can be well formulated analogously to \eqref{mp1} and \eqref{mp2}. Because of the limitation of space, we only focus on the one-dimensional case in this paper, and leave the problem in higher dimensions for future work.

\section{Main results}
To present the main results of this paper. We shall first introduce (recall) some notation and hypotheses which will be used. 

We denote by $\RR$ the set of real numbers, by $\RR_+$ the set of nonnegative numbers, and by $\NN=\{1,2,\dots\}$ the set of natural numbers. Notation $\cS(\RR^d)$ and $\cS'(\RR^d)$ are used for the space of Schwartz functions and its dual space, respectively, on $\RR^d$ for all $d\in \NN$. Let $\cM_F(\RR)$ be the set of all finite measures on $\RR$, let $\sP(\RR_+)$ be the collection of all Borel probability measures on $\RR_+$  equipped with the weak topology, namely, $\lim_{n\to\infty}\mathbb{P}_n=\mathbb{P}$ in $\sP(\RR_+)$, denoted by $\mathbb{P}_n\Rightarrow \mathbb{P}$, if
\[
\lim_{n\to\infty}\int_{\RR_+}\phi(x)\mathbb{P}_n(dx)=\int_{\RR_+}\phi(x)\mathbb{P}(dx),
\]
for all $\phi\in \cS(\RR)$. We write $\cM(\RR;\sP(\RR_+))$ for the collections of measurable functions on $\RR$ with values in $\sP(\RR_+)$. For any $x\in \RR$, notation $\delta_{x}$ denotes the Dirac delta measure at $x$. We sometimes use $\delta$ for a small positive number, it should not be confused with the delta function $\delta_x$. Finally, we also remark that in the present paper, notation $C$, $c_1$ and $c_2$ are used for nonnegative constants that may vary from line to line.

\begin{hypothesis}\label{hyp_0}
\begin{enumerate}[(i)]
\item $\sigma^2$ is positive and bounded, that is, there exists a positive constant $K_0$ such that 
\[
0< \sigma^2(t,x,\mu)\leq K_0
\]
for all $(t,x,\mu)\in \RR_+\times \RR\times \sP(\RR_+)$.

\item $\sigma^2$ is continuous with respect to all the arguments, in the sense that for any $(t_n,x_n)\to (t,x)\in \RR_+\times \RR$ and $\mu_n\Rightarrow \mu$ in $\sP(\RR_+)$, it follows that
\[
\lim_{n\to \infty}\sigma^2 (t_n,x_n,\mu_n)=\sigma^2(t,x,\mu).
\]
\end{enumerate}
\end{hypothesis}

\begin{hypothesis}\label{hyp_fn}
For any $(t,x,\mu)\in [0,T]\times \RR\times \sP(\RR_+)$,
\[\sigma^2(t,x,\mu) = f(t,x,\EE [X_{\mu}],\EE [X_{\mu}^2],\dots, \EE [X_{\mu}^N])\,, 
\]
where $N\in\NN$, $X_{\mu}$ is a random variable with distribution $\mu$ and $f$ is a continuous function on $[0,T]\times \RR\times \RR_+^N$ that is positive and  bounded. Moreover,  $f$ is assumed to be differentiable in the last $N$ spatial arguments with bounded derivatives.
\end{hypothesis}

In the next hypothesis, we let $N$ in Hypothesis \ref{hyp_fn} to be infinity. Before stating the hypothesis, let us first introduce the following Hilbert space of real sequences. For any $\gamma\in \RR$, the Hilbert space $\cH_{\gamma}$ is a collection of real sequences, namely, $x=(x_i)_{i\in \NN}$ with $x_i\in \RR$ for all $i\in \NN$, equipped with inner product
\begin{align}\label{def_ch}
\langle x, y\rangle_{\cH_{\gamma}} =\sum_{n=1}^{\infty}(n!)^{-2\gamma}x_ny_n,
\end{align}
for all $x=(x_i)_{i\in \NN}$ and $y=(y_i)_{i\in \NN}$.
\begin{hypothesis}\label{hyp_ifn}
Let $\cH=\cH_{\gamma}$ with $\gamma>\frac{3}{2}$, and let $\cH_+=\{x=(x_i)_{i\in \NN}\in \cH:x_i\geq 0,\forall i\geq 1\}$. Then, $\sigma$ can be represented as $\sigma(t,x,\mu)^2=f(t,x,\EE [X_{\mu}],\EE [X_{\mu}^2],\dots )$ for some measurable function $f$ on $[0,T]\times \RR\times\cH_+$ that is positive and bounded. Moreover, $f$ is Lipschitz in $y\in\cH_+$ with  uniform  constant  in $(t,x)\in [0,T]\times \RR$, namely,
\[
\sup_{(t,x)\in [0,T]\times \RR}|f(t,x,y_1)-f(t,x,y_2)|\leq L\|y_1-y_2\|_{\cH},
\] 
for all $y_1,y_2\in \cH_+$ with some constant $L>0$.
\end{hypothesis}

\begin{remark}
\begin{enumerate}[(i)]
\item Hypothesis \ref{hyp_0}, which seems the most general one among Hypotheses \ref{hyp_0}, \ref{hyp_fn} and \ref{hyp_ifn}, ensures the existence of the solution to \eqref{sbm}. To obtain the uniqueness, we need to assume that  $\sigma$ is of a special form as in Hypothesis \ref{hyp_fn} or \ref{hyp_ifn}.

\item   Hypothesis \ref{hyp_0} is  inconsistent with Hypothesis \ref{hyp_fn} or \ref{hyp_ifn}. Consider the simplest example that $\sigma^2(t,x,\mu)=f(\EE [X_{\mu}])$ with $f\in C_c^{\infty}(\RR)$, the space of infinitely differentiable functions with compact support, such that $f(0)\neq f(1)$. Then, Hypothesis \ref{hyp_fn} holds for $\sigma$. For any $n=1,2,\dots$, let $\mu_n$ be the counting measure on $\{0,n\}$ with $\mu_n(0)=1-1/n$ and $\mu_n(n)=1/n$. Then, $\mu_n\Rightarrow\delta_0$ in the weak topology. However, $\lim_{n\to \infty}\sigma^2(t,x,\mu_n)=f(1)\neq f(0)=\sigma(t,x,\delta_0)$. In fact, except for some trivial cases, like $\sigma\equiv 1$, any $\sigma$ that satisfies Hypothesis \ref{hyp_fn} or \ref{hyp_ifn} does not satisfy Hypothesis \ref{hyp_0}.

\item Because of (ii),  the existence result under Hypothesis \ref{hyp_0} can not be transferred to situations satisfying Hypothesis \ref{hyp_fn} or \ref{hyp_ifn}. In Section \ref{sec_extc}, we only prove the existence under Hypothesis \ref{hyp_0}. This proof can be modified to cover cases under other hypotheses (see Remark \ref{rmk_mmt}). 
\end{enumerate}
\end{remark}

Next, we state the last hypothesis about the initial condition $X_0$.
\begin{hypothesis}\label{hyp_x0}
$X_0\in \cM_F(\RR)$ has a bounded density, still denoted by $X_0$, such that $X_0\in H_{1,2}(\RR)$, namely, $\|X_0\|_{1,2}=\|X_0\|_2+\|\nabla X_0\|_2<\infty$.
\end{hypothesis}

Now, we are ready to state the main results of the present paper.
\begin{theorem}\label{thm_exs}
Assume $X_0$ satisfying Hypothesis \ref{hyp_x0}. Then, equation \eqref{sbm} with initial condition  $X_0$ has a weak solution on any time interval $[0,T]$ under one of Hypotheses \ref{hyp_0}, \ref{hyp_fn} and \ref{hyp_ifn}.  Additionally, the solution is unique in distribution under either Hypothesis \ref{hyp_fn} or \ref{hyp_ifn}.
\end{theorem}

We organize this paper as follows: In Section \ref{sec_extc}, we prove the existence of the solution to equation \eqref{sbm}. Section \ref{sec_mmtfor} is devoted to a moment formula for any solution to \eqref{sbm} and some related estimates. This formula will be the key to the uniqueness result provided in Section \ref{sec_unq}. In the last Section \ref{sec_mmtreg}, we study the regularity of the moments of the solution.

\section{The existence}\label{sec_extc}
In this section, we prove the existence of the solution to equation \eqref{sbm} by using two-step approximation. The approximating equation \eqref{sbm1} is studied in Sections \ref{sec_appx} and \ref{sec_dual}. This will help us to provide a proof of the existence of solutions to equation \eqref{sbm} in Section \ref{sec_cov}.

\subsection{Branching particle approximation}\label{sec_appx}
Let 
$\cI=\{\alpha=(\alpha_0,\dots,\alpha_N):N\in \NN,\alpha_0\in \NN,\alpha_i\in \{1,2\},1\leq i\leq N\}\}$. 
The set $\cI$ is used to label all possible particles in the system. Thus by definition of $\cI$, we see that each particle is allowed to generate at most $2$ offspring. For any $\alpha=(\alpha_0,\dots, \alpha_N)$, we write $\alpha-1=(\alpha_0,\dots, \alpha_{N-1})$. 
Then, $\alpha-1$ is uniquely determined as the mother of particle $\alpha$
and we can define $\alpha-2$, $\alpha-3$, ... and $\alpha-N=(\alpha_0)$ iteratively. Write $\{B^{\alpha}: \alpha \in  \mathcal{I} \}$ for a family of independent one-dimensional Brownian motions.

Let $n\in \NN$ be a scaling parameter and let $\delta>0$ be a smoothing parameter. Consider a branching particle system on $\RR$ with initial distribution $X^{\delta,n}_0=\frac{1}{n}\sum_{i=1}^{K_n}\delta_{x_i}$ for some $K_n\in N$,   the number of initial particles and $x_i\in \RR$ the initial position of each particle for all $1\leq i\leq K_n$.  Denote by $\xi^{\alpha}_t$ the position of each particle, and by
\[
X_t^{\delta,n} = \frac{1}{n}\sum_{\alpha \sim_n t} \delta_{  \xi^{\alpha}_t },
\]
the empirical measure of the system where the summation over $\alpha \sim_n t$ is among all particles ``alive'' at time $t$  (to be defined later).  
We also associate a smoothing random field $Y^{\delta,n}$ on $\RR_+\times \RR$ given by
\begin{align*}
Y_t^{\delta,n}(x) =\langle X_t^{\delta,n}, p_{\delta}(x-\cdot)\rangle= \int_{\RR} p_{\delta}(x-y)X_t^{\delta,n}(dy).
\end{align*}
The lifetime of each particle $\alpha$ is controlled by an independent 
exponential clock. 
The parameter of each clock is $n\widetilde{\sigma}_{ \delta}^2(t,\xi^{\alpha}_t, \PP_{Y^{\delta, n}_t})$, where $\widetilde{\sigma}_{ \delta}$ is defined as in \eqref{def_tsgm} with some measurable function $\sigma:\RR_+\times \RR\times \sP(\RR_+)\to \RR_+$. This means for any living particle $\alpha$ at time $t\geq 0$ with position $\xi^{\alpha}_t$, the probability that she dies in the time interval $[t,t+\Delta t)$ is 
\[
n\widetilde{\sigma}_{  \delta}^2\big(t,\xi^{\alpha}_t, \PP_{Y^{\delta, n}_t}\big)\Delta t+o(\Delta t).
\]
Each individual leaves behind $0$ or $2$ offspring with equal probability when she dies. $N(\alpha)$ is used for the number of offspring of particle $\alpha$.   Denote by $\beta^{\delta,n}(\alpha)$ and $\zeta^{\delta,n}(\alpha)$ the birth and death time of particle $\alpha$. If particle $\alpha-1$, the parent of particle $\alpha$, does not branch any offspring, namely $N(\alpha-1)=0$, then we write $\beta^{\delta,n}(\alpha)=\zeta^{\delta,n}(\alpha)=\infty$. Moreover, we say $\alpha$ is alive at time $t$, if $\beta^{\delta,n}(\alpha)\leq t<\zeta^{\delta,n}(\alpha)$.

The initial position of each particle inherits her mother's death position, and its motion can be described by $B^{\alpha}$ before she dies. To be more precise, for any $\beta^{\delta,n}(\alpha)\leq t<\zeta^{\delta,n}(\alpha)$, 
\[
\xi^{\alpha}_t=\xi^{\alpha-1}_{\beta^{\delta,n}(\alpha)^-}+B^{\alpha}_t-B^{\alpha}_{\beta^{\delta,n}(\alpha)}.
\]

Let $\widetilde{\sigma}_{ \delta}$ be given as in \eqref{def_tsgm}. It is clear that under  one of Hypotheses \ref{hyp_0}, \ref{hyp_fn} and \ref{hyp_ifn}, $\widetilde{\sigma}_{\delta}$ is also positive and bounded.
By using the classical tightness arguments (c.f. Dawson et al. \cite[Lemmas 2.3 and 2.4]{aihpps-00-dawson-vaillancourt-wang}), one can easily show the following lemma.
\begin{lemma}\label{lmm_tight}
Fix $\delta>0$. 
Assume  Hypothesis \ref{hyp_0}  and $X_0^{\delta,n} \Rightarrow X_0\in \cM_F(\RR)$   as $n\to\infty$. Then,
  \begin{enumerate}[(i)]
    \item $\{X^{\delta,n}; t\ge 0\}_{ n\in \NN}$ is a tight family of processes with sample paths in $D(\RR_+, \cM_{F}(\RR))$ with limit in $C(\RR_+, \cM_{F}(\RR))$.
    \item Let $X^{\delta}$ denote a limit point of $X^{\delta,n}$ and 
    \[
Y^{\delta}_t(x)=\langle X_t^{\delta}, p_{\delta}(x-\cdot)\rangle.
  \]
  Then, $(X^{\delta},Y^{\delta})$ is a solution of the following MP: for all $\phi \in \cS(\RR)$,
  \begin{align}\label{mart1}
  M^{\delta}_t(\phi):= \langle \phi, X^{\delta}_t \rangle - \langle \phi, X^{\delta}_0 \rangle - \frac{1}{2}\int_0^t \langle  \Delta \phi, X^{\delta}_s \rangle ds
  \end{align}
  is a continuous square integrable martingale such that $M_0(\phi)=0$ and 
  \begin{align}\label{mart2}
  \langle M^{\delta}(\phi) \rangle_t = \int_0^tds\int_{\RR}\phi^2(x) \widetilde{\sigma}_{ \delta}^2(s,x, \PP_{Y^{\delta}_s(x)})X^{\delta}_s(dx).
  \end{align}
  \end{enumerate} 
\end{lemma}
\begin{proof}
Inasmuch as the proof of property (i) is quite standard, we omit it for the sake of brevity. 
It also follows by standard arguments that $M^{\delta}_t(\phi)$, given by \eqref{mart1}, is a continuous square integrable martingale. 
It remains to prove its quadratic variation satisfies equation \eqref{mart2}.
Let $\{X^{\delta,n_k}\}_{k\in \NN}$ be a subsequence of $\{X^{\delta,n}\}_{n\in \NN}$ with limit $X^{\delta}$ in $D(\RR_+, \cM_{F}(\RR))$. By Skorohod representation theorem, we assume this convergence is almost surely.  Then, by classical tightness arguments (c.f. Wang \cite[Corollary 7.3]{saa-98-wang}), we know that
\begin{align*}
\langle M^{\delta}(\phi) \rangle_t=&\lim_{k\to\infty}\langle M^{\delta,n_k}(\phi) \rangle_t,
\end{align*}
where $M^{\delta,n_k}$ is a $\cS'(\RR)$-valued martingale given by
\[
M^{\delta,n_k}_t(\phi)=\frac{1}{n_k}\sum_{\zeta^{\delta,n_k}(\alpha)<t,\alpha\sim_{n_k}\beta^{\delta,n_k}(\alpha)}\phi(\xi^{\alpha}_{\zeta^{\delta,n_k}(\alpha)^-})(N(\alpha)-1),
\]
with $N(\alpha)$ defined as before, the offspring number of particle $\alpha$ and $
\alpha\sim_n t $ meaning particle $\alpha$ is alive at time $t$ in the $n$-th approximation. The quadratic variation of $M^{\delta,n}$ can be written as
\begin{align*}
\langle M^{\delta,n_k}(\phi)\rangle_t=\int_0^t \lim_{\Delta s\downarrow 0}\frac{\EE \big[M^{\delta,n_k}_{s+\Delta s}(\phi)^2-M^{\delta,n_k}_s(\phi)^2\big|\cF_{s}^{\delta, n_k}\big]}{\Delta s}ds.
\end{align*}
Following the idea of Dawson et al. 
\cite[Lemma 2.3]{aihpps-00-dawson-vaillancourt-wang}, one can show that
\begin{align*}
\langle M^{\delta,n_k}(\phi)\rangle_t=\int_0^t \phi^2(x)\widetilde{\sigma}_{  \delta}\big(t,x, \PP_{Y^{\delta, n_k}_t(x)}\big)^2X_s^{\delta,n_k}(dx)ds.
\end{align*}
Letting $k\to\infty$, we find that $X_t^{\delta,n_k}\to X_t^{\delta}$ a.s. for all $t\in \RR_+$. Thus for fixed $\delta>0$, one has that
\[
Y^{\delta,n_k}_t(x)=\langle X_t^{\delta,n_k},p_{\delta}(x-\cdot)\rangle\to \langle X_t^{\delta},p_{\delta}(x-\cdot)\rangle=Y^{\delta}_t(x)
\]
for any $(t,x)\in \RR_+\times \RR$ as $k\to\infty$. Therefore, equation \eqref{mart2} is a consequence of the continuity and boundedness of $\sigma$ and the dominated convergence theorem. The proof of this lemma is complete.
\end{proof}

\begin{remark}\label{rmk_mmt}
Note that in the proof of Lemma \ref{lmm_tight}, the continuity condition of $\sigma$ is used when evaluating the limit of quadratic variation of $M^{\delta, n_k}_t(\phi)$. Instead, assume Hypothesis \ref{hyp_fn} or \ref{hyp_ifn}. Due to the tightness argument again, it holds that for every $m\in \NN$, $\EE[(Y_t^{\delta,n_k}(x))^m]$ is bounded uniformly in $k\in \NN$, $(t,x)\in [0,T]\times \RR$ for any $T>0$. Concerning the fact that $Y_t^{\delta,n_k}(x)\to Y^{\delta}(x)$ almost surely, and thus in probability, the convergence is also in $L^m(\Omega)$ for all $m\in \NN$. As a result, $\lim_{k\to\infty}\EE[(Y_t^{\delta,n_k}(x))^m]=\EE[(Y_t^{\delta}(x))^m]$ for all $m\in \NN$ and $(t,x)\in \RR_+\times \RR$. This proves equation \eqref{mart2} under either Hypothesis \ref{hyp_fn} or \ref{hyp_ifn}, and hence the result as in Lemma \ref{lmm_tight} holds as well.
\end{remark}

\subsection{Moment duality and existence of solution to \texorpdfstring{\eqref{sbm1}}{}}\label{sec_dual}
Let $(X^{\delta},Y^{\delta})$ be a solution to MP 
\eqref{mart1} and \eqref{mart2} 
with initial condition $X_0\in \cM_F(\RR)$. In this section, we prove that $X^{\delta}_t$ has a Lebesgue density almost surely. To this end, we need to provide moment formulas for $\langle X_t^{\delta}, \phi\rangle$ with some function $\phi\in \cS(\RR)$. The moment formula can be derived by the method of moment duality (c.f. Dawson and Kurtz \cite{springer-82-dawson-kurtz}).

For any $t\in [0,T]$, $\mu\in \cM_F(\RR)$, $n\in \NN$, $\phi\in C_b^2(\RR^n)$, we define functions $F$ and $G^{\delta}$ as follows,
\begin{align}\label{def_F}
F(\mu,(n,\phi))=\langle \mu^{\otimes n}, \phi\rangle=\int_{\RR^n}\mu^{\otimes n}(d\by_n)\phi(\by_n),
\end{align} 
and
\begin{align}\label{def_G}
G^{\delta}(t,\mu,(n,\phi))=\frac{1}{2}\langle \mu^{\otimes n}, \Delta_n \phi\rangle+\sum_{1\leq i<j\leq n}\langle \mu^{\otimes (n-1)}, \Phi_{ij}^{\delta}(t) \phi\rangle,
\end{align}
where $\by_n$ is short for $(y_1,\dots, y_n)\in \RR^n$, $\Delta_n $ denotes the $n$-dimensional Laplacian operator in space and $\Phi_{ij}^{\delta}(t)\phi$ is a function of $n-1$ variables with the $i$-th and $j$-th variables of $\phi$ coalesced, and  then multiplying by $\widetilde{\sigma}_{ \delta}^2( t,x_i, \PP_{Y_{t}^{\delta}(x_i)})$, namely,
\begin{align}\label{def_phiij}
\Phi_{ij}^{\delta}(t) \phi(x_1,\dots, x_{n-1})=&\widetilde{\sigma}_{ \delta}^2( t,x_i, \PP_{Y_{t}^{\delta}})(\varphi_{ij}\phi)(x_1,\dots, x_{n-1}),
\end{align}
with the coalescing operator $\varphi_{ij}$ given by
\begin{align*}
(\varphi_{ij}\phi)(x_1,\dots, x_{n-1})=\phi(x_1,\dots, x_{j-1},x_i,x_j,\dots, x_{n-1}).
\end{align*}
Then, we have the next lemma, whose proof is just an application of It\^{o}'s formula, we skip it for the sake of conciseness and refer readers to Xiong \cite[Lemma 1.3.2]{ws-13-xiong} for a similar result.
\begin{lemma}
  Let $(X^{\delta},Y^{\delta})$ be a solution to MP \eqref{mart1} and \eqref{mart2} with initial condition $X_0\in \cM_F(\RR)$. Then, for any $n\in \NN$ and $\phi\in \cS(\RR^n)$, the process
  \begin{align}\label{dual1}
  F(X^{\delta}_t, (n,\phi)) - \int_0^t G^{\delta}(s, X^{\delta}_s, (n,\phi))ds
  \end{align}
  is a martingale.
\end{lemma}

In the next step, we define the moment dual of $X^{\delta}$. 
Given $n\in \NN$, let $\{\tau_k\}_{1\leq k\leq n-1}$ be independent exponential random variables. For each $k$, the parameter of $\tau_k$ is $\frac{1}{2}(n-k+1)(n-k)$.  Let $\eta_{k}=\sum_{i=1}^k\tau_i$ for all $k=1,\dots, n-1$, and by convention $\eta_0=0$ and $\eta_n=\infty$. Then, we define  an $\NN$-valued decreasing Markov process starting at $n$, by
\[
n_{\eta_k^-}=n-k+1, \quad n_{\eta_k}=n-k,
\]
for $k=1,\dots, n-1$ and $n_t=1$ for all $t\geq \eta_{n-1}$. Then, we can  also write
\begin{align}\label{def_nt}
n_t=\sum_{k=1}^n(n-k+1)\1_{\eta_{k-1}\leq t<\eta_{k}}.
\end{align}
Let $\{S_k^{\delta}(t): 1\leq k\leq n-1,t>0\}$ be a collection of independent random variables defined as follows. For any $k=1,\dots, n-1$ and $t>0$, $S_k^{\delta}(t)$ is uniformly  distributed on $\{\Phi_{ij}^{\delta}(t):1\leq i<j\leq k\}$ where $\Phi_{ij}^{\delta}(t)$ are defined as in \eqref{def_phiij}. We also write $T^{\otimes k}$ for the the semigroup generated by $\frac{1}{2}\Delta_k$ on $\RR^k$ for all $k=1,\dots,n$,  namely,
\[
T^{\otimes k}_t \varphi(\bx_k)=\int_{\RR^k}d\by_k\prod_{i=1}^kp_{t}(x_i-y_i)\varphi(\by_k)
\]
for all $\varphi\in C_b^2(\RR^n)$, the space of bounded functions on $\RR$ with bounded first and second derivatives. Let $\phi\in \cS(\RR^{n})$, we define a stochastic process $f^{\delta}_t$ starting  at $f_0=   \phi $  by
\begin{align}\label{def_ft}
f^{\delta}_t=T_{t-\eta_{k}}^{\otimes n-k}S_{n-k+1}^{\delta}(\eta_{k})T_{\tau_{k}}^{\otimes n-k+1}\cdots S^{\delta}_{n-1}(\eta_{2})T^{\otimes (n-1)}_{\tau_{2}} S_{n}^{\delta}(\eta_{1})T^{\otimes n}_{\tau_{1}}\phi,
\end{align}
provided $\eta_k\leq t<\eta_{k+1}$ with $k=0,\dots,n-1$.
\begin{lemma}\label{lmm_mtgdual}
  Let $F$, $G^{\delta}$, $n_t$ and $f_t$ be given as in \eqref{def_F}, \eqref{def_G}, \eqref{def_nt} and \eqref{def_ft} respectively. Define a function $H^{\delta}$ by
  \[
  H^{\delta}(t,\mu, (k,\phi)) =G^{\delta}(t,\mu,(k,\phi)) - \frac{1}{2}k(k-1)F(\mu,(k,\phi))
  \]
  for all $t\in \RR_+$, $\mu\in \cM_F(\RR)$, $n\in \NN$ and $\phi\in C_b^2(\RR)$. Then, for any $\mu \in \cM_F(\RR)$, the process
  \begin{align}\label{def_mdual}
  F(\mu, (n_t,f^{\delta}_t)) - \int_0^t H^{\delta}(s,\mu, (n_s,f^{\delta}_s))ds
  \end{align}
  is a martingale.
\end{lemma}
Before providing the proof, we remark here that the smoothness of $\sigma$ (see \eqref{def_tsgm}) ensures that as a function of $x$, $\widetilde{\sigma}_{ \delta}(t,x, \PP_{Y^{\delta}_t(x)})$ is in $C_b^2(\RR)$. Thus $f^{\delta}_t\in C_b^2(\RR^{n_t})$ for all $t\in \RR_+$. It is necessary to make sense of $H^{\delta}(t,\mu, (n_t,f^{\delta}_t))$ that involves a Laplacian operator acting on $f^{\delta}_t$.
\begin{proof}[Proof of Lemma \ref{lmm_mtgdual}]
The proof of this lemma can be done as an application of Either and Kurtz's result  \cite[Proposition 4.1.7]{wiley-86-ethier-kurtz} by showing that
  \begin{align}\label{for_cdg}
  \lim_{\Delta t\downarrow 0}\frac{1 }{\Delta t}\EE\big[ \big(F(\mu, (n_{t+\Delta t},f^{\delta}_{t+\Delta t}))-F(\mu, (n_t,f^{\delta}_t))\big)\big|(n_t,f^{\delta}_t)\big]=H^{\delta}(t,\mu, (n_t,f^{\delta}_t))
  \end{align}
  for all $t\in \RR_+$. Without loss of generality, we only prove  equality  \eqref{for_cdg} on the set $\{n_t = k\}$ for  some $k=1,\dots, n$. Recall that
  \[
  \{n_t = k\}=\{\eta_{n-k}\leq t<\eta_{n-k+1} \},
  \] 
  $\tau_{n-k+1}=\eta_{n-k+1}-\eta_{n-k}$ is an exponential random variable with parameter $\frac{1}{2}k(k-1)$ if $k>1$ and $\tau_{n}=\infty$. We see 
  \begin{align}\label{hmart1}
  &\EE \big[\big(F(\mu, (n_{t+\Delta t},f^{\delta}_{t+\Delta t}))-F(\mu, (n_t,f^{\delta}_t))\big)\big|(n_t,f^{\delta}_t){ \1_{n_t=k}}\big]=I_1+I_2+o(\Delta t),
  \end{align}
  where 
\[
I_1=\EE \big[\big(F(\mu, (k,f^{\delta}_{t+\Delta t}))-F(\mu, (k,f^{\delta}_t))\big)\1_{\{n_{t+\Delta t}=k\}}\big|(n_t,f^{\delta}_t){ \1_{n_t=k}}\big]
\]
and
\[
I_2=\EE \big[\big(F(\mu, (k-1,f^{\delta}_{t+\Delta t}))-F(\mu, (k,f^{\delta}_t))\big)\1_{\{n_{t+\Delta t}=k-1\}}\big|(n_t,f^{\delta}_t){ \1_{n_t=k}}\big].
\]
Suppose that $k>1$. Then, by the memoryless property of exponential random variables, we have
\begin{align*}
I_1=&\big(\langle \mu^{\otimes k}, T_{\Delta t}^{\otimes k}f_t^{\delta}\rangle \PP(\tau_{n-k+1}>\Delta t)-\langle \mu^{\otimes k}, f_t^{\delta}\rangle\big)\1_{\{n_t=k\}}.
\end{align*}
This implies that
\begin{align}
\lim_{\Delta t\downarrow 0}\frac{1}{\Delta t}I_1=\frac{1}{2}\big(\langle \mu^{\otimes k}, \Delta_kf^{\delta}_t\rangle {  -}k(k-1)\langle \mu^{\otimes k}, f^{\delta}_t\rangle\big)\1_{\{n_t=k\}}.
\end{align}
On the other hand, since $\Phi^{\delta}_{ij}$ is uniformly distributed, it follows that
\begin{align*}
I_2=\sum_{1\leq i<j\leq k}\frac{2}{k(k-1)}\EE\big[&\langle \mu^{\otimes (k-1)}, T^{\otimes (k-1)}_{\Delta t-\tau_{n-k+1}}\Phi^{\delta}_{ij}(t+\tau_{n-k+1})T^{\otimes k}_{\tau_{n-k+1}}f^{\delta}_{t}\rangle\\
&\1_{\{\tau_{n-k+1}\leq \Delta t <\tau_{n-k+1}+\tau_{n-k}\}}\big|(n_t,f^{\delta}_t)\big]\1_{\{n_t=k\}}-\langle \mu^{\otimes k},f^{\delta}_t\rangle\1_{\{n_t=k\}}.
\end{align*}
Notice that if $k>2$, we have
\begin{align*}
&\EE\big[\langle \mu^{\otimes (k-1)}, T^{\otimes (k-1)}_{\Delta t-\tau_{n-k+1}}\Phi^{\delta}_{ij}(t+\tau_{n-k+1})T^{\otimes k}_{\tau_{n-k+1}}f^{\delta}_{t}\rangle\1_{\{\tau_{n-k+1}\leq \Delta t <\tau_{n-k+1}+\tau_{n-k}\}}\big|(n_t,f^{\delta}_t){ \1_{n_t=k}}\big]\\
=&\int_0^{\Delta t}ds\int_{\Delta t-s}^{\infty}dr\langle \mu^{\otimes (k-1)}, T^{\otimes (k-1)}_{\Delta t-s}\Phi^{\delta}_{ij}(t+s)T^{\otimes k}_{s}f^{\delta}_{t}\rangle\\
&\times \frac{1}{2}k(k-1)e^{-\frac{1}{2}k(k-1)s}\times \frac{1}{2}(k-1)(k-2)e^{-\frac{1}{2}(k-1)(k-2)r}\\
=&\frac{1}{2}k(k-1)\int_0^{\Delta t}ds\langle \mu^{\otimes (k-1)}, T^{\otimes (k-1)}_{\Delta t-s}\Phi^{\delta}_{ij}(t+s)T^{\otimes k}_{s}f^{\delta}_{t}\rangle e^{-\frac{1}{2}k(k-1)s}e^{-\frac{1}{2}(k-1)(k-2)(\Delta t-s)},
\end{align*}
and for $k=2$, the following equality holds
\begin{align*}
&\EE\big[\langle \mu^{\otimes (k-1)}, T^{\otimes (k-1)}_{\Delta t-\tau_{n-k+1}}\Phi^{\delta}_{ij}(t+\tau_{n-k+1})T^{\otimes k}_{\tau_{n-k+1}}f^{\delta}_{t}\rangle\1_{\{\tau_{n-k+1}\leq \Delta t <\tau_{n-k+1}+\tau_{n-k}\}}\big|(n_t,f^{\delta}_t){ \1_{n_t=k}}\big]\\
=&\frac{1}{2}k(k-1)\int_0^{\Delta t}ds\langle \mu^{\otimes (k-1)}, T^{\otimes (k-1)}_{\Delta t-s}\Phi^{\delta}_{ij}(t+s)T^{\otimes k}_{s}f^{\delta}_{t}\rangle e^{-\frac{1}{2}k(k-1)s}.
\end{align*}
Therefore, 
\begin{align}\label{hmart3}
\lim_{\Delta t\downarrow 0}\frac{1}{\Delta t}I_2=\sum_{1\leq i<j\leq k}\langle \mu^{\otimes (k-1)}, \Phi^{\delta}_{ij}(t)f^{\delta}_t\rangle \1_{\{n_t=k\}}.
\end{align}
Equality \eqref{for_cdg} follows from \eqref{hmart1}-\eqref{hmart3}. Hence, by Proposition 4.1.7 of Ethier and Kurtz \cite{wiley-86-ethier-kurtz}, the process given by \eqref{def_mdual} is a martingale. The proof of this lemma is complete.
\end{proof}

\begin{proposition}\label{mmtdual}
Let $(X^{\delta},Y^{\delta})$ be a solution to MP \eqref{mart1} and \eqref{mart2} with initial condition $X_0\in \cM_F(\RR)$, and let $(n_t,f_t )$ be the pair of processes defined as in \eqref{def_nt} and \eqref{def_ft} respectively with $n_0=n\in \{1,2,\dots\}$ and $f_0=\phi \in \cS(\RR^n)$. Then, we have 
\[
\EE \big[ F(X_t^{\delta}, (n,\phi)) \big]= \EE \bigg[F(X_0, (n_t,f_t^{\delta})) \exp\Big( \int_0^t\frac{n_s(n_s-1)}{2}ds\Big) \bigg].
\]
\end{proposition}
\begin{proof}
The proof follows from the development of Dawson and Kurtz \cite[Corollary 3.3]{springer-82-dawson-kurtz}, 
with minor adjustments due to the fact that $X^{\delta}_t$ and 
$(n_t,f^{\delta}_t)$ are time-inhomogeneous vis-\'{a}-vis time-homogeneous 
Markov processes. 
We omit the details for the sake of conciseness.
\end{proof}

\begin{proposition}
Let $(X^{\delta}, Y^{\delta})=(X^{\delta},\langle X^{\delta}, p_{\delta}(x-\cdot)\rangle)$ be a solution to MP \eqref{mp1} and \eqref{mp2} with initial condition $X_0\in \cM_F(\RR)$. Then, for every $t\in \RR_+$, $X^{\delta}_t$ has a Lebesgue density. Moreover, identifying $X_t^{\delta}(x)$ as the density of $X_t^{\delta}$, the pair  $(X^{\delta},Y^{\delta})$ satisfies equation \eqref{sbm1} for some space-time white noise $W$.
\end{proposition}

  \begin{proof}
Following the idea of Wang \cite[Theorem 2.1]{ptrf-97-wang}, as an application of Proposition \ref{mmtdual}, we can show that $X^{\delta}$ is absolutely continuous with respect to the Lebesgue measure. Then, by proceeding along a similar argument as in Dawson et al. \cite[Theorem 1.2]{aihpps-00-dawson-vaillancourt-wang}, it can be proved that the density, still denoted by $X^{\delta}$, together with $Y^{\delta}$ satisfies equation \eqref{sbm1}. The proof is quite standard, we omit it in the present paper.
\end{proof}

\subsection{Convergence of \texorpdfstring{$X^{\delta}$}{} and \texorpdfstring{$Y^{\delta}$}{}}\label{sec_cov}
Fix a time horizon $[0,T]$. In this section, we will show the convergence of $\{X^{\delta}\}_{\delta>0}$ as $\delta \downarrow 0$. We need the following tightness criteria for probability measures on $C([0,T]\times \RR;\RR)$.
\begin{lemma}\label{lmm_tt}
A family $\mathcal{P}$ of probability measures on $C([0,T]\times \RR;\RR)$ is precompact if
\begin{enumerate}[(i)]
\item \label{lmtt1} $\displaystyle \lim_{A\uparrow \infty} \sup_{\PP\in\mathcal{P}} \PP(|X_t(0)|>A)=0$.

\item \label{lmtt2} For each $x\in \RR$ and $\rho>0$,
\[
\lim_{\epsilon \downarrow 0}\sup_{\PP\in\mathcal{P}} \PP\Big(\sup_{0\leq s\leq t\leq T, |t-s|<\epsilon}|X_t(x)-X_s(x)|>\rho\Big)=0.
\]

\item \label{lmtt3} For every $R>0$  and $\rho>0$,
\[
\lim_{\epsilon \downarrow 0}\sup_{\PP\in\mathcal{P}} \PP\Big(\sup_{0\leq t\leq T, -R\leq x\leq y\leq R, |x-y|<\epsilon}|X_t(x)-X_t(y)|>\rho\Big)=0.
\]
\end{enumerate}
\end{lemma}
\begin{proof}
Results for probability measures on space of one parameter processes are well-known (c.f. Stroock and Varadhan \cite[Theorem 1.3.2]{springer-06-stroock-varadhan}). Multi-parameter cases are quite similar, we omit the proof for the sake of conciseness. We also refer readers to Hu et al. 
\cite[Appendices A.2 and A.3]{ap-17-hu-huang-le-nualart-tindel} for similar criteria.
\end{proof}

The next lemmas show the uniform boundedness of the moments of $X_t^{\delta}(x)$ and its increments in time and in space. 
\begin{lemma}\label{lmm_ubn}
Let $(X^{\delta},Y^{\delta})$ be a weak solution to \eqref{sbm1} with initial condition $X_0\in \cM_F$.  Suppose that $X_0$ has a bounded density, still denoted by $X_0$. 
Then, for all $n\in \NN$,
\begin{align}\label{unby}
\sup_{\delta\in (0,1)}\sup_{(t,x)\in [0,T]\times \RR}\big(\EE [Y_t^{\delta}(x)^{n}]+\EE [X_t^{\delta}(x)^{n}]\big)< \infty.
\end{align}
\end{lemma}

\begin{proof}
We prove this lemma following the idea in Xiong \cite[Lemma 1.4.5]{ws-13-xiong}. By using the moment duality (see Proposition \ref{mmtdual}), we can write
\begin{align}\label{unby1}
\EE [Y_t^{\delta}(x)^{n}]=\EE \big[\big\langle  (X_t^{\delta})^{\otimes n}, p_{\delta}(x-\cdot)^{\otimes n}\big\rangle\big]=\EE \Big[\big\langle  X_0^{\otimes n_t}, f^{\delta}_t\big\rangle \exp \Big(\frac{1}{2}\int_0^t n_s(n_s-1)ds\Big)\Big],
\end{align}
where $(f^{\delta}_t,n_t)$ is the moment dual process of $X^{\delta}_t$ as in Section \ref{sec_dual} with $f_0^{\delta}=p_{\delta}(x-\cdot)^{\otimes n}$ and $n_0=n$. Notice that the exponential term in \eqref{unby1} is bounded by $e^{\frac{1}{2}n(n-1)T}$. It suffices to estimate the following quantity
\begin{align*}
\EE \big[\big\langle  X_0^{\otimes n_t}, f^{\delta}_t\big\rangle\big] =&\sum_{k=0}^{n-1}\EE \Big[\big\langle  X_0^{\otimes (n-k)}, f^{\delta}_t\big\rangle \1_{\{n_t=n-k\}}\Big]=\sum_{k=0}^{n-1}\EE \Big[\big\langle X_0^{\otimes (n-k)}, f^{\delta}_t\big\rangle \1_{\{\eta_{k}\leq t<\eta_{k+1}\}}\Big].
\end{align*}
If $k=0$, by using the semigroup property of the heat kernel and the nonnegativity and symmetry in space  variables of $f^{\delta}$, we can easily show that
\begin{align*}
\EE \Big[\big\langle X_0^{\otimes n}, &f^{\delta}_t\big\rangle \1_{\{\eta_{0}\leq t<\eta_{1}\}}\Big] =
\int_{\RR^n}  X_0^{\otimes n}(d \by_n) T_{t}^{\otimes n}f_0(\by_n)  \EE [ \1_{\{\tau_{1}> t\}}]\\
=&\int_{\RR^n}  X_0^{\otimes n}(d \by_n) \int_{\RR^n} d\bz_n p_t^{(n)}(\by_n-\bz_n) p_{\delta}^{(n)}(x^{\otimes n}-\bz_n)  e^{-\frac{n(n-1)}{2}t}\\
=&\langle X_0, p_{t+\delta}(x-\cdot)\rangle^n e^{-\frac{n(n-1)}{2}t}\leq \|X_0\|_{\infty}^n,
\end{align*}
where $x^{\otimes n}$ denotes a point on $\RR^n$ with each coordinate being $x$ and $p_t^{(n)}(\bx_n-\by_n):=\prod_{i=1}^np_t(x_i-y_i)$. Similarly, if $k=1$, we deduce that
\begin{align*}
&\EE \Big[\big\langle X_0^{\otimes (n-1)}, f^{\delta}_t\big\rangle \1_{\{\eta_1\leq t<\eta_2\}}\Big]\\
=&\EE \Big[\int_{\RR^{n-1}}  X_0^{\otimes (n-1)}(d \by_{n-1}) \big(T_{t-\eta_1}^{n-1}S_{n}^{\delta}(\eta_{1})f^{\delta}_{\eta_{1}^-}\big)(\by_{n-1})  \1_{\{\eta_{1}\leq t<\eta_{2}\}}\Big]\\
\leq &\|\sigma\|_{\infty}^2\EE \Big[\int_{\RR^{n-1}}  X_0^{\otimes (n-1)}(d \by_{n-1})\int_{\RR^{n-1}}d\bz_{n-1} p_{t-\eta_{1}}^{(n-1)}(\by_{n-1} - \bz_{n-1})\\
&\times \sum_{1\leq i<j\leq n}\frac{2(\varphi_{ij}f^{\delta}_{\eta_{1}^-})(\bz_{n-1})}{n(n-1)} \1_{\{\eta_{1}\leq t<\eta_{2}\}}\Big].
\end{align*}
Firstly, by the symmetry of $p^{(n)}_{t}(x^{\otimes n}-\cdot)$, we can write
\begin{align*}
&\sum_{1\leq i<j\leq n}\EE \Big[\int_{\RR^{2n-2}}  X_0^{\otimes (n-1)}(d \by_{n-1})d\bz_{n-1} p_{t-\eta_{1}}^{(n-1)}(\by_{n-1} - \bz_{n-1}) \frac{2(\varphi_{ij}f^{\delta}_{\eta_{1}^-})(\bz_{n-1})}{n(n-1)} \\
&\hspace{17mm}\times \1_{\{\eta_{1}\leq t<\eta_{2}\}}\Big]\\
= &\EE \Big[\int_{\RR^{2n-2}}  X_0^{\otimes (n-1)}(d \by_{n-1})d\bz_{n-1} p_{t-\eta_{1}}^{(n-1)}(\by_{n-1} - \bz_{n-1}) \big(\varphi_{12}p_{\eta_{1}+\delta}^{(n)}(x^{\otimes n}-\cdot)\big)(\bz_{n-1})\\
&\quad\times \1_{\{\eta_{1}\leq t<\eta_{2}\}}\Big].
\end{align*}
We estimate the above integral using the semigroup property of the heat kernel and the fact that $p_t(x)\leq (2\pi t)^{-\frac{1}{2}}$, and get
\begin{align*}
&\EE \Big[\int_{\RR^{2n-2}}  X_0^{\otimes (n-1)}(d \by_{n-1})d\bz_{n-1} p_{t-\eta_{1}}^{(n-1)}(\by_{n-1} - \bz_{n-1}) \big(\varphi_{12}p_{\eta_{1}+\delta}^{(n)}(x^{\otimes n}-\cdot)\big)(\bz_{n-1})\\
&\quad \times \1_{\{\eta_{1}\leq t<\eta_{2}\}}\Big]\\
=&\EE \Big[\big\langle X_0^{\otimes (n-2)}, p_{t+\delta}^{(n-2)}(x^{\otimes n-2}-\cdot) \big\rangle \int_{\RR^2}X_0(dy)dz p_{t-\eta_1}(y-z) p_{\eta_1+\delta}(x-z)^2\1_{\{\eta_{1}\leq t<\eta_{2}\}}\Big]\\
\leq & \langle X_0, p_{t+\delta}(x-\cdot) \rangle^{n-1}\EE \big[(2\pi (\eta_1+\delta))^{-\frac{1}{2}}\1_{\{\eta_{1}\leq t<\eta_{2}\}}\big].
\end{align*}
Recall that by definition $\eta_1=\tau_1$. As a consequence, we have
\[
\EE \Big[\big\langle X_0^{\otimes (n-1)}, f^{\delta}_t\big\rangle \1_{\{\eta_1\leq t<\eta_2\}}\Big] \leq c_1\|X_0\|_{\infty}^{n-1}\|\sigma\|_{\infty}^2\EE  [\tau_1^{-\frac{1}{2}}].
\]
For general $k$,  similar to the case $k=1$ one can find   by iteration that
\begin{align*}
\EE \Big[\big\langle X_0^{\otimes k}, f^{\delta}_t\big\rangle \1_{\{\eta_{k}\leq t<\eta_{k+1}\}}\Big]
\leq c_1 c_2^n \|\sigma\|_{\infty}^{2(n-k)}\EE \bigg[\prod_{i=1}^{n-k}\tau_{i}^{-\frac{1}{2}}\bigg],
\end{align*}
for some universal constants $c_1,c_2>0$. Notice that $\{\tau_k:k=1,\dots,n-1\}$ are independent exponential random variables with parameter $\frac{1}{2}(n-k+1)(n-k)$ respectively. Furthermore, for any $\lambda>0$, we can show that
\begin{align*}
\lambda\int_0^{\infty}s^{-\frac{1}{2}} e^{-\lambda s}ds\leq \lambda\Big(\int_0^{\epsilon}s^{-\frac{1}{2}} ds+\epsilon^{-\frac{1}{2}}\int_{\epsilon}^{\infty} e^{-\lambda s}ds\Big)=2\lambda \epsilon^{\frac{1}{2}}+\epsilon^{-\frac{1}{2}}e^{-\lambda\epsilon}
\end{align*}
for all $\epsilon>0$. Choosing $\epsilon=\lambda^{-1}$, we get $\lambda\int_0^{\infty}s^{-\frac{1}{2}} e^{-\lambda s}ds\leq 3\lambda^{\frac{1}{2}}$. It follows that
\begin{align*}
\EE \Big[\big\langle X_0^{\otimes k}, f^{\delta}_t\big\rangle \1_{\{\eta_{k}\leq t<\eta_{k+1}\}}\Big]
\leq c_1 c_2^n \|\sigma\|_{\infty}^{2(n-k)} \prod_{i=k}^{n-1}(i(i+1))^{\frac{1}{2}}\leq c_1c_2^n n!,
\end{align*}
which implies that
\begin{align}\label{unby2}
\EE \big[\big\langle  X_0^{\otimes n_t}, f^{\delta}_t\big\rangle \big] \leq c_1c_2^n(n+1)!.
\end{align}
Combining \eqref{unby1} and \eqref{unby2}, and observing that $n_t$ is a decreasing process, we have
\begin{align}\label{unby3}
\sup_{\delta\in (0,1)}\sup_{(t,x)\in [0,T]\times \RR}\EE [Y_t^{\delta}(x)^{n}]\leq c_1c_2^n(n+1)!e^{\frac{1}{2}n(n-1)}<\infty.
\end{align}
In the next step, we replace $p_{\delta}$ by $p_{\epsilon}$ in \eqref{unby1} and construct the moment dual process $f^{\delta,\epsilon}$ with $f^{\delta,\epsilon}_0=p_{\epsilon}$. By the same argument, we get
\begin{align*}
\sup_{\delta\in (0,1)}\sup_{(t,x)\in \RR_+\times \RR}\sup_{\epsilon>0}\EE \big[\big\langle (X_t^{\delta})^{\otimes n}, p_{\epsilon}(x-\cdot)\big\rangle\big] < \infty.
\end{align*} 
Thus, by Fatou's lemma and the fact that $\lim_{\epsilon\downarrow 0} \big\langle (X_t^{\delta})^{\otimes n}, p_{\epsilon}(x-\cdot)^{\otimes n}\big\rangle=X_t^{\delta}(x)^{n}$ for almost every $x\in \RR$ almost surely, we get
\begin{align}\label{unbx1}
\sup_{\delta\in (0,1)}\sup_{(t,x)\in [0,T]\times \RR}\EE [X_t^{\delta}(x)^{n}]<\infty,
\end{align}
as well. The proof of this lemma is complete by combining inequalities \eqref{unby3} and \eqref{unbx1}.
\end{proof}
In the next lemma, we provide the estimates for time and spatial increments of $Y^{\delta}$ that will be used in the proof of H\"{o}lder continuity of $Y^{\delta}$.

\begin{lemma}\label{lmm_hdyd}
Let $(X^{\delta},Y^{\delta})$ be a weak solution to \eqref{sbm1} with initial condition $X_0\in \cM_F(\RR)$ satisfying Hypothesis \ref{hyp_x0}. Then for any $\alpha\in (0,1)$ and $n\geq 1$, there exists a constant $C>0$ such that
\begin{align}\label{sicr1}
\EE \big[ |Y^{\delta}_t(x)-Y^{\delta}_t(y)|^{2n}\big]\leq C (|x-y|^{n\alpha}\vee |x-y|^n)
\end{align}
and
\begin{align}\label{sicr2}
\EE \big[ |Y^{\delta}_t(x)-Y^{\delta}_s(x)|^{2n}\big]\leq C |t-s|^{\frac{1}{2}n}
\end{align}
for all $0\leq s<t\leq T$, $x,y\in \RR$ and $\delta\in (0,1)$.
\end{lemma}
\begin{proof}
We follow the ideas in Konno and Shiga \cite[Lemma 2.8]{ptrf-88-konno-shiga}. Write $X^{\delta}_{t}(x)$ in the mild formulation,
\[
X^{\delta}_t(x)=\int_{\RR}p_{t}(x-y)X_0(y)dy+\int_0^t\int_{\RR}p_{t-s}(x-y)\widetilde{\sigma}_{ \delta}(s,y,\PP_{Y^{\delta}_s})\sqrt{X^{\delta}_s(y)}W(ds,dy).
\]
Then, for any $t\in [0,T]$ and $x,y\in \RR$, using the semi-group property of the heat kernel and Burkholder-Davis-Gundy's inequality, we can write
\begin{align*}
\EE &\big[|Y^{\delta}_t(x)-Y^{\delta}_t(y)|^{2n}\big] =\EE \bigg[\Big|\int_{\RR}(p_{\delta}(x-z)-p_{\delta}(y-z))X^{\delta}_t(z)dz\Big|^{2n}\bigg]\\
=&\EE \bigg[ \Big|\int_{\RR}dz(p_{\delta}(x-z)-p_{\delta}(y-z))\int_{\RR}dz'p_t(z-z')X_0(z')\\
&+\int_{\RR}dz(p_{\delta}(x-z)-p_{\delta}(y-z))\int_0^t\int_{\RR}p_{t-s}(z-z')\widetilde{\sigma}_{ \delta}(s,z',\PP_{Y^{\delta}_s})\sqrt{X_s^{\delta}(z')}W(ds,dz')\Big|^{2n}\bigg]\\
\leq &c_1 \(I_1+I_2\),
\end{align*} 
where
\[
I_1= \Big|\int_{\RR}dz(p_{t+\delta}(x-z)-p_{t+\delta}(y-z))X_0(z)\Big|^{2n}
\]
and
\[
I_2=\EE\bigg[\Big|\int_0^tds\int_{\RR}dz(p_{t-s+\delta}(x-z)-p_{t-s+\delta}(y-z))^2\widetilde{\sigma}_{ \delta}(s,z',\PP_{Y^{\delta}_s})^2X_s^{\delta}(z)\Big|^{n}\bigg].
\]
The first term is easy to handle. In fact, using Fubini's theorem, the integration by parts formula and Cauchy-Schwarz's inequality, one can show that
\begin{align}\label{sicr11}
I_1= &\Big|\int_{\RR}dz\int_{y}^xd\xi \nabla p_{t+\delta}(\xi-z) X_0(z)\Big|^{2n}=\Big|\int_{\RR}dz\int_{y}^x d\xi p_{t+\delta}(\xi-z) \nabla X_0(z)\Big|^{2n}\nonumber\\
\leq &\Big[\int_{\RR}dz\Big(\int_{y}^x d\xi p_{t+\delta}(\xi-z)\Big)^2\Big]^{n}\Big[\int_{\RR}dz |\nabla X_0(z)|^2\Big]^n\nonumber\\
\leq  &\|X_0\|_{1,2}^{2n}\Big[\int_{\RR}dz\int_{y}^x d\xi \int_{\RR}d\xi'p_{t+\delta}(\xi-z)p_{t+\delta}(\xi'-z)\Big]^{n}=\|X_0\|_{1,2}^{2n}|x-y|^n.
\end{align}
For the second term, by Lemma \ref{lmm_ubn}, Cauchy-Schwarz's inequality and the H\"{o}lder continuity of the heat kernel in space, namely (c.f. Xiong \cite[Lemma 1.4.4]{ws-13-xiong}), for any $\alpha\in(0,1)$,
\[
\int_0^tds\int_{\RR}dz|p_{s}(x-z)-p_{s}(y-z)|^2\leq C|x-y|^{\alpha},
\] 
we have
\begin{align}\label{sicr12}
I_2\leq &C\EE \bigg[\Big|\int_0^tds\int_{\RR}dz'(p_{t-s+\delta}(x-z')-p_{t-s+\delta}(y-z'))^2\widetilde{\sigma}_{ \delta}(s,z',\PP_{Y^{\delta}_s})^2X_s^{\delta}(z')\Big|^n \bigg]\\
\leq & C \sup_{\delta\in (0,1)}\sup_{(t,x)\in [0,T]\times \RR}\EE \big[|X_t^{\delta}(x)|^n \big]\|\sigma\|_{\infty}^{2n}\nonumber\\
&\times \Big(\int_0^tds\int_{\RR}dz'|p_{t-s+\delta}(x-z')-p_{t-s+\delta}(y-z')|^2\Big)^n\nonumber\\
\leq & C |x-y|^{n\alpha},\nonumber
\end{align}
where the last inequality is due to the H\"{o}lder regularity for the heat kernel. Therefore, inequality \eqref{sicr1} is a consequence of \eqref{sicr11} and \eqref{sicr12}. The proof of \eqref{sicr2} is very similar, but using H\"{o}lder regularity of the heat kernel in time, namely,
\[
\int_0^sdr\int_{\RR}dx|p_{t-r}(x)-p_{s-r}(x)|^2\leq C|t-s|^{\frac{1}{2}}.
\]
For the sake of brevity, we omit the remainder of proof of \eqref{sicr2}. The proof of this lemma is therefore complete. 
\end{proof}

\begin{lemma}\label{coro}
Assume Hypothesis \ref{hyp_x0}. Then for any $\delta>0$, the following properties hold.
\begin{enumerate}[(i)]
\item  Fix $x \in \RR$.
There exists a nonnegative random variable $Z_{\delta}(x)\in L^1(\Omega)$ such that $\sup_{\delta\in (0,1)} \EE [Z_{\delta}(x)]<\infty$ and 
\begin{align}\label{coro1}
|Y^{\delta}_t(x)-Y^{\delta}_s(x)|\leq Z_{\delta}(x)^{\frac{1}{8}}|t-s|^{\frac{1}{16}}
\end{align}
for all $s,t\in [0,T]$.

\item  Fix $R>0$. There exists a nonnegative random variable $Z_{\delta, R}'$ such that $ \EE [Z_{\delta, R}']$ is uniformly bounded in $\delta\in (0,1)$, and
\begin{align}\label{coro2}
|Y^{\delta}_t(x)-Y^{\delta}_t(y)-Y^{\delta}_s(x)+Y^{\delta}_s(y)|\leq (Z_{\delta,R}')^{\frac{1}{16}}|t-s|^{\frac{1}{32}}|x-y|^{\frac{1}{32}}
\end{align}
for all $s,t\in[0,T]$ and $x,y\in [-R, R]$.
\end{enumerate} 
\end{lemma}
\begin{proof}
The proof of this lemma is based on (multi-parameter) Garsia--Rodemich--Rumsey's inequality (c.f. Hu \cite[Theorem 2.1]{ws-16-hu} and Hu and Le \cite[Theorem 2.3]{spa-13-hu-le}). We only provide the proof of inequality \eqref{coro2}. The proof of  \eqref{coro1} can be done similarly. Due to Lemma \ref{lmm_hdyd}, we can write
\begin{align*}
&\EE \big[|Y^{\delta}_t(x)-Y^{\delta}_t(y)-Y^{\delta}_s(x)+Y^{\delta}_s(y)|^{2n}\big]\\
\leq& C\Big[\EE \big[|Y^{\delta}_t(x)-Y^{\delta}_t(y)|^{2n}\big]+\EE\big[ |Y^{\delta}_s(x)-Y^{\delta}_s(y)|^{2n}\big]\Big]\\
\leq &C(|x-y|^{n\alpha}\vee|x-y|^{n})
\end{align*}
and
\begin{align*}
&\EE \big[|Y^{\delta}_t(x)-Y^{\delta}_t(y)-Y^{\delta}_s(x)+Y^{\delta}_s(y)|^{2n}\big]\\
\leq& C\Big[\EE \big[|Y^{\delta}_t(x)-Y^{\delta}_s(x)|^{2n}\big]+\EE\big[ |Y^{\delta}_t(y)-Y^{\delta}_s(y)|^{2n}\big]\Big]\\
\leq &C|t-s|^{\frac{1}{2}n}
\end{align*}
for some $\alpha\in (0,1)$ and $C$ depending on $\alpha$. It follows that for any $\gamma \in(0,1)$,
\begin{align}\label{ieq_rec1}
&\EE \big[|Y^{\delta}_t(x)-Y^{\delta}_t(y)-Y^{\delta}_s(x)+Y^{\delta}_s(y)|^{2n}\big]\nonumber\\
\leq& C\Big[\EE \big[|Y^{\delta}_t(x)-Y^{\delta}_t(y)|^{2n}\big]+\EE\big[ |Y^{\delta}_s(x)-Y^{\delta}_s(y)|^{2n}\big]\Big]\nonumber\\
\leq &C(|x-y|^{n\alpha}\vee|x-y|^{n})^{\gamma}|t-s|^{\frac{1}{2}(1-\gamma)n}.
\end{align}
In order to apply Garsia--Rodemich--Rumsey's inequality, we choose $\alpha=\frac{3}{4}$ and $\gamma=\frac{1}{2}$, let $\Psi:\RR\to \RR_+$, $\rho_1:[0,T]^2\to \RR_+$ and $\rho_2:\RR^2\to \RR_+$ be given by
\[
\Psi(x)=|x|^{16}, \rho_1(t,s)=|t-s|^{\frac{5}{32}}\ \mathrm{and}\ \rho_2(x,y)=|x-y|^{\frac{5}{32}}
\]
respectively, and let
\begin{align*}
Z_{\delta,R}'=& \int_{[0,T]^2}dsdt\int_{[-R,R]^2}dxdy \Psi\Big(\frac{|Y^{\delta}_t(x)-Y^{\delta}_t(y)-Y^{\delta}_s(x)+Y^{\delta}_s(y)|}{\rho_1(s,t)\rho_2(x,y)}\Big).
\end{align*}
Then, by Fubini's theorem for nonnegative functions and inequality \eqref{ieq_rec1}, we can show that
\begin{align*}
\EE [Z_{\delta,R}']\leq &C\int_{[0,T]^2}dsdt\int_{[-R,R]^2}dxdy (|x-y|^{\frac{1}{2}}\vee|x-y|^{\frac{3}{2}})|t-s|^{-\frac{1}{2}},
\end{align*}
that is uniformly bounded in $\delta>0$. On the other hand as a consequence of Hu and Le \cite[Theorem 2.3]{spa-13-hu-le}, we have
\begin{align*}
|Y^{\delta}_t(x)-Y^{\delta}_t(y)-Y^{\delta}_s(x)+Y^{\delta}_s(y)|\leq &c_1\int_0^{|t-s|}\rho_1(du)\int_{0}^{|x-y|}\rho_2(dv)\Psi^{-1}\Big(\frac{c_2Z_{\delta,R}'}{u^2v^2}\Big)\\
\leq &C(Z_{\delta,R}')^{\frac{1}{16}}|t-s|^{\frac{1}{32}}|x-y|^{\frac{1}{32}}.
\end{align*}
Here, the constant $C$ is independent of $s,t,x$ and $y$, and thus can be absorbed into the random variable $Z_{\delta,R}'$. The proof of this lemma is complete.
\end{proof}

\begin{proposition}\label{proptt}
Assume Hypothesis \ref{hyp_x0}. Then, $\{Y^{\delta}\}_{\delta>0}$ is a tight sequence in $C([0,T]\times \RR;\RR)$.
\end{proposition}
\begin{proof}
It suffices to verify conditions \eqref{lmtt2} and \eqref{lmtt3} in Lemma \ref{lmm_tt}.  For condition \eqref{lmtt2}, using Lemma \ref{coro} (i) and Chebyshev's inequality, for any $x\in \RR$,
\begin{align*}
\PP\Big(\sup_{0\leq s\leq t\leq T, |t-s|<\epsilon}&|Y^{\delta}_t(x)-Y^{\delta}_s(x)|>\rho\Big)\leq \rho^{-8}\EE \big[ \sup_{0\leq s\leq t\leq T, |t-s|<\epsilon}|Y_t^{\delta}(x)-Y_s^{\delta}(x)|^8\big]\\
\leq &C\rho^{-8}\EE \big[ \sup_{0\leq s\leq t\leq T, |t-s|<\epsilon}Z_{\delta}(x)|t-s|^{\frac{1}{2}}\big]
\leq C\rho^{-8}\epsilon^{\frac{1}{2}}\EE [Z_{\delta}(x)]\to 0
\end{align*}
uniformly in $\delta>0$, as $\epsilon\downarrow 0$. This verifies condition \eqref{lmtt2}.

The proof of condition \eqref{lmtt3} is similar. Notice that
\[
|Y^{\delta}_t(x)-Y^{\delta}_t(y)|\leq |Y^{\delta}_t(x)-Y^{\delta}_t(y)-Y^{\delta}_0(x)+Y^\delta_0(y)|+|Y^{\delta}_0(x)-Y^\delta_0(y)|.
\]
This implies 
\begin{align*}
&\PP\Big(\sup_{0\leq t\leq T, -R\leq x\leq y\leq R, |x-y|<\epsilon}|Y^{\delta}_t(x)-Y^{\delta}_t(y)|>\rho\Big)\leq P_1+P_2,
\end{align*}
where
\begin{align*}
P_1=\PP\Big(\sup_{0\leq t\leq T, -R\leq x\leq y\leq R, |x-y|<\epsilon}|Y^{\delta}_t(x)-Y^{\delta}_t(y)-Y^{\delta}_0(x)+Y^\delta_0(y)|>\frac{\rho}{2}\Big)
\end{align*}
and
\begin{align*}
P_2=\PP \big(\sup_{ -R\leq x\leq y\leq R, |x-y|<\epsilon}|Y^{\delta}_0(x)-Y^\delta_0(y)|>\frac{\rho}{2}\big).
\end{align*}
By a similar argument as in \eqref{sicr11}, we can show that  $|Y^{\delta}_0(x)-Y^\delta_0(y)|\leq \|X_0\|_{1,2}|x-y|^{\frac{1}{2}}$
for some constant uniformly in $\delta>0$. Thus, $P_2=0$ for $\epsilon>0$ that is small enough. The estimate for $P_1$ can be done as an application of Chebyshev's inequality and Lemma \ref{coro} (ii). Finally, we can conclude that 
\[
\PP\Big(\sup_{0\leq t\leq T, -R\leq x\leq y\leq R, |x-y|<\epsilon}|Y^{\delta}_t(x)-Y^{\delta}_t(y)|>\rho\Big)\to 0,
\]
uniformly in $\delta>0$, as $\epsilon\downarrow 0$. Thus condition \eqref{lmtt3} holds. Therefore, $\{Y^{\delta}\}_{\delta}$ is a tight sequence in $C([0,T]\times \RR;\RR)$ according to Lemma \ref{lmm_tt}.
\end{proof}

\begin{proof}[Proof of Theorem \ref{thm_exs}:  existence]
  We only present the proof under Hypothesis \ref{hyp_0}. For other situations under Hypothesis \ref{hyp_fn} or \ref{hyp_ifn}, we can modify the proof following the idea as in Remark \ref{rmk_mmt}.   Let $(X^{\delta},Y^{\delta})$ be a solution to \eqref{sbm1}. Then, by Proposition \ref{proptt}, there exists a sequence $\delta_n\downarrow$ such that $Y^{\delta_n}$ is convergent in distribution in $C([0,T]\times \RR; \RR)$ to some random field $Y$. By the typical tightness argument, one can show that $\{X^{\delta_n}:n\geq 1\}$ is tight in $D([0,T];\cM_F(\RR))$. Therefore, by taking subsequence of $\{X^{\delta_n}\}$, we can assume it converges in distribution to some $\cM_F(\RR)$-valued process $X$. By the Skorohod representation theorem, we can also assume this convergence is almost surely.

In the next step, we show that $(X,Y)$ is a weak solution to the following equation
\[
\frac{\partial}{\partial t}X_t(x)=\frac{1}{2}\Delta X_t(x)+\sigma(t,x,\PP_{Y_t(x)})\sqrt{X_t(x)}\dot{W}(t,x).
\]
Equivalently, it suffices to show that $X_t$ is a solution to the following martingale problem (c.f. Stroock and Varadhan \cite{psbs-72-stroock-varadhan} and Kurtz \cite{sa-11-kurtz}), for any $\phi\in \cS(\RR)$,
\begin{align}
M_t(\phi)=X_t(\phi)-X_0(\phi)-\frac{1}{2}\int_0^tX_s(\Delta \phi)ds
\end{align}
is a continuous square integrable martingale, with quadratic variation
\begin{align}\label{ax_mart2}
\langle M(\phi)\rangle_{t}=\int_0^t\sigma(s,x, \PP_{Y_s(x)})^2\phi(x)^2X_s(dx)ds.
\end{align}
Notice that, using Perkins \cite[Theorem II.4.5]{springer-02-perkins}, we have
\begin{align*}
\langle M(\phi)\rangle_{t}=&\lim_{n\to\infty}\langle M^{\delta_n}(\phi)\rangle_{t}\\
=&\lim_{n\to\infty}\int_0^tds\int_{\RR}X^{\delta_n}_s(dx)\Big(\int_{\RR}dyp_{\delta_n}(x-y)\sigma(s,y, \PP_{Y^{\delta_n}_s(y)})\Big)^2\phi(x)^2.
\end{align*}
To verify the limit, we compute the following quantity, 
\begin{align*}
&\Big|\int_0^tds\int_{\RR}X^{\delta_n}_s(dx)\Big(\int_{\RR}dyp_{\delta_n}(x-y)\sigma(s,y, \PP_{Y^{\delta_n}_s(y)})\Big)^2\phi(x)^2\\
&-\int_0^tds\int_{\RR}\sigma(s,x,\PP_{Y_s(x)})^2\phi(x)^2X_s(dx)\Big|\leq I_1+I_2
\end{align*}
where
\begin{align*}
I_1=&\Big|\int_0^tds\int_{\RR}\Big[\Big(\int_{\RR}dyp_{\delta_n}(x-y)\sigma(s,y, \PP_{Y^{\delta_n}_s(y)})\Big)^2-\sigma(s,x,\PP_{Y_s(x)})^2\Big]\phi(x)^2X_s^{\delta_n}(dx)\Big|
\end{align*}
and
\begin{align*}
I_2=&\Big|\int_0^tds\int_{\RR}dx\sigma(s,x, \PP_{Y_s(x)})^2\phi(x)^2\big(X^{\delta_n}_s(dx)-X_s(dx)\big)\Big|.
\end{align*}
It is clear that $I_2\to 0$ as $n\to\infty$ because $X^{\delta_n}\to X$ in $D([0,T];\cM_F(\RR))$. On the other hand, notice that $X_s^{\delta_n}$ has a density almost surely. Thus, by Cauchy-Schwarz's inequality
\begin{align*}
I_1\leq& \bigg(\int_{\RR}dx\Big[\Big(\int_{\RR}dyp_{\delta_n}(x-y)\sigma(s,y, \PP_{Y^{\delta_n}_s(y)})\Big)^2-\sigma(s,x,\PP_{Y_s(x)})^2\Big]^2\phi(x)^2\bigg)^{\frac{1}{2}}\\
&\times\int_0^t\Big(\int_{\RR}\phi(x)^2X_s^{\delta_n}(x)^2dx\Big)^{\frac{1}{2}}:=I_{11}\times I_{12}.
\end{align*}
By dominated convergence theorem, we know that $I_{11}\to 0$ as $n\to\infty$. Furthermore, by Lemma \ref{lmm_ubn}, one can show that $\EE [I_{12}]$ is uniformly bounded in $n$. As a consequence, it follows by Fatou's lemma that
\[
\EE\big[ \liminf_{n\to\infty} I_1\big]\leq \lim_{n\to \infty}\EE [I_1]=0.
\]
This implies that $\liminf_{n\to \infty}I_1=0$ almost surely. That is enough to prove \eqref{ax_mart2} because we can take subsequence so that the above $\liminf_{n\to\infty}$  
can be replaced by   $\lim_{n\to\infty}$.

Finally, we complete the proof of this theorem by showing that for any $t\in [0,T]$, the distribution of $X_t$ and $Y_t$ coincide. Indeed, for any $\phi\in \cS(\RR)$, we can show that
\begin{align*}
\EE [\langle X_t, \phi\rangle] -\EE [\langle Y_t, \phi\rangle]\leq &\big|\EE [\langle X_t, \phi\rangle] - \EE [\langle X^{\delta_n}_t, \phi\rangle]\big|+ \big|\EE [\langle Y_t, \phi\rangle] - \EE [\langle Y^{\delta_n}_t, \phi\rangle] \big|\\
&+ \big|\EE [\langle X^{\delta_n}_t, \phi\rangle] - \EE [\langle Y^{\delta_n}_t, \phi\rangle] \big|.
\end{align*}
It suffices to show the convergence to $0$ of the last term. Recall that $Y_t^{\delta_n}(x)=\langle X_t^{\delta_n},p_{\delta}(x-\cdot)\rangle$ for all $(t,x)\in [0,T]\times \RR$. Thus, we can write
\begin{align*}
\big|\EE [\langle X^{\delta_n}_t, \phi\rangle] - \EE [\langle Y^{\delta_n}_t, \phi\rangle] \big|=&\EE \bigg[\Big|\int_{\RR}dx\int_{\RR}dy  p_{\delta_n}(x-y)(\phi(x)-\phi(y))X_t(dx)\Big|\bigg]\\
\leq &\sup_{x\in \RR}\Big|\int_{\RR}dy  p_{\delta_n}(x-y)(\phi(x)-\phi(y))\Big|\EE [\langle X_t, \1\rangle].
\end{align*}
Notice that
\begin{align*}
&\Big|\int_{\RR}dy  p_{\delta_n}(x-y)(\phi(x)-\phi(y))\Big|\\
\leq &\int_{|x-y|\leq \delta_n^{\frac{1}{3}}}dy  p_{\delta_n}(x-y)|\phi(x)-\phi(y)|+2\|\phi\|_{\infty}\int_{|x-y|> \delta_n^{\frac{1}{3}}}dy p_{\delta_n}(x-y)\\
\leq & \|\phi\|_{1,\infty}\delta_n^{\frac{1}{3}}\int_{|z|\leq \delta_n^{\frac{1}{3}}}dz  p_{\delta_n}(z)+2\|\phi\|_{\infty}\int_{|z|> \delta_n^{-\frac{1}{6}}}dz \frac{1}{\sqrt{2\pi}}e^{-\frac{z^2}{2}}\to 0,
\end{align*}
as $n\to\infty$. As a consequence, we have $\EE [\langle X_t, \phi\rangle] =\EE [\langle Y_t, \phi\rangle]$ for all $\phi\in \cS(\RR)$. The proof of the existence part of Theorem \ref{thm_exs} is complete.
\end{proof}

\section{A moment formula and some estimates}\label{sec_mmtfor}
In the proof of Lemma \ref{lmm_ubn}, we could presumably obtain a moment formula for $X_t(x)$ based on the method of moment duality by simply letting $\delta\downarrow 0$   in  \eqref{unby1}.   However, as a product of two dependent random variables, 
this formula \eqref{unby1} depends on a nonlinear function of the pure jump process $n_t$ and a linear function of $f_t^{\delta}$ which is a deterministic process with random jumps. If we want to write the moment formula in an explicit form, namely without involving further expectations of the solution, one needs to deal with all the jumps $n\to (n-1), (n-1)\to (n-2)$, etc, until $2\to 1$. A simple calculation related to one jump $k\to (k-1)$ is carried out in the proof of Lemma \ref{lmm_mtgdual}.  We see that this is already complex. It is difficult for us to obtain an explicit formula for the higher order moments for the solution by using the formula derived from \eqref{unby1}.  

Since the moments play a critical role in our approach, in this section by the mild formulation for solution $X_t(x)$ and by using It\^{o}'s formula iteratively, we establish an explicit formula (see Theorem \ref{coro_mnt}) for the moments of $X_t(x)$ that satisfies equation \eqref{sbm}. The proof of this theorem is given in Sections \ref{ssec_itrt} and \ref{ssec_pfmnt}. Afterward, in Section  \ref{ssec_mmtest}, an upper bound for the moments is obtained, which will be used in the proof of uniqueness results in Section \ref{sec_unq}. 

The next theorem presents a moment formula for $X_t(x)$.
\begin{theorem}\label{coro_mnt}
Suppose that $X_0\in\cM_F(\RR)$ satisfies Hypothesis \ref{hyp_x0}. Let $n\in\NN$. Then, for any $(t,x)\in[0,T]\times \RR$, the following equation holds:
\begin{align}\label{for_mmt0}
\EE [X_{t}(x)^n]=&\sum_{n'=0}^{n-1}\sum_{(\alpha,\beta,\tau)\in \cJ_{n,n'}}\prod_{i=1}^n\Big(\int_{\RR}dz p_{t}(x-z)X_0(z)\Big)^{1-\alpha_i}\nonumber\\
&\times \int_{\mathbb{T}_{n'}^t}d\bs_{n'}\int_{\RR^{n'}}d\bz_{n'}\prod_{i=1}^{n'}\Big(\int_{\RR}dz p_{s_i}(z_i-z)X_0(z)\Big)^{1-\beta_i} \prod_{i=1}^{|\alpha|}p(t-s_{\tau(i)},x-z_{\tau(i)})\nonumber\\
&\times\prod_{i=|\alpha|+1}^{2n'}p(s_{\iota_{\beta}(i-|\alpha|)}-s_{\tau(i)},z_{\iota_{\beta}(i-|\alpha|)}-z_{\tau(i)})\prod_{i=1}^{n'}\sigma(s_i,z_i,\PP_{X_{s_i}(z_i)})^2,
\end{align}
where the set $\cJ_{n,n'}$ of triples $(\alpha,\beta,\tau)$ is defined as in \eqref{def_jkm} below,
\begin{align}\label{def_dtk}
\mathbb{T}_{n'}^t=\big\{\bs_{n'}=(s_1,\dots,s_{n'})\in [0,T]^{n'}:0<s_{n'}<s_{n'-1}<\dots<s_1<t\big\},
\end{align}
and $p(t,x)=p_t(x)$ to avoid long sub-indexes.
\end{theorem}

Before presenting the proof of Theorem \ref{coro_mnt}, let us first take a look at the following example that may bring us some insight into this moment formula. Let $X$ be a solution to \eqref{sbm}. Then, for any fixed $(t,x)\in [0,T]\times \RR$, the following mild formulation holds,
\[
X_t(x)=\int_{\RR}dyp_t(x-y)X_0(y)+\int_0^t\int_{\RR}p_{t-s}(x-y)\sigma(s,y,\PP_{X_s(y)})\sqrt{X_s(y)}W(ds,dy).
\]
Due to the singularity of $p_{t-s}(x-y)$ when 
$t=s$, $X_t(x)$ is not a semimartingale in $t$.  We introduce an auxiliary process $Y^t=\{Y^t_s(x):0\leq s\leq t, x\in \RR\}$, where
\begin{align}\label{def_yst}
Y^t_s(x)=\int_{\RR}dyp_t(x-y)X_0(y)+\int_0^s\int_{\RR}p_{t-r}(x-y)\sigma(r,y,\PP_{X_r(y)})\sqrt{X_r(y)}W(dr,dy).
\end{align}
As a process in $s$ it is a semimartingale. 
Applying It\^{o}'s formula on $f(X_t(x))=X_t(x)^n=(Y^t_t(x))^n$ with some $n\in \NN$, and noticing that $\widehat{\sigma}(t,x)=\sigma(t,x,\PP_{X_t(x)})$ is  a deterministic function, we can write
\begin{align*}
X_t(x)^n=&\Big[\int_{\RR}dyp_t(x-y)X_0(y)\Big]^n+n\int_0^t\int_{\RR}p_{t-s}(x-y)\widehat{\sigma}(s,y)\sqrt{X_s(y)}Y_s^t(x)^{n-1}W(ds,dy)\\
&+\frac{1}{2}n(n-1)\int_0^tds\int_{\RR}dyp_{t-s}(x-y)^2\widehat{\sigma}(s,y)^2Y_s^t(x)^{n-2}X_s(y)\,. 
\end{align*}  
Taking expectations on both sides, one gets
\begin{align*}
\EE [X_t(x)^n] = &\EE [Y_t^t(x)^n]=\Big(\int_{\RR}p_t(x-y)X_0(y)dy\Big)^n\\
&+\frac{1}{2}n(n-1)\int_0^tds\int_{\RR}dyp_{t-s}(x-y)^2\widehat{\sigma}(s,y)^2\EE \big[Y_s^t(x)^{n-2}Y^s_s(y)\big].
\end{align*} 
In other words, $\EE [Y^t_t(x)^n]$ can be represented in terms of $\{\EE [Y^t_s(x)^{n-2}Y_s^s(y)]: (s,y)\in [0,t]\times \RR\}$. Applying It\^{o}'s formula to  $Y^t_s(x)^{n-2}Y_s^s(y)$, we can write $\EE (Y^t_s(x)^{n-2}Y_s^s(y))$ in terms of $\{\EE [Y^t_r(x)^{n-4}Y_r^s(y)Y^r_r(z)]:(r,z)\in [0,s]\times \RR\}$ and $\{\EE [Y^t_r(x)^{n-3}Y^r_r(z)]:(r,z)\in [0,s]\times \RR\}$. In fact, 
each time when we apply 
It\^{o}'s formula to a nominal $f$ of degree $n$ 
we have two terms:  the first derivative and the second derivative terms. Noticing that we can write 
\eqref{def_yst} as $dY_s^t=dY_s^{(t, 1)}+dY_s^{(t,2)}$,
the first term containing no unknown and the second term containing a  square root of unknown, 
the expectation of the first derivative term of $f$ will produce
a degree less (of the unknowns). As for the second derivative, observe that the second derivative of $f$ is a nominal of degree $n-2$  but it must be multiplied by the quadratic variation of 
$dY_s^{(t,2)}$,  a factor of $X_s(y)$. Thus, we also obtain a nominal of degree $n-1$. This means that 
when we apply 
It\^{o}'s formula, we can represent the expectation of the  nominal 
$\prod_{i=1}^k (Y_{r_i}^{t_i})^{n_i}$ of degree $n=n_1+\cdots+n_k$ 
by the expectation of a polynomial 
of degree $n-1$.   
This iteration can be proceeded a finite number of times until  
we arrive at the expectation of a linear function of  $Y^t_s(y)$, 
whose expectation is immediately computed by 
\eqref{def_yst}. Then, an explicit formula for $\EE [Y_t^t(x)^n]=\EE [X_t(x)^n]$ is obtained.

Therefore, to provide a moment formula for $X_t(x)$ with rigorous proof, we need to know the   expectation  for  the form 
\begin{align*}
\prod_{i=1}^{n}Y^{t_i}_{t_n}(x_i),
\end{align*}
where $n\in\NN$ and $\bt_n=(t_1,\dots,t_n)\in \mathbb{T}_n^T$ (see \eqref{def_dtk}).

\subsection{Iteration for power functions}\label{ssec_itrt}
Let $n\in\NN$, and let $f_0:\RR^{n}\to \RR$ be given by $f_0(\bx_n)=\prod_{i=1}^n x_i$.  Consider the iteration as follows. 

\begin{enumerate}[(1)]
\setcounter{enumi}{-1}
\item The $0$-th iteration just keeps $f_0$ invariant. Denote by $\cF_0=\{f_0\}$, the set of all outputs in the $0$-th iteration. 

\item In the $1$-st iteration, we choose $f_0\in\cF_0$, then differentiate $f_0$ twice with respect to arbitrary arguments and multiply the derivative by $y_1$. Write $\cF_1$ for the collection of all non-zero outputs. Then, $f_1\in \cF_1$ if and only if
\[
f_1(\bx_n,y_1)=\frac{\partial^2f_0}{\partial x_i\partial x_j}(\bx_n)\times y_1=y_1\prod_{\substack{1\leq m\leq n\\ m\notin \{i, j\}}}x_m ,
\]
for some $1\leq i<j\leq n$. Thus, an element $f_1$ in $\cF_1$ will be a function of $(n-1)$-variables. 

\item The $2$-nd iteration is very similar. Choose any $f_1\in \cF_1$. Differentiating $f_1$ twice, then multiplying by $y_2$, and denoting the set of all possible outputs by $\cF_2$. In this case, $f_2\in \cF_2$, if and only if  
\[
f_2(\bx_n,\by_2)=\frac{\partial^2f_1}{\partial x_{i'}\partial x_{j'}}(\bx_n,y_1)\times y_2=y_1y_2\prod_{\substack{1\leq m\leq n\\ m\notin \{i,j,i',j'\}}}x_m,
\]
for some $1\leq i'<j'\leq n$ and $\{i',j'\}\cap \{i,j\}=\emptyset$; or
\[
f_2(\bx_n,\by_2)=\frac{\partial^2f_1}{\partial x_{i'}\partial y_1}(\bx_n,y_1)\times y_2=y_2\prod_{\substack{1\leq m\leq n\\ m\notin \{i,j,i'\}}}x_m,
\]
for some $1\leq i'\leq n$ and $i'\notin\{i,j\}$. Thus, an element $f_2$ in $\cF_2$ will be a function of $(n-2)$-variables. 

\item[$\vdots$]

\item[(n)] In the $n$-th iteration, one should choose any function $f_{n-1}\in \cF_{n-1}$ and then differentiate it twice and multiply by $y_n$. In fact, $\cF_{n-1}$ consists of the   function  of 
$n-(n-1)=1$ variable, namely,   $f_{n-1}(\bx_n,\by_{n-1})=y_{n-1}$ for all $(\bx_n,\by_{n-1})\in \RR^{2n-1}$. Therefore, after this iteration, $\cF_n$, the set of non-zero outputs is empty. The whole iteration stops. 
\end{enumerate}

Fix $ n'\in \{0,\dots,n-1\}$, and choose any function $f_{n'}\in \cF_{n'}$. Then, we associate a multi-index $[\alpha,\beta]=[(\alpha_1,\dots,\alpha_n),(\beta_1,\dots,\beta_{n'})]\in \{0,1\}^{n+n'}$ to $f_{n'}$, such that
\begin{align}\label{def_ab}
f_{n'}(\bx_n,\by_{n'})=\prod_{i=1}^n x_i^{1-\alpha_i}\prod_{i=1}^{n'}y_i^{1-\beta_i}.
\end{align}
Denote by $\cI_{n,n'}$ the collection of  all multi-index $[\alpha,\beta]$  such that there exists a function $f_{n'}\in \cF_{n'}$ with the representation \eqref{def_ab}. In particular, if $n'=0$, then $\beta$ should be a ``$0$-dimensional index'', and it will be written as $\beta=\partial$. From the definition of $\cF_{n'}$, it is easy to see that $[\alpha,\beta]\in \{0,1\}^{n+n'}$ is an element of $\cI_{n,n'}$, if and only if
\begin{enumerate}[(i)]
\item $\beta_{n'}=0$.

\item Let $|\alpha|=\sum_{i=1}^n \alpha_i$ and $|\beta|=\sum_{i=1}^{n'}\beta_i$, then $|\alpha|+|\beta|=2n'$.
\end{enumerate}

For example, assume $n=3$. Then, 
\[
\cI_{3,0}=\{[(0,0,0),\partial]\},
\]
 \[\cI_{3,1}=\{[(1,1,0),(0)], [(1,0,1),(0)],[(0,1,1), (0)]\},
\]
and
\[
\cI_{3,2}=\{[(1,1,1),(1,0)]\}.
\]

On the other hand, for any $[\alpha,\beta]\in \cI_{n,n'}$ with $0\leq n'\leq n-1$, there exist $1\leq j_1<\cdots<j_{|\alpha|}\leq n$ such that $\alpha_{j_1}=\dots=\alpha_{j_{|\alpha|}}=1$ and $\alpha_i=0$ for all $i\in \{1,\dots, n\}\setminus \{j_1,\dots, j_{|\alpha|}\}$. This fact allows us to define a one to one increasing map $\iota_{\alpha}:\{1,\dots,|\alpha|\}\to \{1,\dots,n\}$ such that $\iota_{\alpha}(i)=j_i$. The function $\iota_{\beta}:\{1,\dots, |\beta|\}\to \{1,\dots, n'\}$ is also defined in a similar way. 

Given any multi-index $[\alpha,\beta]\in \cI_{n,n'}$, we can find a unique $f_{n'}\in \cF_{n'}$ satisfying \eqref{def_ab}. Notice that $f_{n'}$ is the output of the $n'$-th iteration for a function $f_{n'-1}\in \cF_{n'-1}$. But it is impossible in general to recover this $f_{n'-1}$ from $f_{n'}$, because we do not know which operator is applied in the $n'$-th iteration. 
Similarly, the operator in the $k$-th iteration with $k=1,\dots, n'-1$ is also unknown. 
Thus, it is reasonable to introduce a map $\tau:\{1,\dots, |\alpha|+|\beta|=2n'\}\to \{1,\dots, n'\}$,   given as follows. If $i\leq |\alpha|$ and the differentiation $\partial/\partial x_{\iota_{\alpha}(i)}$ occurs in the $k$-th iteration with $k\in\{1,\dots, n'\}$, then $\tau(i)=k$; instead, if $i>|\alpha|$ and $\partial/\partial y_{\iota_{\beta}(i-|\alpha|)}$ occurs in the $k$-th iteration, then $\tau(i)=k$.  
To be more precise, given $[\alpha,\beta]\in \cI_{n,n'}$, we write $\cK_{n,0}^{\alpha,\beta}=\emptyset$, and otherwise the set $\cK_{n,n'}^{\alpha,\beta}$ with $n'>0$, is defined to be a collection of maps $\tau:\{1,\dots, 2n'\}\to \{1,\dots, n'\}$ satisfying the following properties,
\begin{enumerate}[(i)]
\item For any $k\in \{1,\dots,n'\}$, there exist   exactly two indexes $ \{i_1,i_2\}\in\{1,\dots, 2n'\}$  such that $\tau(i_1)=\tau(i_2)=k$.
\item For all $i\in \{|\alpha|+1,\dots, 2n'\}$, $\tau(i)>\iota_{\beta}(i-|\alpha|)$.
\end{enumerate}
Property (i) means that in each iteration, differentiation occurs twice. 
Additionally, we notice that for any $k\in \{1,\dots, n'\}$, the set $\cF_{k}$ consists of functions of $(\bx_n,\by_k)$. Thus at the $(k+1)$-th iteration, the differentiation $\frac{\partial}{\partial y_{k'}}$ would not occur, if $k'>k$. 
This fact explains why we need property (ii) to be fulfilled as well. Furthermore, it ensures that if $1\leq i_1<i_2\leq 2n'$ are such that $\tau(i_1)=\tau(i_2)=1$, then $i_2\leq |\alpha|$.

For any $n\in\NN$ and nonnegative integer $n'\leq n-1$, we denote by
\begin{align}\label{def_jkm}
\cJ_{n,n'}=\{(\alpha,\beta,\tau):[\alpha,\beta]\in\cI_{n,n'},\tau\in\cK_{n,n'}^{\alpha,\beta}\}.
\end{align}
Note that if $n'=0$, then for any $n\geq 1$, $(\alpha,\beta,\tau)\in \cJ_{n,0}$ if and only if $\alpha=\mathbf{0}_n$, the $0$-vector in $\RR^n$, $\beta=\partial$ and $\tau\in\emptyset$. In this case, we write $\cJ_{n,0}=\{(\mathbf{0}_n,\partial,\partial)\}$. Let us take a look at the following example of an element in $\cJ_{4,2}$.
\begin{figure}
\centering
\begin{tabular}{ |c||c|c|c|c||c|c| } 
 \hline
  & $x_1$ & $x_2$ & $x_3$ & $x_4$ & $y_1$ & $y_2$  \\ 
  \hline
  $f_0(\bx_4)=x_1x_2x_3x_4$ & {\color{red} $\bullet$} & {\color{red}$\bullet$} & {\color{red}$\bullet$} & {\color{red}$\bullet$} &  &   \\
  \hline
  $f_1(\bx_4,\by_1)=x_2x_3y_1$ &  & {\color{red}$\bullet$} & {\color{red}$\bullet$} &  & {\color{blue}$\bullet$} &  \\
  \hline
  $f_2(\bx_4,\by_2)=x_2y_2$ &  & {\color{red}$\bullet$} &  &  &  & {\color{blue}$\bullet$} \\
  \hline
\end{tabular}
 \captionof{table}{ \label{tab_1}}
\end{figure}
  Consider 
\[
f_2(\bx_4,\by_2)=x_2y_2=x_1^{1-\alpha_1}x_2^{1-\alpha_2}x_2^{1-\alpha_2}x_4^{1-\alpha_4}y_1^{1-\beta_1}y_2^{1-\beta_2}\in \cF_{2},
\]
obtained by the iteration given as in Table \ref{tab_1}. The associated multi-index to $f_2$ is $[\alpha,\beta]=[(1,0,1,1),(1,0)]$. In this case, $|\alpha|=3$, $|\beta|=1$, $\iota_{\alpha}:\{1,2,3\}\to\{1,2,3,4\}$, with $\iota_{\alpha}(1)=1$, $\iota_{\alpha}(2)=3$, $\iota(3)=4$, and $\iota_{\beta}:\{1\}\to \{1,2\}$, given by $\iota_{\beta}(1)=1$. On the other hand, in the first iteration, we differentiate $x_1=x_{\iota_{\alpha}(1)}$ and $x_4=x_{\iota_{\alpha}(3)}$. Thus, $\tau(1)=\tau(3)=1$. Similarly, in the next iteration, we differentiate $x_3=x_{\iota_{\alpha}(2)}$ and $y_1=y_{\iota_{\beta}(4-3)}$. It follows that $\tau(2)=\tau(4)=2$.

The next lemmas provide some properties of the set $\cJ_{n,n'}$, which will be used in the proof of Theorem \ref{coro_mnt}.
\begin{lemma}\label{lmm_cat}
Let $n\in\NN$, let $n'\in \{1,\dots, n-1\}$, and let $\cJ_{n,n'}$ be given as in \eqref{def_jkm}. Denote by
\[
\cJ_{n,n'}'=\big\{(i,j,\alpha,\beta,\tau):1\leq i<j\leq n,  (\alpha,\beta,\tau)\in \cJ_{n-1,n'-1}\big\}.
\]
Then, there exists a bijection $\cM:\cJ_{n,n'}\to \cJ_{n,n'}'$.
\end{lemma}

  Lemma \ref{lmm_cat} states that any $n'$ times iteration for an $n$-variable (product) function can be decomposed uniquely and reversible to a  single iteration for an $n$-variable function and an $n'-1$ times iteration for an $(n-1)$-variable function.

Take the iteration described in Table \ref{tab_1} as an example. Note that in the first iteration, we differentiate $x_1$ and $x_4$. Thus, we write $(i,j)=(1,4)$. Then, consider $f_1(\bx_4,y_1)$ as a new function $f_0'(\bx_3')=x_1'x_2'x_3'$ with $x_1'=x_2$, $x_2'=x_3$ and $x_3'=y_1$. Then, by deleting the row of $f_0$, and  the columns of $x_1$ and $x_4$  in Table \ref{tab_1}, we get 
 \begin{figure}
 \centering
\begin{tabular}{ |c||c|c|c||c| } 
 \hline
  & $x_1'=x_2$ & $x_2'=x_3$ & $x_3'=y_1$ & $y_1'=y_2$    \\ 
  \hline
  $f_0'(\bx_3')=x_1'x_2'x_3'$ & {\color{red}$\bullet$} & {\color{red}$\bullet$} & {\color{red}$\bullet$}  &   \\
  \hline
  $f_1'(\bx_3',\by_1')=x_1'y_1'$ &  {\color{red}$\bullet$} & &  & {\color{blue}$\bullet$}   \\
  \hline
\end{tabular}
 \captionof{table}{\label{tab_2}}
\end{figure}
  As shown in Table \ref{tab_2}, the iteration $f_1\Rightarrow f_2$ can be understood as $f_0'\Rightarrow f_1'$ with $f_1'(\bx_3',\by_1')=x_1'y_1'=x_2y_2$. Then, the associated triple $(\alpha',\beta',\tau')\in \cJ_{3,1}$ can be written as  $\alpha'=(0,1,1)$, $\beta'=(0)$ and $\tau':\{1,2\}\to\{1\}$ given by $\tau'(1)=\tau'(2)=1$. In this case, it is easy to check that $(i,j,\alpha',\beta',\tau')\in \cJ_{4,2}'$. Conversely, it is not hard to see that we can also recover $(\alpha,\beta,\tau)$ as in Table \ref{tab_1} from $(i,j,\alpha',\beta',\tau')\in \cJ_{4,2}'$ with $(i,j)=(1,4)$ and $(\alpha',\beta',\tau')$ defined as in Table \ref{tab_2}. 

\begin{proof}[Proof of Lemma \ref{lmm_cat}]
Choose any $(\alpha,\beta,\tau)\in \cJ_{n,n'}$. Then, there exist $j_1,j_2\in \{1,\dots, |\alpha|\}$ with $j_1<j_2$  such that $\tau(j_1)=\tau(j_2)=1$. We define 
\[
\cM(\alpha, \beta, \tau) = (\iota_{\alpha}(j_1), \iota_{\alpha}(j_2), \alpha', \beta', \tau')\]
with
\begin{align}\label{def_alph'}
\alpha'=(\alpha_1,\dots,\alpha_{\iota_{\alpha}(j_1)-1},\alpha_{\iota_{\alpha}(j_1)+1},\dots, \alpha_{\iota_{\alpha}(j_2)-1},\alpha_{\iota_{\alpha}(j_2)+1},\dots, \alpha_{n},\beta_1),
\end{align}
\begin{align}\label{def_beta'}
\beta'=(\beta_2,\dots,\beta_{n'}),
\end{align}
and
\begin{align}\label{def_tau'2}
\tau'(i)=\begin{cases}
\tau(i)-1, &i< j_1,\\
\tau(i+1)-1, & j_1\leq i< j_2-1,\\
\tau(i+2)-1, & j_2-1\leq i\leq 2n'-2.
\end{cases}
\end{align}
It is clear that $\alpha'\in \{0,1\}^{n-1}$, $\beta'\in \{0,1\}^{n'-1}$ with $\beta'_{n'-1}=\beta_{n'}=0$ and $|\alpha'|+|\beta'|=|\alpha|+|\beta|-2=2(n'-1)$. In other words, $[\alpha',\beta']\in \cI_{n-1,n'-1}$. 

It suffices to show that $\tau'\in \cK_{n-1,n'-1}^{\alpha',\beta'}$. By definition \eqref{def_tau'2} and the fact that $\tau(j_1)=\tau(j_2)=1$, it is easy to see that $\tau':\{1,\dots, 2(n'-1)\}\to \{1,\dots, n'-1\}$, and for every $k\in \{1,\dots, n'-1\}$, there exists $1\leq i<j\leq 2(n'-1)$ such that $\tau'(i)=\tau'(j)=k$.

In the next step, we prove that $\tau'(i)>\iota_{\beta'}(i-|\alpha'|)=\iota_{\beta'}(i+2-|\alpha|-\beta_1)$ for all $i\in\{|\alpha'|+1,\dots, 2(n'-1)\}=\{|\alpha|-1,\dots, 2(n'-1)\}$. Choose such an $i$. Noticing that $j_2\leq |\alpha|$, we have $i\geq |\alpha|-1\geq j_2-1$ and thus $i+2\geq |\alpha|+1$. As a consequence,
\[
\tau'(i)=\tau(i+2)-1>\iota_{\beta}(i+2-|\alpha|)-1.
\]
On the other hand, $\beta'_{\iota_{\beta'}(i+2-|\alpha|-\beta_1)}$ is the $(i+2-|\alpha|-\beta_1)$-th non-zero coordinate of $\beta'$. This yields that $\beta_{\iota_{\beta'}(i+2-|\alpha|-\beta_1)+1}=\beta'_{\iota_{\beta'}(i+2-|\alpha|-\beta_1)}$ is the $(i+2-|\alpha|)$-th non-zero coordinate of $\beta$. In other words, 
\[
\iota_{\beta'}(i+2-|\alpha|-\beta_1)+1=\iota_{\beta}(i+2-|\alpha|).
\] 
It follows that $\tau'(i)>\iota_{\beta'}(i-|\alpha'|)=\iota_{\beta'}(i+2-|\alpha|-\beta_1)$ for all $i\in\{|\alpha'|+1,\dots, 2(n'-1)\}=\{|\alpha|-1,\dots, 2(n'-1)\}$. 
Consequently, we have $\tau\in \cK_{n-1,n'-1}^{\alpha',\beta'}$, and thus $\cM$ maps $\cJ_{n,n'}$ to $\cJ_{n,n'}'$. 

In reverse, for any $(i,j,\alpha',\beta',\tau')\in \cJ_{n,n'}'$, we can also find a unique $(\alpha,\beta,\tau)\in \cJ_{n,n'}$ such that $\cM(\alpha,\beta,\tau)=(i,j,\alpha',\beta',\tau')$. This proves that the map $\cM$ is a bijection on $\cJ_{n,n'}$ with values in $\cJ_{n,n'}'$. The proof of this lemma is complete.
\end{proof}
Denote by $|\cJ_{n,n'}|$ the number of elements in $\cJ_{n,n'}$. We have the next lemma as a consequence of Lemma \ref{lmm_cat}.
\begin{lemma}\label{lmm_ncj}
Let $\cJ_{n,n'}$ be defined as in \eqref{def_jkm} with some positive integer $n$ and nonnegative integer $n'\leq n-1$. Then,
\begin{align}\label{for_cjit}
|\cJ_{n,n'}|=\frac{n!(n-1)!}{2^{n'}(n-n')!(n-n'-1)!}\,, 
\end{align}
where by convention $0!=1$.
\end{lemma}
\begin{proof}
By definition, we know that $\cJ_{n,0}=\{\mathbf{0}_n,\partial,\partial\}$ for all $n\geq 1$. This yields that $|\cJ_{n,0}|=1$ which coincides with \eqref{for_cjit}. It suffices to show the case $n'\geq 1$. By Lemma \ref{lmm_cat}, we can write
\begin{align*}
|\cJ_{n,n'}|=\frac{1}{2}n(n-1)|\cJ_{n-1,n'-1}|.
\end{align*}
Then, \eqref{for_cjit} follows by iteration. The proof of this lemma is complete.
\end{proof}

\subsection{Proof of Theorem {\ref{coro_mnt}}}\label{ssec_pfmnt}
In this subsection, we provide the proof of Theorem \ref{coro_mnt}. In fact, we can show a generalized version of this theorem (see Proposition \ref{prop_mnt} below). Let us start this subsection by introducing the following notation.

Let $n\in\NN$, and let $n'\in\{1,\dots, n-1\}$. Fix $(\alpha,\beta,\tau)\in \cJ_{n,n'}$. For any $\bt_{n}\in \mathbb{T}_n^T$, $\bs_{n'}\in \mathbb{T}_{n'}^{t_n}$, $\bx_{n}\in \RR^n$ and $\bz_{n'}\in \RR^{n'}$, we define 
the following expressions
\begin{align}\label{def_A}
A_{n,n'}^{\alpha}(\bt_n,\bx_n)=\prod_{i=1}^{n}\Big(\int_{\RR}dzp_{t_i}(x_i-z)X_0(z)\Big)^{1-\alpha_i},
\end{align}
\begin{align}\label{def_B}
B_{n,n'}^{\beta}(\bs_{n'},\bz_{n'})=\prod_{i=1}^{n'}\Big(\int_{\RR}dz p_{s_i}(z_i-z)X_0(z)\Big)^{1-\beta_i},
\end{align}
\begin{align}
C_{n,n'}^{\alpha,\tau}(\bt_n,\bx_n,\bs_{n'},\bz_{n'})=\prod_{i=1}^{|\alpha|}p(t_{\iota_{\alpha}(i)}-s_{\tau(i)},x_{\iota_{\alpha}(i)}-z_{\tau(i)}),
\end{align}
\begin{align}\label{def_D}
D_{n,n'}^{\beta,\tau}(\bs_{n'},\bz_{n'})=\prod_{i=|\alpha|+1}^{2n'}p(s_{\iota_{\beta}(i-|\alpha|)}-s_{\tau(i)}, z_{\iota_{\beta}(i-|\alpha|)}-z_{\tau(i)}),
\end{align}
and
\begin{align}\label{def_E}
E_{n,n'}^{\alpha,\beta}(\bs_{n'},\bz_{n'})=\prod_{i=1}^{n'}\widehat{\sigma}(s_i,z_i)^2=\prod_{i=1}^{n'}\sigma(s_i,z_i,\PP_{X_{s_i}(z_i)})^2.
\end{align}
 By convention, we write 
\[
A_{n,0}^{\alpha}(\bt_n,\bx_n)=\prod_{i=1}^{n}\Big(\int_{\RR}dzp_{t_i}(x_i-z)X_0(z)\Big),
\]
and $B^{\beta}_{n,0}=B^{\alpha,\tau}_{n,0}=D^{\beta,\tau}_{n,0}=E^{\alpha,\beta}_{n,0}=1$.  

\begin{proposition}\label{prop_mnt}
Suppose that $X_0\in\cM_F(\RR)$ satisfies Hypothesis \ref{hyp_x0}. Let $n\in\NN$, and let $Y$ be given as in \eqref{def_yst}. Then, for any $\bt_n\in\mathbb{T}_n^T$, and $\bx_n \in\RR^n$,
\begin{align}\label{for_mmt}
\EE \Big[\prod_{i=1}^{n}Y^{t_i}_{t_n}(x_i) \Big]=&\sum_{n'=0}^{n-1}\sum_{(\alpha,\beta,\tau)\in \cJ_{n,n'}}A_{n,n'}^{\alpha}(\bt_n,\bx_n)\int_{\mathbb{T}_{n'}^{t_n}}d\bs_{n'}\int_{\RR^{n'}}d\bz_{n'}B_{n,n'}^{\beta}(\bs_{n'},\bz_{n'})\nonumber\\
&\times C_{n,n'}^{\alpha,\tau}(\bt_n,\bx_n,\bs_{n'},\bz_{n'})D_{n,n'}^{\beta,\tau}(\bs_{n'},\bz_{n'})E_{n,n'}^{\alpha,\beta}(\bs_{n'},\bz_{n'}),
\end{align}
where $\cJ_{n}$ and $A$ - $E$ are defined as in \eqref{def_jkm} and \eqref{def_A}-\eqref{def_E} respectively, $\mathbb{T}_{n}^t$ is defined as in \eqref{def_dtk}.
\end{proposition}
\begin{proof}
We prove this proposition by induction in $n$. If $n=1$, it is clear that $n'=0$ and thus
\begin{align*}
\EE \big[Y^{t_1}_{t_1}(x_1)\big]=&\EE \big[X_{t_1}(x_1)\big]=\int_{\RR}dz p_{t_1}(x_1-z)X_0(z),
\end{align*}
coincides with \eqref{for_mmt}. Suppose that $n\geq 2$. Recall that $Y$ satisfies the mild formulation \eqref{def_yst}. Thus, one can deduce by It\^{o}'s formula that
\begin{align}\label{for_mmt1}
&\EE \Big[\prod_{i=1}^{n}Y^{t_i}_{t_n}(x_i)\Big]=I_0+I_1,
\end{align}
where
\begin{align}
I_0=&\prod_{i=1}^{n}\int_{\RR}dzp_{t_i}(x_i-z)X_0(z)=A_{n,n'}^{\alpha}(\bt_n,\bx_n)\Big|_{n'=0},
\end{align}
and
\begin{align}\label{def_I1}
I_1=&\sum_{1\leq k_1<k_2\leq n}\int_0^{t_n}ds\int_{\RR}dz p_{t_{k_1}-s}(x_{k_1}-z)p_{t_{k_2}-s}(x_{k_2}-z)\widehat{\sigma}(s,z)^2\\
&\times  \EE \bigg[Y_s^s(z)\prod_{\substack{1\leq i\leq n\\ i\notin \{k_1,k_2\}}}Y^{t_i}_{s}(x_i)\bigg].\nonumber
\end{align}
Applying the induction hypothesis, we can write the expectation in \eqref{def_I1} as follows,
\begin{align}\label{for_mnt3}
\EE \bigg[&Y_s^s(z)\prod_{\substack{1\leq i\leq n\\ i\notin \{k_1,k_2\}}}Y^{t_i}_{s}(x_i)\bigg]=\sum_{n'=1}^{n-1}\sum_{(\alpha',\beta',\tau')\in \cJ_{n-1,n'-1}}A_{n-1,n'-1}^{\alpha'}\big((\bt_n^{k_1,k_2},s),(\bx_n^{k_1,k_2},z)\big)\\
&\times \int_{\mathbb{T}_{n'-1}^{t_n}}d\bs_{n'-1}\int_{\RR^{n'-1}}d\bz_{n'-1}B_{n-1,n'-1}^{\beta'}(\bs_{n'-1},\bz_{n'-1})\nonumber\\
 &\qquad\times C_{n-1,n'-1}^{\alpha',\tau'}\big((\bt_n^{k_1,k_2},s),(\bx_n^{k_1,k_2},z),\bs_{n'-1},\bz_{n'-1}\big)D_{n-1,n'-1}^{\beta',\tau'}(\bs_{n'-1},\bz_{n'-1})\nonumber\\
&\qquad\times E_{n-1,n'-1}^{\alpha',\beta'}(\bs_{n'-1},\bz_{n'-1}),\nonumber
\end{align}
where $\bt_n^{k_1,k_2}=(t_1,\dots, t_{k_1-1},t_{k_1+1},\dots, t_{k_2-1},t_{k_2+1},\dots, t_n)\in [0,T]^{n-2}$ and $\bx_n^{k_1,k_2}$ is defined in the same way.

Let $\cM$ be the bijection defined as in Lemma \ref{lmm_cat}. Choose $(\alpha,\beta,\tau)\in \cJ_{n,n'}$ with $n'\geq 1$. Let $(k_1,k_2,\alpha',\beta',\tau')=\cM(\alpha,\beta,\tau)\in \cJ_{n,n'}'$. Then, due to Lemma \ref{lmm_cat}, there exist $1\leq j_1<j_2\leq |\alpha|$ such that $\iota_{\alpha}(j_1)=k_1$, $\iota_{\alpha}(j_2)=k_2$, with $\tau(j_1)=\tau(j_2)=1$. This also yields that $\alpha_{k_1}=\alpha_{k_2}=1$. Recall that $\alpha'$, $\beta'$ and $\tau'$ are defined as in \eqref{def_alph'}-\eqref{def_tau'2} respectively. As a result, we deduce that
\begin{align}\label{for_ita}
&A^{\alpha'}_{n-1,n'-1}\big((\bt_n^{k_1,k_2},s),(\bx_n^{k_1,k_2},z)\big)\\
=&\prod_{\substack{1\leq i\leq n\\i\notin \{k_1,k_2\}}}\Big(\int_{\RR}dzp_{t_i}(x_i-z)X_0(z)\Big)^{1-\alpha_i}\Big(\int_{\RR}dy p_{s}(z-y)X_0(y)\Big)^{1-\beta_1}\nonumber\\
=&A_{n,n'}^{\alpha}(\bt_n,\bx_n)\Big(\int_{\RR}dy p_{s}(z-y)X_0(y)\Big)^{1-\beta_1},\nonumber
\end{align}
\begin{align}\label{for_itb}
&B_{n,n'}^{\beta}\big((s,\bs_{n'-1}),(z,\bz_{n'-1})\big)\\
=&\Big(\int_{\RR}dy p_{s}(z-y)X_0(y)\Big)^{1-\beta_1}\prod_{i=2}^{n'}\Big(\int_{\RR}dy p_{s_{i-1}}(z_{i-1}-y)X_0(y)\Big)^{1-\beta_i}\nonumber\\
=&B_{n-1,n'-1}^{\beta'}(\bs_{n'-1},\bz_{n'-1})\Big(\int_{\RR}dy p_{s}(z-y)X_0(y)\Big)^{1-\beta_1},\nonumber
\end{align}
\begin{align}\label{for_itc}
&C_{n,n'}^{\alpha,\tau}\big(\bt_n,\bx_n,(s,\bs_{n'-1}),(z,\bz_{n'-1})\big)p(s-s_{\tau(|\alpha|+1)-1},z-z_{\tau(|\alpha|+1)-1})^{\beta_1}\\
=&p(s-s_{\tau'(|\alpha'|)},z-z_{\tau'(|\alpha'|)})^{\beta_1}\prod_{i=1}^{|\alpha'|-\beta_1}p(t_{\iota_{\alpha'}(i)}-s_{\tau'(i)}, x_{\iota_{\alpha'}(i)}-z_{\tau'(i)})\nonumber\\
&\times p(t_{k_1}-s,x_{k_1}-z)p(t_{k_2}-s,x_{k_2}-z)\nonumber\\
=&C_{n-1,n'-1}^{\alpha',\tau'}\big((\bt_n^{k_1,k_2},s),(\bx_n^{k_1,k_2},z),\bs_{n'-1},\bz_{n'-1}\big)\nonumber\\
&\times p(t_{k_1}-s,x_{k_1}-z)p(t_{k_2}-s,x_{k_2}-z),\nonumber
\end{align}
\begin{align}\label{for_itd}
&D_{n,n'}^{\beta,\tau}\big((s,\bs_{n'-1}),(z,\bz_{n'-1})\big)\\
=&p(s-s_{\tau(|\alpha|+1)-1},z-z_{\tau(|\alpha|+1)-1})^{\beta_1}\prod_{i=|\alpha'|+1}^{|\alpha'|+|\beta'|}p(s_{\iota_{\beta'}(i-|\alpha'|)}-s_{\tau'(i)}, z_{\iota_{\beta'}(i-|\alpha|)}-z_{\tau(i)})\nonumber\\
=&D_{n-1,n'-1}^{\beta',\tau'}(\bs_{n'-1},\bz_{n'-1})p(s-s_{\tau(|\alpha|+1)-1},z-z_{\tau(|\alpha|+1)-1})^{\beta_1},\nonumber
\end{align}
and
\begin{align}\label{for_ite}
E_{n,n'}^{\alpha,\beta}\big((s,\bs_{n'-1}),(z,\bz_{n'-1})\big)=\widehat{\sigma}(s,z)^2E_{n-1,n'-1}^{\alpha',\beta'}(\bs_{n'-1},\bz_{n'-1}).
\end{align}
Combining equations \eqref{def_I1}-\eqref{for_ite}, we get
\begin{align}\label{propmnt1}
I_1=&\sum_{n'=1}^{n-1}\sum_{(\alpha,\beta,\tau)\in \cJ_{n,n'}}A_{n,n'}^{\alpha}(\bt_n,\bx_n)\int_{\mathbb{T}_{n'}^{t_n}}d\bs_{n'}\int_{\RR^{n'}}d\bz_{n'}B_{n,n'}^{\beta}(\bs_{n'},\bz_{n'})\\
& \times C_{n,n'}^{\alpha,\tau}(\bt_n,\bx_n,\bs_{n'},\bz_{n'}) D_{n,n'}^{\beta,\tau}(\bs_{n'},\bz_{n'})E_{n,n'}^{\alpha,\beta}(\bs_{n'},\bz_{n'}).\nonumber
\end{align}
Therefore, formula \eqref{for_mmt} follows from \eqref{for_mmt1}-\eqref{def_I1} and \eqref{propmnt1} and Lemma \ref{lmm_cat}. The proof of this Proposition is complete.
\end{proof}
Having Proposition \ref{prop_mnt}, Theorem \ref{coro_mnt} follows immediately.
\begin{proof}[Proof of Theorem \ref{coro_mnt}]
Taking $(t_1,x_1)=\dots=(t_n,x_n)=(t,x)$ as in Proposition \ref{prop_mnt} and writing $A$ - $E$ explicitly using \eqref{def_A}-\eqref{def_E}, then we get equality \eqref{for_mmt0}. This completes the proof of Theorem \ref{coro_mnt}.
\end{proof}

\subsection{Some estimates}\label{ssec_mmtest}
In this subsection, we provide some estimates for expressions related to moments of $X_t(x)$. They will be used in the proof of the uniqueness of solutions to equation \eqref{sbm} under certain hypotheses (see Section \ref{sec_unq}).

\begin{lemma}\label{lmm_gau}
Suppose that $X_0\in \cM_F(\RR)$ satisfies Hypothesis \ref{hyp_x0}. Let $n\geq 2$ be a positive integer, and let $n'\in \{1,\dots, n-1\}$. Fix $(\alpha,\beta,\tau)\in \cJ_{n,n'}$. Let $B$ - $D$ be given as in \eqref{def_B}-\eqref{def_D}, $\bt_n\in \mathbb{T}^T_n$ and $\bx_n\in \RR$.  Then, for any $1\leq j\leq n'$, and $s\in (0,t_n)$,
\begin{align}\label{lmmgau1}
\int_{\mathbb{T}^{s}_{n'-j}}d\bs_{j+1:n'}\int_{\mathbb{T}^{s,t_n}_{j-1}}d\bs_{j-1}&\int_{\RR^{n'}}d\bz_{n'}B_{n,n'}^{\beta}(\bs_{n'}^j(s),\bz_{n'})C_{n,n'}^{\alpha,\tau}(\bt_n,\bx_n,\bs_{n'}^j(s),\bz_{n'})\\
&\times D_{n,n'}^{\beta,\tau}(\bs_{n'}^j(s),\bz_{n'})\leq \frac{c_1c_2^n(t_n-s)^{\frac{1}{2}j-1}s^{\frac{1}{2}(n'-j)}}{\Gamma(\frac{1}{2}j)\Gamma(\frac{1}{2}(n'-j)+1)},\nonumber
\end{align}
where $\bs_{j+1:n'}=(s_{j+1},\dots, s_{n'})$,  $\bs_{k'}^j(s)=(s_1,\dots, s_{j-1},s, s_{j+1},\dots, s_{k'})$, $\mathbb{T}^{s,t_n}_{j-1}=\{(s_1,\dots, s_{j-1}):s\leq s_{j-1}\leq \dots \leq s_1\leq t_n\}$, and $c_1,c_2>0$ depending on $\|X_0\|_{\infty}$.
\end{lemma}
\begin{proof}
Denote by $LHS$ the left hand side of \eqref{lmmgau1}. We prove this lemma by induction in $n$. First, we prove \eqref{lmmgau1} for  $n=2$. We can write $\cJ_{2,1}=\{(\alpha,\beta,\tau)\}$, where $\alpha=(1,1)$, $\beta=(0)$ and $\tau:\{1,2\}\to \{1\}$ is given by $\tau(1)=\tau(2)=1$. Thus, under Hypothesis \ref{hyp_x0}, 
\begin{align*}
LHS=&\int_{\RR}dz_1\Big(\int_{\RR}dzp_s(z_1-z)X_0(z)\Big) p_{t_1-s}(x_1-z_1)p_{t_2-s}(x_2-z_1)\\
\leq& (2\pi)^{-\frac{1}{2}}\|X_0\|_{\infty}(t_1+t_2-2s)^{-\frac{1}{2}}\leq (4\pi)^{-\frac{1}{2}}\|X_0\|_{\infty}(t_2-s)^{-\frac{1}{2}}.
\end{align*}
This proves inequality \eqref{lmmgau1} for $n=2$.

In the next step, we prove inequality \eqref{lmmgau1} for any  $n>2$. Choose $n'\in \{1,\dots,n-1\}$. Let $(\alpha,\beta,\tau)\in \cJ_{n,n'}$, and let $(k_1,k_2,\alpha',\beta',\tau')=\cM(\alpha,\beta,\tau)$ with $\cM$ defined as in Lemma \ref{lmm_cat}.

 Assume that $j> 1$. Then, due to \eqref{for_itb}-\eqref{for_itd}, we have
\begin{align}\label{ine_jg1}
LHS=&\int_{\mathbb{T}^{s}_{n'-j}}d\bs_{j+1:n'}\int_{\mathbb{T}^{s,t_n}_{j-1}}d\bs_{j-1}\int_{\RR^{n'}}d\bz_{n'}\Big(\int_{\RR}dy p_{s_1}(z_1-y)X_0(y)\Big)^{1-\beta_1}\\
&\times p(t_{k_1}-s_1,x_{k_1}-z_1)p(t_{k_2}-s_1,x_{k_2}-z_1)B_{n-1,n'-1}^{\beta'}(\bs_{2:n'-1}^j(s),\bz_{2:n'})\nonumber\\
&\times C_{n-1,n'-1}^{\alpha',\tau'}\big((\bt_n^{k_1,k_2},s_1),(\bx_n^{k_1,k_2},z_1),\bs_{2:n'}^j(s),\bz_{2:n'}) D_{n-1,n'-1}^{\beta',\tau'}(\bs_{2:n'}^j(s),\bz_{2:n'}),\nonumber
\end{align}
where $\bs_{2:n'}^j(s)=(s_2,\dots, s_{j-1},s,s_{j-2},\dots,s_{n'})$.
Notice that the induction hypothesis implies that
\begin{align*}
&\int_{\mathbb{T}^{s}_{n'-j}}d\bs_{j+1:n'}\int_{\mathbb{T}^{s,s_1}_{j-2}}d\bs_{2:j-1}\int_{\RR^{n'-1}}d\bz_{2:n'}B_{n-1,n'-1}^{\beta'}(\bs_{2:n'-1}^j(s),\bz_{2:n'})\nonumber\\
&\times C_{n-1,n'-1}^{\alpha',\tau'}\big((\bt_n^{k_1,k_2},s_1),(\bx_n^{k_1,k_2},z_1),\bs_{2:n'}^j(s),\bz_{2:n'})D_{n-1,n'-1}^{\beta',\tau'}(\bs_{2:n'}^j(s),\bz_{2:n'})\\
\leq& \frac{c_1c_2^{n-1}(s_1-s)^{\frac{1}{2}(j-1)-1}s^{\frac{1}{2}(n'-j)}}{\Gamma(\frac{1}{2}(j-1))\Gamma(\frac{1}{2}(n'-j)+1)}.
\end{align*}
Combining this fact with the boundedness of $X_0$, we obtain the next inequality immediately,
\begin{align*}
LHS\leq &\frac{c_1c_2^{n-1}s^{\frac{1}{2}(n'-j)}}{\Gamma(\frac{1}{2}(j-1))\Gamma(\frac{1}{2}(n'-j)+1)}\int_s^{t_n} ds_1 (s_1-s)^{\frac{1}{2}(j-1)-1}\\
&\times\int_{\RR}dz_1p(t_{k_1}-s_1,x_{k_1}-z_1)p(t_{k_2}-s_1,x_{k_2}-z_1)\\
\leq &\frac{(4\pi)^{-\frac{1}{2}}c_1c_2^{n-1}s^{\frac{1}{2}(n'-j)}}{\Gamma(\frac{1}{2}(j-1))\Gamma(\frac{1}{2}(n'-j)+1)}\int_s^{t_n} ds_1 (s_1-s)^{\frac{1}{2}(j-1)-1}(t_n-s_1)^{-\frac{1}{2}}\\
=&\frac{(4\pi)^{-\frac{1}{2}}\Gamma(\frac{1}{2})c_1c_2^{n-1}s^{\frac{1}{2}(n'-j)}(t_n-s)^{\frac{1}{2}j-1}}{\Gamma(\frac{1}{2}j)\Gamma(\frac{1}{2}(n'-j)+1)}\leq \frac{c_1c_2^{n}s^{\frac{1}{2}(n'-j)}(t_n-s)^{\frac{1}{2}j-1}}{\Gamma(\frac{1}{2}j)\Gamma(\frac{1}{2}(n'-j)+1)},
\end{align*}
provided that $c_2\geq (4\pi)^{-\frac{1}{2}}\Gamma(\frac{1}{2})$.

On the other hand, if $j=1$, we can write the following equation analogous to \eqref{ine_jg1},
\begin{align*}
LHS&=\int_{\mathbb{T}^{s}_{n'-j}}d\bs_{j+1:n'}\int_{\mathbb{T}^{s,t_n}_{j-1}}d\bs_{j-1}\int_{\RR^{n'}}d\bz_{n'}\Big(\int_{\RR}dy p_{s_1}(z_1-y)X_0(y)\Big)^{1-\beta_1}\nonumber\\
&\times p(t_{k_1}-s_1,x_{k_1}-z_1) p(t_{k_2}-s_1,x_{k_2}-z_1)B_{n-1,n'-1}^{\beta'}(\bs_{2:n'-1},\bz_{2:n'})\nonumber\\
&\times C_{n-1,n'-1}^{\alpha',\tau'}\big((\bt_n^{k_1,k_2},s_1),(\bx_n^{k_1,k_2},z_1),\bs_{2:n'},\bz_{2:n'}\big) D_{n-1,n'-1}^{\beta',\tau'}(\bs_{2:n'},\bz_{2:n'}).
\end{align*}
By using the induction hypothesis again, we deduce that
\begin{align*}
&\int_{\mathbb{T}^{s_2}_{n'-2}}d\bs_{3:n'}\int_{\RR^{n'-1}}d\bz_{2:n'}B_{n-1,n'-1}^{\beta'}\big((s_2,\bs_{3:n'-1}^j(s)),\bz_{2:n'})\nonumber\\
&\times C_{n-1,n'-1}^{\alpha',\tau'}\big((\bt_n^{k_1,k_2},s_1),(\bx_n^{k_1,k_2},z_1),(s_2,\bs_{3:n'}),\bz_{2:n'}\big)D_{n-1,n'-1}^{\beta',\tau'}(\bs_{2:n'}^j(s),\bz_{2:n'})\\
\leq& \frac{c_1c_2^{n-1}(s-s_2)^{-\frac{1}{2}}s_2^{\frac{1}{2}(n'-2)}}{\Gamma(\frac{1}{2})\Gamma(\frac{1}{2}n')}.
\end{align*}
As a consequence, we have
\begin{align*}
LHS\leq &\frac{(4\pi)^{-\frac{1}{2}}c_1c_2^{n-1}(t_n-s)^{-\frac{1}{2}}}{\Gamma(\frac{1}{2})\Gamma(\frac{1}{2}n')}\int_0^s ds_2 (s-s_2)^{-\frac{1}{2}}s_2^{\frac{1}{2}(n'-2)}\leq \frac{c_1c_2^{n}(t_n-s)^{-\frac{1}{2}}s^{\frac{1}{2}(n'-1)}}{\Gamma(\frac{1}{2})\Gamma(\frac{1}{2}(n'-1)+1)},
\end{align*}
if $c_2\geq (4\pi)^{-\frac{1}{2}}\Gamma(\frac{1}{2})$. This completes the proof of this lemma.
\end{proof}

\begin{remark}\label{rmk_gau}
From the proof of Lemma \ref{lmm_gau}, we see that the term $B$ only contributes in $c_1c_2^{n}$. It can be relaxed a little bit, namely,
\begin{align*}
B_{n,n'}^{\beta}(\bs_{n'},\bz_{n'})=\prod_{i=1}^{n'}\Big(\int_{\RR}dz p_{s_i}(z_i-z)f_i(z)\Big)^{1-\beta_i},
\end{align*}
with $\{f_i\}_{i\geq 1}$ being a sequence of nonnegative functions on $\RR$ such that  $\sup_{i\geq 1}\|f_i\|_{\infty}<\infty$.
Then, inequality \eqref{lmmgau1} still holds with constants depending on $\sup_{i\geq 1}\|f_i\|_{\infty}$ instead of $\|X_0\|_{\infty}$.
\end{remark}

As a consequence of Lemma \ref{lmm_gau}, we have the next proposition immediately.
\begin{proposition}\label{lmm_bunif}
Assume   that $X_0\in\cM_F(\RR)$ satisfies  Hypothesis \ref{hyp_x0} and let $X=\{X_t(x):(t,x)\in [0,T]\times \RR\}$ be a solution to equation \eqref{sbm}. Then, 
\begin{align}\label{ine_bunif}
\sup_{(t,x)\in [0,T]\times \RR}\EE [X_t(x)^n]\leq c_1c_2^n(n!)^{\frac{3}{2}},
\end{align}
with constants $c_1,c_2>0$ independent of $n$.
\end{proposition}
\begin{proof}
The case $n=1$ is trivial. Suppose that $n\geq 2$. Consider moment formula \eqref{for_mmt0}. If $n'=0$, then $\cJ_{n,0}=\{(\mathbf{0}_n,\partial,\partial)\}$. This implies that the corresponding summand is
\[
\Big(\int_{\RR}dzp_t(x-z)X_0(z)\Big)^n\leq \|X_0\|_{\infty}^n,
\]
under Hypothesis \ref{hyp_x0}. Additionally, combining this result with Lemmas \ref{lmm_ncj} and \ref{lmm_gau}, we can write
\begin{align*}
\EE [X_t(x)^n] \leq &\|X_0\|_{\infty}^n+ c_1c_2^n\sum_{n'=1}^{n-1}\sum_{j=1}^{n'}\bigg(\frac{n!(n-1)!}{(n-n')!(n-n'-1)!\Gamma(\frac{1}{2}j) \Gamma(\frac{1}{2}(n'-j)+1)}\\
&\hspace{35mm} \times \int_0^t(t-s)^{\frac{1}{2}j-1}s^{\frac{1}{2}(n'-j)}ds\bigg)\\
=&\|X_0\|_{\infty}^n+ c_1c_2^n\sum_{n'=1}^{n-1}\frac{n!(n-1)!}{(n-n')!(n-n'-1)! \Gamma(\frac{1}{2}n'+1)}t^{\frac{1}{2}n'}.
\end{align*}
By Stirling's formula (c.f. Jameson \cite[Theorem 1]{mg-15-jameson}), one can show that, for any $n\in\NN$ and $n'\in \{1,\dots, n-1\}$,
\[
\frac{n!(n-1)!}{(n-n')!(n-n'-1)! \Gamma(\frac{1}{2}n'+1)\frac{1}{2}n'}\leq c_1c_2^n\Gamma\Big(\frac{3}{2}n'+1\Big)\leq c_1c_2^n(n!)^{\frac{3}{2}},
\]
where constants $c_1$ and $c_2$ are independent of $n$ 
and may vary from line to line. Thus, inequality \eqref{ine_bunif} follows immediately.
\end{proof}

\section{Proof of the uniqueness}\label{sec_unq}
In this section, we prove the weak uniqueness for equation \eqref{sbm}, or equivalently for MP \eqref{mart1} and \eqref{mart2}, under certain conditions. In the classical theory of Markov processes, there are several approaches to this question. By the method of duality (c.f. \cite{wiley-86-ethier-kurtz,ap-98-mytnik}), one can obtain the well-posedness of the martingale problem by proving the uniqueness of its Laplace transformation (log-Laplace equation). Besides, the desired uniqueness result can be obtained by studying corresponding historical processes (c.f. \cite{ap-96-overbeck,mams-95-perkins}). In recent years, a new approach was introduced by Xiong \cite{ap-13-xiong} that connects the weak uniqueness for MP \eqref{mart1} and \eqref{mart2} to the strong uniqueness of solutions to a backward doubly SDE. This method was successfully employed for nonlinear Mckean-Vlasov MPs (c.f. \cite{arxiv-21-ji-xiong-yang,ijm-15-mytnik-xiong}).

In this paper, the classic duality, log-Laplace-equation method will be adapted
to prove weak uniqueness. 
However, the log-Laplace equation for \eqref{sbm} depends on $\sigma$ and thus on the distribution of solution(s) to equation \eqref{sbm}. 
Hence, it appears we are not able to show the uniqueness for the log-Laplace equation without knowing that for \eqref{sbm} itself. 
To address this issue, we introduce the following two alternative hypotheses. Under either hypothesis, we can show that as a function of $(t,x)$, $\widehat{\sigma}(t,x)=\sigma(t,x,\PP_{X_{t}(x)})$ is invariant for any solution $X$ to equation \eqref{sbm}. This implies the log-Laplace equation for any solution to \eqref{sbm} is unique. The well-posedness of \eqref{sbm} is thus straightforward.

\subsection{Proof of the uniqueness part of Theorem \ref{thm_exs} under Hypothesis \ref{hyp_fn}}\label{sec_unqfn}
In this subsection, we prove the weak uniqueness for equation \eqref{sbm} under Hypothesis \ref{hyp_fn}. Notice that under Hypothesis \ref{hyp_fn}, $\sigma$ depends only on $t,x$ and the moments of $X_t(x)$ up to order $N$. 
The weak uniqueness for the equation will reduce to the uniqueness for moments of solutions up to order $N$.

Let $X=\{X_t(x):(t,x)\in[0,T]\times \RR\}$ be a solution to \eqref{sbm}, and let $u:[0,T]\times \RR\to \RR^N$ be given by
\[
u_n(t,x)=\EE [X_t(x)^n],
\]
for all $n=1,\dots, N$ and $(t,x)\in [0,T]\times \RR$. Then, by Theorem \ref{coro_mnt}, $u$ is a solution to the following integral equation with initial condition $u(0,x)=(X_0(x),\dots, X_0(x)^N)$,
\begin{align}\label{equ_u}
u_n(t,&x)=\sum_{n'=0}^{n-1}\sum_{(\alpha,\beta,\tau)\in \cJ_{n,n'}}\prod_{i=1}^n\Big(\int_{\RR}dz p_{t}(x-z)X_0(z)\Big)^{1-\alpha_i}\\
&\times \int_{\mathbb{T}_{n'}}d\bs_{n'}\int_{\RR^{n'}}d\bz_{n'}\prod_{i=1}^{n'}\Big(\int_{\RR}dz p_{s_i}(z_i-z)X_0(z)\Big)^{1-\beta_i} \prod_{i=1}^{\alpha}p(t-s_{\tau(i)},x-z_{\tau(i)})\nonumber\\
&\times\prod_{i=|\alpha|+1}^{2n'}p(s_{\iota_{\beta}(i-|\alpha|)}-s_{\tau(i)},z_{\iota_{\beta}(i-|\alpha|)}-z_{\tau(i)})\prod_{i=1}^{n'}f(s_i,z_i,u(s_i,z_i)),\nonumber
\end{align} 
for all $n=1,\dots, N$. In the next proposition, we show the uniqueness of solutions to \eqref{equ_u}.
\begin{proposition}\label{prop_unq}
Suppose that $X_0\in\cM_F(\RR)$ satisfies Hypothesis \ref{hyp_x0}. Then, equation \eqref{equ_u} has a unique 
solution in $C_b([0,T]\times \RR;\RR_+^N)$.
\end{proposition}
\begin{proof}
The existence in $C_b([0,T]\times \RR; \RR_+^N)$ follows from  Theorem \ref{thm_exs} and Proposition \ref{lmm_bunif}. It suffices to show the uniqueness. Let $v$ be another solution to \eqref{equ_u}. Then, by mean value theorem, for any $n'\in\NN$, and $(\bs_{n'},\bz_n')\in \mathbb{T}^t_{n'}\times \RR^{n'}$,
\begin{align*}
&\prod_{i=1}^{n'}f(s_i,z_i,u(s_i,z_i))-\prod_{i=1}^{n'}f(s_i,z_i,v(s_i,z_i))\\
=&\sum_{k=1}^N\sum_{j=1}^{n'}\prod_{\substack{1\leq i\leq n'\\i\neq j}}f(s_i,z_i,\xi_1^i,\dots,\xi_N^i)\frac{\partial}{\partial \xi_k}f(s_i,z_i,\xi_1^j,\dots, \xi_N^j) [u_k(s_j,z_j)-v_k(s_j,z_j)]\\
\leq& \|f\|_{1,\infty}^{n'}\sum_{k=1}^N\sum_{j=1}^{n'}|u_k(s_j,z_j)-v_k(s_j,z_j)|,
\end{align*}
where $\xi_k^i$ is between $ u_k(s_i,z_i)$ and $ v_k(s_i,z_i)$ for all $1\leq i\leq n'$ and $1\leq k\leq N$ and the last inequality is due to Hypothesis \ref{hyp_fn}. Combining this inequality with equation \eqref{equ_u}, we get
\begin{align*}
|u_n(t,x)&-v_n(t,x)|\leq \sum_{n'=0}^{n-1}\sum_{(\alpha,\beta,\tau)\in \cJ_{n,n'}}\sum_{k=1}^N\sum_{j=1}^{n'}\|f\|_{1,\infty}^{n'}\prod_{i=1}^n\Big(\int_{\RR}dz p_{t}(x-z)X_0(z)\Big)^{1-\alpha_i}\nonumber\\
&\times \int_{\mathbb{T}_{n'}}d\bs_{n'}\int_{\RR^{n'}}d\bz_{n'}\prod_{i=1}^{n'}\Big(\int_{\RR}dz p_{s_i}(z_i-z)X_0(z)\Big)^{1-\beta_i} \prod_{i=1}^{\alpha}p(t-s_{\tau(i)},x-z_{\tau(i)})\nonumber\\
&\times\prod_{i=|\alpha|+1}^{2n'}p(s_{\iota_{\beta}(i-|\alpha|)}-s_{\tau(i)},z_{\iota_{\beta}(i-|\alpha|)}-z_{\tau(i)})|u_k(s_j,z_j)-v_k(s_j,z_j)|.
\end{align*}
By Lemma \ref{lmm_gau}, we deduce that,
\begin{align}\label{propunq1}
\sup_{x\in \RR}|u_n(t,x)-v_n(t,x)|\leq &C \int_0^{t}ds\sum_{n'=1}^{n-1}\sum_{j=1}^{n'}\frac{(t-s)^{\frac{1}{2}j-1}s^{\frac{1}{2}(n'-j)}}{\Gamma(\frac{1}{2}j)\Gamma(\frac{1}{2}(n'-j)+1)}\\
&\times \sum_{k=1}^N\sup_{x\in \RR}|u_k(s,x)-v_k(s,x)|,\nonumber
\end{align}
for some constant $C>0$ depends on $N$ and $\|f\|_{1,\infty}$. On the other hand, it is clear that
\[
u_1(t,x)=v_1(t,x)=\int_{\RR}dz p_t(x-z)X_0(z),\ \forall (t,x)\in [0,T]\times \RR.
\]
Thus taking the summation among $n=1,\dots,N$ on both sides of \eqref{propunq1}, and noticing that 
\[
(t-s)^{\frac{1}{2}j-1}s^{\frac{1}{2}(n'-j)}\leq (T+1)^{\frac{1}{2}(N-1)}(t-s)^{-\frac{1}{2}}
\]
for all $1\leq j\leq n'$, $1\leq n'\leq n-1\leq N-1$ and $0\leq s\leq t\leq T$, and
\[
\sup_{1\leq j\leq n'\leq n-1\leq N}\Big[\Gamma\Big(\frac{1}{2}j\Big)\Gamma\Big(\frac{1}{2}(n'-j)+1\Big)\Big]^{-1}<\infty,
\]
we can write
\begin{align}\label{propunq2}
h(t):=\sum_{n=1}^N\sup_{x\in \RR}|u_n(t,x)-v_n(t,x)|\leq C_{T,N} \int_0^{t}ds(t-s)^{-\frac{1}{2}} h(s).
\end{align}
for some universal constant $C_{T,N}>0$. As a consequence of a generalized Gr\"{o}nwall inequality (c.f. Ye et al. \cite[Theorem 1]{jmaa-07-ye-gao-ding}), we get $h(t)=0$ for all $t\in [0,T]$. This proves the uniqueness of solutions to equation \eqref{equ_u}, and the proof of this proposition is complete.
\end{proof}

\begin{proof}[Proof of Theorem \ref{thm_exs}: uniqueness under Hypothesis \ref{hyp_fn}]
It suffices to show the weak uniqueness. 
Fix  $(t,x)\in[0,T]\times \RR$.  By Proposition \ref{prop_unq}, we know that       $(\EE [X_t(x)],\dots, \EE [X_t(x)^N]$ remains the same for   any solution $X$ to \eqref{sbm}. This allows us to define the following deterministic function $\widehat{\sigma}$ on $[0,T]\times \RR_+$ given by 
\[
\widehat{\sigma}(t,x)=\sqrt{f(\EE [X_t(x)],\dots,\EE [X_t(x)^N])}.
\]
Then, any weak solution $X$ to \eqref{sbm} is also a weak solution to the following SPDE  
\begin{align}\label{sbm3}
\partial_t X_t(x) =\frac{1}{2} \Delta X_t(x) + \widehat{\sigma}(t,x)\sqrt{X_t(x)} \dot{W}(t, x).
\end{align}
Following the standard arguments, one can show that the log-Laplace equation for \eqref{sbm3} is 
\begin{align}\label{equ_lle}
\begin{cases}
\frac{\partial }{\partial t}v_t(x)=\frac{1}{2}\Delta v_t(x)-\frac{1}{2}\widehat{\sigma}(t,x)^2v_t(x)^2,\\
v_0(x)=\phi(x),
\end{cases}
\end{align}
such that
\[
\EE \big[-\exp(\langle X_t,\phi\rangle)\big]=\exp(-\langle X_0,v_t\rangle),
\]
for any nonnegative function $\phi\in \cS(\RR)$. It is known that equation \eqref{equ_lle} has a unique solution (c.f. Dawson \cite[Sections 4.3 and 4.4]{lnm-93-dawson}, and also Engl\"{a}nder and Pinsky \cite{jde-03-englander-pinsky} for a detailed study about this type of equations). This yields that the probability law of $X$ as a measure-valued process is unique. We complete the proof of this theorem.
\end{proof}

\subsection{Proof of the uniqueness part of Theorem \ref{thm_exs} under Hypothesis \ref{hyp_ifn}}\label{ssec_ifn}
Following the idea   in Section \ref{sec_unqfn}, consider the following infinite  dimensional equation   for  $u=\{u_n(t,x):n\in\{1,2,\dots\},(t,x)\in [0,T]\times \RR\}$:
\begin{align}\label{equ_uifn}
u_n(t,x)=&\sum_{n'=1}^{n-1}\sum_{(\alpha,\beta,\tau)\in \cJ_{n,n'}}\prod_{k=1}^n\Big(\int_{\RR}dz p_{t}(x-z)X_0(z)\Big)^{1-\alpha_k}\nonumber\\
&\times \int_{\mathbb{T}_{n'}}d\bs_{n'}\int_{\RR^{n'}}d\bz_{n'}\prod_{i=1}^{n'}\Big(\int_{\RR}dz p_{s_i}(z_i-z)X_0(z)\Big)^{1-\beta_i} \prod_{i=1}^{|\alpha|}p(t-s_{\tau(i)},x-z_{\tau(i)})\nonumber\\
&\times\prod_{i=|\alpha|+1}^{2n'}p(s_{\iota_{\beta}(i-|\alpha|)}-s_{\tau(i)},z_{\iota_{\beta}(i-|\alpha|)}-z_{\tau(i)})\prod_{i=1}^{k'}f(s_i,z_i,u(s_i,z_i)).
\end{align} 
Let $\cH$ be the  Hilbert space of real sequence with inner product defined by  \eqref{def_ch}. As a consequence of Theorem \ref{coro_mnt} and Lemma \ref{lmm_bunif}, equation \eqref{equ_uifn} has a solution $u$ such that for every $(t,x)\in [0,T]\times \RR$, $u(t,x)$ belongs to the Hilbert space $\cH$. Therefore, 
assume Hypothesis \ref{hyp_ifn}, the weak uniqueness is a direct result of the following Proposition \ref{prop_unqif}.

\begin{proposition}\label{prop_unqif}
Suppose that $X_0\in\cM_F(\RR)$ satisfies Hypothesis \ref{hyp_x0}. Then, under Hypothesis \ref{hyp_ifn}, equation \eqref{equ_uifn} has a unique solution in $C_b([0,T]\times \RR;\cH_+)$.
\end{proposition}
\begin{proof}
It suffices to show the uniqueness. Suppose $v$ is another solution to \eqref{equ_uifn}. By similar arguments to those  in Proposition \ref{prop_unq}, taking account of Lemmas \ref{lmm_ncj} and \ref{lmm_gau}, Stirling's formula and Jensen's inequality, we can write
\begin{align*}
\|u(t,x)-&v(t,x)\|_{\cH}^2=\sum_{n=1}^{\infty}\frac{1}{(n!)^{2\gamma}}|u_n(t,x)-v_n(t,x)|^2\\
\leq &\sum_{n=1}^{\infty}\frac{c_1c_2^n}{(n!)^{2\gamma}}\bigg[\sum_{n'=1}^{n-1}\sum_{(\alpha,\beta,\tau)\in \cJ_{n,n'}}\int_0^tds \frac{(t-s)^{\frac{1}{2}n'-1}}{\Gamma(\frac{1}{2}n'+1)}\sup_{y\in \RR}\|u(s,y)-v(s,y)\|_{\cH}\bigg]^2\\
\leq &\sum_{n=1}^{\infty}\frac{c_1c_2^n}{(n!)^{2\gamma}}\bigg[\sum_{n'=1}^{n-1}\frac{n!(n-1)!}{2^{n'}(n-n')!(n-n'-1)!}\int_0^t ds \sum_{j=1}^{n'}\frac{(t-s)^{\frac{1}{2}j-1}s^{\frac{1}{2}(n'-j)}}{\Gamma(\frac{1}{2}j)\Gamma(\frac{1}{2}(n'-j)+1)}\\
&\qquad\qquad\quad\times\sup_{y\in \RR}\|u(s,y)-v(s,y)\|_{\cH}\bigg]^2\\
\leq & \sum_{n=1}^{\infty}\frac{c_1c_2^n}{\Gamma[(2\gamma-3)n]}\int_0^tds (t-s)^{-\frac{1}{2}}\sup_{y\in \RR}\|u(s,y)-v(s,y)\|_{\cH}^2.
\end{align*}

Denote by $h(t)=\sup_{x\in \RR}\|u(t,x)-v(t,x)\|_{\cH}^2$. It is clear that the summation in $n$ in the above expression is finite. Thus,
\begin{align*}
h(t)\leq C(\gamma,T)\int_0^t ds (t-s)^{-\frac{1}{2}}h(s),
\end{align*}
with some constant $C(\gamma,T)$ depending on $\gamma$ and $T$. Then, a generalized 
Gr\"{o}nwall inequality implies that $h\equiv 0$ and thus $u\equiv v$. The proof of this proposition is complete.
\end{proof}

\subsection{Examples}
 In Sections \ref{sec_extc}, \ref{sec_unqfn} and \ref{ssec_ifn}, we proved the existence and uniqueness of solutions to the mean-field sBm \eqref{sbm} under certain hypotheses.  It is natural to ask for some real examples for the function $\sigma$, such that the hypotheses we proposed are satisfied. 
For simplicity, we assume $\sigma(t,x,\mu)=\sigma(\mu)$ is only a function of
the probability measure.
A  typical example is that
\begin{align}\label{def_sgmg}
\sigma(\mu)=\int_{\RR}g(x)\mu(dx)
\end{align}
for some function $g$. Thus we provide some examples for function $g$ in \eqref{def_sgmg}, such that Hypothesis \ref{hyp_fn} or \ref{hyp_ifn} is satisfied, which implies the weak existence and uniqueness of solutions to \eqref{sbm} via the approach in this paper.

\begin{example}\label{exp_1}
Let $g$ be a polynomial on $\RR_+$ given by
\[
g(x)=\sum_{k=0}^Na_kx^k
\]
for all $x\in \RR_+$ with some constants $a_0,\dots, a_N\in\RR_+$. Let $h:\RR\times \RR_+$ be a Lipschitz function that is uniformly bounded by two positive constants.  Define function $\sigma:\sP(\RR_+)\to \RR_+$ as follows,
\[
\sigma(\mu)^2=h\Big(\int_{\RR}g(x)\mu(dx)\Big)=h\Big(a_0+\sum_{k=1}^Na_k\EE [X_{\mu}^k]\Big),
\]
for all $\mu\in \sP(\RR_+)$, where $X_{\mu}\sim \mu$. Then, it is clear that Hypothesis \ref{hyp_fn} holds in this example.
\end{example}
\begin{example}\label{exp_2}
Let $g(x)=2+\cos (x^{\frac{1}{2}})$ for all $x\in \RR_+$ and let $\sigma(\mu)^2=\int_{\RR}g(x)\mu(dx)$ for all $\mu\in \sP(\RR_+)$. Then, we see that
\[
g(x)=3+\sum_{n=1}^{\infty}\frac{(-1)^{n}x^n}{(2n)!}
\]
and thus $\sigma:\sP(\RR_+)\to \RR_+$, given by
\[
\sigma(\mu)^2=\int_{\RR}g(x)\mu(dx)=3+\sum_{n=1}^{\infty}\frac{(-1)^{n}}{(2n)!}\EE [X_{\mu}^n]
\]
for all $\mu\in \sP(\RR_+)$ with $X_{\mu}\sim \mu$. Choose $\gamma\in (\frac{3}{2},2)$ (see \eqref{def_ch}). Then, by using Cauchy-Schwarz's inequality, we can show that Hypothesis \ref{hyp_ifn} holds,
\begin{align*}
|\sigma(\mu)^2-\sigma(\nu)^2|=&\Big|\sum_{n=1}^{\infty}\frac{(-1)^{n}}{(2n)!}\big(\EE [X_{\mu}^n] - \EE [X_{\nu}^n]\big)\Big| \\
\leq &\Big(\sum_{n=1}^{\infty}\frac{1}{[(2n)!]^{2-\gamma}}\Big)^{\frac{1}{2}}\Big(\frac{1}{[(2n)!]^{\gamma}}\big(\EE [X_{\mu}^n] - \EE [X_{\nu}^n]\big)^2\Big)^{\frac{1}{2}}\leq C_{\gamma}\|X_{\mu}^{\cH}-X_{\nu}^{\cH}\|_{\cH},
\end{align*}
for any $\mu,\nu\in \sP(\RR_+)$ such that $X_{\mu}^{\cH}=(\EE [X_{\mu}],\EE [X_{\mu}^2],\dots)\in \cH_+$ and $X_{\nu}^{\cH}\in \cH_+$.
\end{example}

\section{Regularity for moments of the solution}\label{sec_mmtreg}
Assume Hypothesis \ref{hyp_fn}. Suppose also that $\sigma(t,x,\mu)=\sigma(\mu)$ depends  only on $\mu$ for simplicity. 
Let $X$ be the solution to \eqref{sbm}. In this section, we will study the regularity for all the moments of $X_t(x)$.

Before the rigorous proof, let us take a look at the following example. Let $n=2$ as in Theorem \ref{coro_mnt}. Then, we can write
\begin{align*}
\EE [X_{t}(x)^2]=&\Big(\int_{\RR}dzp_{t}(x-z)X_0(z)\Big)^2\\
&+\int_0^{t}ds\int_{\RR^2}dydz p_{t-s}(x-z)^2p_s(z-y)\sigma(\PP_{X_{s}(z)})^2 X_0(y)X_0(z).
\end{align*}
Taking the derivative in $t$ on both sides, we will get $\delta_0^2$  (the square of the Dirac delta function)  as substituting $s=t$ of $p_{t-s}(x-z)^2$, which is difficult to handle. 
To avoid this singularity, we perform a change of variable $u=t-s$. 
Then, we need to compute the time derivative of $\sigma(\PP_{X_{t-u}}(z))^2$, which depends on all $\EE [X_{t-u}(z)^n]$, $n=1,\dots, N$. 
In order to write a convincing proof, we introduce the following Picard iteration for moments in Section \ref{ssec_pic}. 
The proof for our main moment regularity result, Theorem \ref{thm_rgl}, 
to follow in Section \ref{ssec_ram}, is based on this Picard iteration.

\subsection{Picard iteration for moments}\label{ssec_pic}
Recall that the diffusion coefficient in equation \eqref{sbm} involves a square root, that is not Lipschitz around $0$. 
Thus, it is difficult to find a sequence convergent to the solution to \eqref{sbm} using Picard iteration. 
Fortunately, the Picard iteration for the moments is convergent (see Proposition \ref{prop_pcm}), which is sufficient to study the regularity of the moments. 

Let $X^{(0)}=\{X_t^{(0)}(x): (t,x)\in [0,T]\times \RR\}$ be the unique (weak) solution to
\begin{align}\label{def_x0}
X_t^{(0)}(x)=\int_{\RR}dz p_t(x-z)X_0(z)+\int_0^t\int_{\RR}p_{t-s}(x-z)\sigma(\PP_{X_0(z)})\sqrt{X_s^{(0)}(z)}W(ds,dz),
\end{align}
and for all $k=1,2,\dots$, let $X^{(k)}=\{X_t^{(k)}(x): (t,x)\in [0,T]\times \RR\}$ be the unique solution to
\begin{align}\label{def_xk}
X_t^{(k)}(x)=\int_{\RR}dzp_t(x-z)X_0(z)+\int_0^t\int_{\RR}p_{t-s}(x-z)\sigma\big(\PP_{X_s^{(k-1)}(z)}\big)\sqrt{X_s^{(k)}(z)}W(ds,dz).
\end{align}
Denote by $u_n^{(k)}(t,x)=\EE [X^{(k)}_t(x)^n]$ for all $n=1,\dots, N$ and $k=0,1,2,\dots$. Then, we have the following results analogue to Theorem \ref{coro_mnt},
\begin{align}\label{def_um0}
u_n^{(0)}(t,x)=&\sum_{n'=0}^{n-1}\sum_{(\alpha,\beta,\tau)\in \cJ_{n,n'}}\Big(\int_{\RR}dz p_{t}(x-z)X_0(z)\Big)^{n-|\alpha|}\nonumber\\
&\times \int_{\mathbb{T}_{n'}^t}d\bs_{n'}\int_{\RR^{n'}}d\bz_{n'}\prod_{i=1}^{n'}\Big(\int_{\RR}dz p_{s_i}(z_i-z)X_0(z)\Big)^{1-\beta_i} \prod_{i=1}^{|\alpha|}p(t-s_{\tau(i)},x-z_{\tau(i)})\nonumber\\
&\times\prod_{i=|\alpha|+1}^{2n'}p(s_{\iota_{\beta}(i-|\alpha|)}-s_{\tau(i)},z_{\iota_{\beta}(i-|\alpha|)}-z_{\tau(i)})\prod_{i=1}^{n'}f(X_0(z_i),\dots, X_0(z_i)^N)
\end{align}
and
\begin{align}\label{def_umk}
u_n^{(k)}(t,&x)=\sum_{n'=0}^{n-1}\sum_{(\alpha,\beta,\tau)\in \cJ_{n,n'}}\Big(\int_{\RR}dz p_{t}(x-z)X_0(z)\Big)^{n-\alpha}\nonumber\\
&\times \int_{\mathbb{T}_{n'}^t}d\bs_{n'}\int_{\RR^{n'}}d\bz_{n'}\prod_{i=1}^{n'}\Big(\int_{\RR}dz p_{s_i}(z_i-z)X_0(z)\Big)^{1-\beta_i} \prod_{i=1}^{|\alpha|}p(t-s_{\tau(i)},x-z_{\tau(i)})\nonumber\\
&\times\prod_{i=|\alpha|+1}^{2n'}p(s_{\iota_{\beta}(i-|\alpha|)}-s_{\tau(i)},z_{\iota_{\beta}(i-|\alpha|)}-z_{\tau(i)})\prod_{i=1}^{n'}f\big(u_1^{(k-1)}(s_i,z_i),\dots, u_N^{(k-1)}(s_i,z_i)\big),
\end{align}
for all $n=1,\dots,N$ and $k=1,2,\dots$. We will show the convergence of $\{u^{(k)}\}_{k\geq 0}$ in the next lemma.
\begin{lemma}\label{lmm_cvuk}
Let $\{u^{(k)}\}_{k\geq 0}$ be given as in \eqref{def_um0} and \eqref{def_umk}. Then, it is a convergent sequence in $C_b([0,T]\times \RR;\RR_+^N)$ equipped with the supremum norm.
\end{lemma}
\begin{proof}
The proof of this lemma is similar to Proposition \ref{prop_unq}. In fact, by using Theorem \ref{coro_mnt} and Lemma \ref{lmm_gau}, and the fact that for all $k=0,1,\dots$, and $(t,x)\in [0,T]\times \RR$,
\[
u_1^{(k)}(t,x)=\int_{\RR}dzp_t(x-z)X_0(z),
\] 
we can deduce the next inequality analogously to \eqref{propunq2},
\begin{align*}
h^{(k)}(t):=&\sum_{n=1}^N\sup_{x\in \RR}|u_n^{(k+1)}(t,x)-u_n^{(k)}(t,x)|\\
\leq & C\int_0^{t}ds(t-s)^{-\frac{1}{2}} h^{(k-1)}(s).
\end{align*} 
By iteration, we have
\begin{align}\label{ine_dk}
h^{(k)}(t)\leq C\int_{\mathbb{T}_{k}^t}d\bs_k(t-s_1)^{-\frac{1}{2}}(s_1-s_2)^{-\frac{1}{2}}\cdots (s_{k-1}-s_{k})^{-\frac{1}{2}}h^{(0)}(s_{k}).
\end{align}
Similar argument as in Lemma \ref{lmm_bunif} implies that
\begin{align*}
\sup_{t\in [0,T]}h^{(0)}(t)=&\sup_{t\in [0,T]}\sum_{n=1}^N\sup_{x\in \RR}\big|\EE [X_t^{(1)}(x)^n]-\EE [X^{(0)}_t(x)^n]\big| 
\end{align*}
is finite.  
Now, it follows from Hu et al. \cite[Lemma 4.5]{ejp-15-hu-huang-nualart-tindel} that
\begin{align}\label{ine_pchk}
h^{(k)}(t)\leq C\int_{\mathbb{T}_{k}^t}d\bs_k(t-s_1)^{-\frac{1}{2}}(s_1-s_2)^{-\frac{1}{2}}\cdots (s_{k-1}-s_{k})^{-\frac{1}{2}}\leq \frac{C^{k}t^{\frac{1}{2}k}}{\Gamma(\frac{1}{2}k+1)}.
\end{align}
Finally, by the asymptotic bound of the Mittag-Leffler function (c.f. Kilbas et al. \cite[Formula (1.8.10)]{elsevier-06-kilbas-strivastava-trujillo}), there exist positive constants $c_1$ and $c_2$ such that
\[
\sum_{k=0}^{\infty}h^{(k)}(t)\leq c_1e^{c_2C_T^2},
\]
for all $t\in [0,T]$. As a consequence, $\{u^{(k)}\}_{k\geq 0}$ is a Cauchy sequence in $C_b([0,T]\times \RR;\RR^N)$ under the supremum norm. The proof of this lemma is complete.
\end{proof}
\begin{proposition}\label{prop_pcm}
Assume Hypothesis  \ref{hyp_fn} with $\sigma(t,x,\mu)=\sigma(\mu)$ depending only on $\mu$. Let $X$ be the solution to \eqref{sbm} with initial constitution $X_0\in \cM_F(\RR)$ satisfying Hypothesis \ref{hyp_x0}. Let $X^{(k)}$ be given by  \eqref{def_x0} and \eqref{def_xk}. Then, for any $n=1,2,\dots$,
\begin{align}\label{prop_pclm}
\lim_{k\to \infty} \EE\big[(X_t^{(k)}(x))^n\big]= \EE [X_t(x)^n].
\end{align}
uniformly in $(t,x)\in [0,T]\times \RR$.
\end{proposition}
\begin{proof}
Let $u^{(k)}$ be defined as in \eqref{def_um0} and \eqref{def_umk}
and $u$ be the limit of $u^{(k)}$ in $C_b([0,T]\times \RR;\RR_+^N)$ as $k\to \infty$. 
Then, by a common argument, we conclude that $u$ is the solution to \eqref{equ_u}. As a result, \eqref{prop_pclm} is true for $n\in \{1,\dots, N\}$. On the other hand, suppose $n>N$. Using moment formula \eqref{for_mmt0} and Lemma \ref{lmm_gau}, we deduce that
\begin{align*}
\sup_{x\in \RR}\big|\EE\big[(X_t^{(k)}(x))^n\big]-\EE [X_t(x)^n]\big|\leq &C_{n,T}\int_0^tds(t-s)^{-\frac{1}{2}}\|u-u^{(k)}\|_{\infty}\\
=&2C_{n,T}t^{\frac{1}{2}}\|u-u^{(k)}\|_{\infty},
\end{align*}
where $C_{n,T}$ is a constant depending on $n$ and $T$. This proves \eqref{prop_pclm} for all $n>N$. The proof of this proposition is complete.
\end{proof}

\subsection{Regularity analysis for moments}\label{ssec_ram}
In this section, we will prove that the solution $u$ to \eqref{equ_u} is differentiable in both time and spatial arguments with uniformly bounded derivatives.
\begin{lemma}\label{lmm_duk}
Let $X_0\in H_{2,2}(\RR)\cap C_b^2(\RR)$. Assume Hypothesis \ref{hyp_fn} and assume that $\sigma(t,x,\mu)=\sigma(\mu)$ depends only on $\mu$. For any $k=0,1,\dots$, let $u^{(k)}=\{u^{(k)}_n(t,x):n=1,\dots, N, (t,x)\in [0,T]\times \RR\}$ be defined iteratively as in \eqref{def_um0} and \eqref{def_umk}. Then, $u^{(k)}$ is differentiable in time for all $(t,x)\in (0,T]\times \RR$ and in space for all $(t,x)\in [0,T]\times \RR$.
\end{lemma}
\begin{proof}
We only prove the differentiability in time by verifying the next inequality by induction. 
\begin{align}\label{lmm_duk1}
\sum_{n=1}^N\sup_{x\in \RR}\Big|\frac{\partial}{\partial t}u^{(k)}_n(t,x)\Big|\leq \sum_{i=0}^{k}\frac{C_0^{i+1}}{\Gamma(\frac{1}{2}(i+1))}t^{\frac{1}{2}(i-1)},
\end{align}
for some universal constant $C_0>0$ depends  on $N$, $\|X_0\|_{2,\infty}$, $\|f\|_{1,\infty}$ and $T$. The proof of spatial differentiability can be done in a similar way.

{\bf Step 1.} Assume $k=0$. Recalling moment formula \eqref{def_um0}, in order to estimate the derivative of $u_n^{(0)}$, it suffices to estimate that for every summand in \eqref{def_um0}. Choose $(\alpha,\beta,\tau)\in \cJ_{n,n'}$. Then, we have
\begin{align}\label{for_iabg}
I_{\alpha,\beta,\tau}^{n,0}:=&\prod_{i=1}^{n}\Big(\int_{\RR}dz p_{t}(x-z)X_0(z)\Big)^{1-\alpha_i} \int_{\mathbb{T}_{n'}}d\bs_{n'}\int_{\RR^{n'}}d\bz_{n'}\prod_{i=1}^{n'}\Big(\int_{\RR}dz p_{s_i}(z_i-z)X_0(z)\Big)^{1-\beta_i} \nonumber\\
&\times\prod_{i=1}^{|\alpha|}p(t-s_{\tau(i)},x-z_{\tau(i)})\prod_{i=|\alpha|+1}^{2n'}p(s_{\iota_{\beta}(i-|\alpha|)}-s_{\tau(i)},z_{\iota_{\beta}(i-|\alpha|)}-z_{\tau(i)})\nonumber\\
&\times \prod_{i=1}^{n'}f(X_0(z_i),\dots, X_0(z_i)^N)\nonumber\\
=&I_{\alpha,\beta,\tau}^{n,0,1}\times I_{\alpha,\beta,\tau}^{n,0,2},
\end{align}
where
\begin{align*}
I_{\alpha,\beta,\tau}^{n,0,1}=&\prod_{i=1}^{n}\Big(\int_{\RR}dz p_{t}(x-z)X_0(z)\Big)^{1-\alpha_i} 
\end{align*}
and performing a change of variables $r_i=t-s_i$ for all $i=1,\dots,n'$,
\begin{align*}
I_{\alpha,\beta,\tau}^{n,0,2}=& \int_{\widehat{\mathbb{T}}^{t}_{n'}}d\br_{n'}\int_{\RR^{n'}}d\bz_{n'}\prod_{i=1}^{n'}\Big(\int_{\RR}dz p_{t-r_i}(z_i-z)X_0(z)\Big)^{1-\beta_i}\prod_{i=1}^{|\alpha|}p(r_{\tau(i)},x-z_{\tau(i)}) \nonumber\\
&\times\prod_{i=|\alpha|+1}^{2n'}p(r_{\tau(i)}-r_{\iota_{\beta}(i-|\alpha|)},z_{\tau(i)}-z_{\iota_{\beta}(i-|\alpha|)})\prod_{i=1}^{n'}f(X_0(z_i),\dots, X_0(z_i)^N),
\end{align*}
with $\widehat{\mathbb{T}}^{t}_{n'}=\{\br_{n'}=(r_1,\dots, r_{n'}): 0<r_1<\dots <r_{n'}<t\}$. Firstly, it is clear that
\begin{align}\label{for_piabg1}
\Big|\frac{\partial}{\partial t}I_{\alpha,\beta,\tau}^{n,0,1}\Big|=&(n-|\alpha|)\Big(\int_{\RR}dz p_{t}(x-z)X_0(z)\Big)^{n-|\alpha|-1}\int_{\RR}dz p_{t}(x-z)\Delta X_0(z)\nonumber\\
\leq &(n-|\alpha|)\|X_0\|_{2,\infty}^{n-|\alpha|}.
\end{align}
Additionally, by Lemma \ref{lmm_gau}, we have
\begin{align}
\big|I_{\alpha,\beta,\tau}^{n,0,2}\big|\leq c_1c_2^n\sum_{j=1}^{n'}\int_0^tds\frac{(t-s)^{\frac{1}{2}j-1}s^{\frac{1}{2}(n'-j)}}{\Gamma(\frac{1}{2}j)\Gamma(\frac{1}{2}(n'-j)+1)}\leq \frac{c_1c_2^n t^{\frac{1}{2}n'}}{\Gamma(\frac{1}{2}n'+2)}\,.
\end{align}
In the next step, we need to compute $\frac{\partial}{\partial t}I^{n,0,2}_{\alpha,\beta,\tau}$. If $n'=0$, we have $I_{\alpha,\beta,\tau}^{n,0,2}=1$, and thus $\frac{\partial}{\partial t} I_{\alpha,\beta,\tau}^{n,0,2}=0$.
On the other hand, suppose that $n'\geq 1$. Then, we have
\begin{align}\label{for_piabg2}
\frac{\partial}{\partial t}I^{n,0,2}_{\alpha,\beta,\tau}=J_1+J_2,
\end{align}
where
\begin{align*}
J_1=&\int_{\widehat{\mathbb{T}}^{t}_{n'-1}}d\br_{n'-1}\int_{\RR^{n'}}d\bz_{n'}\prod_{i=1}^{n'}\Big(\int_{\RR}dz p_{t-r_i}(z_i-z)X_0(z)\Big)^{1-\beta_i}\prod_{i=1}^{|\alpha|}p(r_{\tau(i)},x-z_{\tau(i)}) \nonumber\\
&\times\prod_{i=|\alpha|+1}^{2n'}p(r_{\tau(i)}-r_{\iota_{\beta}(i-|\alpha|)},z_{\tau(i)}-z_{\iota_{\beta}(i-|\alpha|)})\prod_{i=1}^{n'}f(X_0(z_i),\dots, X_0(z_i)^N)\Big|_{r_{n'}=t},
\end{align*}
and
\begin{align*}
J_2=&\int_{\widehat{\mathbb{T}}^{t}_{n'}}d\br_{n'}\int_{\RR^{n'}}d\bz_{n'}\frac{\partial}{\partial t}\bigg[\prod_{i=1}^{n'}\Big(\int_{\RR}dz p_{t-r_i}(z_i-z)X_0(z)\Big)^{1-\beta_i}\bigg]\prod_{i=1}^{|\alpha|}p(r_{\tau(i)},x-z_{\tau(i)}) \nonumber\\
&\times\prod_{i=|\alpha|+1}^{2n'}p(r_{\tau(i)}-r_{\iota_{\beta}(i-|\alpha|)},z_{\tau(i)}-z_{\iota_{\beta}(i-|\alpha|)})\prod_{i=1}^{n'}f(X_0(z_i),\dots, X_0(z_i)^N).
\end{align*}
Using Lemma \ref{lmm_gau} and observing that $\|f\|_{\infty}<\infty$, we find
for $n'\leq n-1$ that
\begin{align}
|J_1|\leq \int_0^tds \frac{c_1c_2^n(t-s)^{\frac{1}{2}n'-1}}{\Gamma(\frac{1}{2}n')}\leq \frac{c_1c_2^n t^{\frac{1}{2}n'}}{\Gamma(\frac{1}{2}n'+1)}.
\end{align}
In the next step, we write the derivative explicitly as follows
\begin{align*}
J_2=&\int_{\widehat{\mathbb{T}}^{t}_{n'}}d\br_{n'}\int_{\RR^{n'}}d\bz_{n'}\bigg[\sum_{j=1}^{n'}\prod_{\substack{1\leq i\leq n'\\i\neq j}}\Big(\int_{\RR}dz p_{t-r_i}(z_i-z)X_0(z)\Big)^{1-\beta_i}\\
&\times \int_{\RR}dz p_{t-r_j}(z_j-z)\Delta X_0(z)\1_{\beta_j=0}\bigg]\prod_{i=1}^{|\alpha|}p(r_{\tau(i)},x-z_{\tau(i)}) \nonumber\\
&\times\prod_{i=|\alpha|+1}^{2n'}p(r_{\tau(i)}-r_{\iota_{\beta}(i-|\alpha|)},z_{\tau(i)}-z_{\iota_{\beta}(i-|\alpha|)})\prod_{i=1}^{n'}f(X_0(z_i),\dots, X_0(z_i)^N)\,.
\end{align*}
Then, it follows from Remark \ref{rmk_gau} that
\begin{align}\label{for_piabg22}
|J_2|\leq c_1c_2^n\int_0^tds\frac{(t-s)^{\frac{1}{2}n'-1}}{\Gamma(\frac{1}{2}n'+1)}\leq \frac{c_1c_2^n t^{\frac{1}{2}n'}}{\Gamma(\frac{1}{2}n'+2)}.
\end{align}
Combining inequalities \eqref{for_iabg}-\eqref{for_piabg22}, we get that
\begin{align*}
\Big|\frac{\partial}{\partial t}I_{\alpha,\beta,\tau}^{n,0}\Big|=\Big|I_{\alpha,\beta,\tau}^{n,0,1}\frac{\partial}{\partial t}I_{\alpha,\beta,\tau}^{n,0,2}+I_{\alpha,\beta,\tau}^{n,0,2}\frac{\partial}{\partial t}I_{\alpha,\beta,\tau}^{n,0,1}\Big|\leq \frac{c_1c_2t^{\frac{1}{2}n'-1}}{\Gamma(\frac{1}{2}n'+1)}\,. 
\end{align*}
Taking account of Lemma \ref{lmm_ncj}, it follows that
\begin{align*}
\Big|\frac{\partial}{\partial t}u^{(0)}_n(t,x)\Big|=\Big|\sum_{n'=0}^{n-1}\sum_{(\alpha,\beta,\tau)\in \cJ_{n,n'}}\frac{\partial}{\partial t}I_{\alpha,\beta,\tau}^{n,0}\Big|\leq C_{N,T}t^{-\frac{1}{2}}.
\end{align*}

{\bf Step 2.} Let $k\geq 1$. Then, we can write
\begin{align}\label{lmm_duk2}
\frac{\partial}{\partial t}u^{(k)}_n(t,x)=\sum_{n'=0}^{n-1}\sum_{(\alpha,\beta,\tau)\in \cJ_{n,n'}}\Big(I_{\alpha,\beta,\tau}^{n,k,1} \frac{\partial}{\partial t}I_{\alpha,\beta,\tau}^{n,k,2}+I_{\alpha,\beta,\tau}^{n,k,2}\frac{\partial}{\partial t} I_{\alpha,\beta,\tau}^{n,k,1}\Big)\,, 
\end{align}
where
\begin{align*}
I_{\alpha,\beta,\tau}^{n,k,1}=I_{\alpha,\beta,\tau}^{n,0,1}=&\prod_{i=1}^n\Big(\int_{\RR}dz p_{t}(x-z)X_0(z)\Big)^{1-\alpha_i}
\end{align*}
and 
\begin{align*}
I_{\alpha,\beta,\tau}^{n,k,2}= &\int_{\widehat{\mathbb{T}}_{n'}^t}d\br_{n'}\int_{\RR^{n'}}d\bz_{n'}\prod_{i=1}^{n'}\Big(\int_{\RR}dz p_{t-r_i}(z_i-z)X_0(z)\Big)^{1-\beta_i} \nonumber\\
&\times\prod_{i=1}^{|\alpha|}p(r_{\tau(i)},x-z_{\tau(i)})\prod_{i=|\alpha|+1}^{2n'}p(r_{r_{\tau(i)}-\iota_{\beta}(i-|\alpha|)}, z_{\tau(i)}-z_{\iota_{\beta}(i-|\alpha|)})\\
&\times \prod_{i=1}^{n'}f(u^{(k-1)}(t-r_i,z_i)).
\end{align*}
By \eqref{for_piabg1} and Lemma \ref{lmm_gau}, we get that,
\begin{align}\label{lmm_duk11}
\sum_{n'=0}^{n-1}\sum_{(\alpha,\beta,\tau)\in \cJ_{n,n'}}\frac{\partial}{\partial t} I_{\alpha,\beta,\tau}^{n,k,1}I_{\alpha,\beta,\tau}^{n,k,2}\leq C_{N,T}
\end{align}
for some constant $C_{N,T}$ independent of $k$. On the other hand, fixing $(\alpha,\beta,\tau)\in \cJ_{n,n'}$, we see that
\begin{align}\label{for_piabtk}
\frac{\partial}{\partial t}I_{\alpha,\beta,\tau}^{n,k,2}= & J_1'+J_2'+J_3'\,, 
\end{align}
where
\begin{align*}
J_1'=J_1'(k)=&\int_{\widehat{\mathbb{T}}_{n'-1}^t}d\br_{n'-1}\int_{\RR^{n'}}d\bz_{n'}\prod_{i=1}^{n'}\Big(\int_{\RR}dz p_{t-r_i}(z_i-z)X_0(z)\Big)^{1-\beta_i} \nonumber\\
&\times \prod_{i=1}^{|\alpha|}p(r_{\tau(i)}, x-z_{\tau(i)}) \prod_{i=|\alpha|+1}^{2n'}p(r_{\tau(i)}-r_{\iota_{\beta}(i-|\alpha|)}, z_{\tau(i)}-z_{\iota_{\beta}(i-|\alpha|)})\\
&\times \prod_{i=1}^{n'}f(u^{(k-1)}(t-r_i,z_i))\bigg|_{r_{n'}=t},
\end{align*}
\begin{align*}
J_2'=J_2'(k)=&\int_{\widehat{\mathbb{T}}_{n'}^t}d\br_{n'}\int_{\RR^{n'}}d\bz_{n'}\frac{\partial}{\partial t}\bigg(\prod_{i=1}^{n'}\Big(\int_{\RR}dz p_{t-r_i}(z_i-z)X_0(z)\Big)^{1-\beta_i}\bigg)  \nonumber\\
&\times \prod_{i=1}^{|\alpha|}p(r_{\tau(i)},x-z_{\tau(i)}) \prod_{i=|\alpha|+1}^{2n'}p(r_{\tau(i)}-r_{\iota_{\beta}(i-|\alpha|)}, z_{\tau(i)}-z_{\iota_{\beta}(i-|\alpha|)})\\
&\times  \prod_{i=1}^{n'}f(u^{(k-1)}(t-r_i,z_i)),
\end{align*}
and
\begin{align*}
J_3'=J_3'(k)=&\int_{\widehat{\mathbb{T}}_{n'}^t}d\br_{n'}\int_{\RR^{n'}}d\bz_{n'}\prod_{i=1}^{n'}\Big(\int_{\RR}dz p_{t-r_i}(z_i-z)X_0(z)\Big)^{1-\beta_i} \prod_{i=1}^{|\alpha|}p(r_{\tau(i)},x-z_{\tau(i)})\nonumber\\
&\times\prod_{i=|\alpha|+1}^{2n'}p(r_{\tau(i)}- r_{\iota_{\beta}(i-|\alpha|)}, z_{\tau(i)}-z_{\iota_{\beta}(i-|\alpha|)})\frac{\partial}{\partial t}\Big(\prod_{i=1}^{n'}f(u^{(k-1)}(t-r_i,z_i))\Big).
\end{align*}
Notice that $J_1'$ and $J_2'$ are almost the same as $J_1$ and $J_2$ in Step 1, while  different terms $f(u^{(k-1)}(t-r_i,z_i))$, $i=1,\dots, n'$, can be simply bounded by $\|f\|_{\infty}$. Thus, we can write
\begin{align}\label{inq_j12'}
|J_1'|\leq C_{N,T}t^{-\frac{1}{2}}\quad \mathrm{and}\quad |J_2'|\leq C_{N,T}.
\end{align}
Furthermore, using Lemma \ref{lmm_gau} again, we can deduce that 
\begin{align}\label{inq_j3'}
|J_3'|\leq C_{N,T}\int_0^tds(t-s)^{-\frac{1}{2}}\sum_{l=1}^N\sup_{x\in \RR}\Big|\frac{\partial}{\partial s}u^{(k-1)}_l(s,x)\Big|.
\end{align}
As a result, combining \eqref{lmm_duk2}-\eqref{inq_j3'}, we have
\begin{align*}
\Big|\frac{\partial}{\partial t}u^{(k)}_n(t,x)\Big|\leq C_{N,T}\bigg(t^{-\frac{1}{2}}+\int_0^tds(t-s)^{-\frac{1}{2}}\sum_{l=1}^N\sup_{x\in \RR}\Big|\frac{\partial}{\partial s}u^{(k-1)}_l(s,x)\Big|\bigg).
\end{align*}
Using the induction hypothesis, we have
\begin{align*}
\sum_{n=1}^N\sup_{x\in \RR}\Big|\frac{\partial}{\partial t}u^{(k)}_n(t,x)\Big|\leq &NC_{N,T}\bigg(t^{-\frac{1}{2}}+\sum_{i=0}^{k-1}\frac{C_0^{i+1}}{\Gamma(\frac{1}{2}(i+1))}\int_{0}^tds(t-s)^{-\frac{1}{2}}s^{\frac{1}{2}(i-1)}\Big)\\
=&NC_{N,T}t^{-\frac{1}{2}}+NC_{N,T}\sum_{i=0}^{k-1}\frac{C_0^{i+1}\Gamma(\frac{1}{2})}{\Gamma(\frac{1}{2}i+1)}t^{\frac{1}{2}i}\leq \sum_{i=0}^{k}\frac{C_0^{i+1}}{\Gamma(\frac{1}{2}(i+1))}t^{\frac{1}{2}(i-1)},
\end{align*}
provided that $C_0\geq \sqrt{\pi}NC_{N,T}$. This proves inequality \eqref{lmm_duk1}. The proof of this lemma is complete.
\end{proof}

\begin{lemma}\label{lmm_udu}
Assume conditions in Lemma \ref{lmm_duk}. Then, 
\begin{enumerate}[(i)]
\item For any $s>0$, $\frac{\partial}{\partial t}u^{(k)}(t,x)$ is convergent uniformly on $[s,T]\times \RR$ as $k\to \infty$. 

\item $\frac{\partial}{\partial x}u^{(k)}(t,x)$ is convergent uniformly on $[0,T]\times \RR$ as $k\to \infty$.
\end{enumerate} 
\end{lemma}
\begin{proof}
We only show property (i). Property (ii) can be proved in a similar way. Recalling formula \eqref{lmm_duk2}, and noticing that $I_{\alpha,\beta,\tau}^{n,k,1}$ is invariant in $k$, we can write
\begin{align}\label{lmm_udu1}
\Big|\frac{\partial}{\partial t}\big(u^{(k+1)}_n(t,x)-u^{(k)}_n(t,x)\big)\Big|=\bigg|\sum_{n'=0}^{n-1}\sum_{(\alpha,\beta,\tau)\in \cJ_{n,n'}}\big(G_{\alpha,\beta,\tau}^{n,1}(k)+G_{\alpha,\beta,\tau}^{n,2}(k)\big)\bigg|,
\end{align}
where
\[
G_{\alpha,\beta,\tau}^{n,1}(k)=I_{\alpha,\beta,\tau}^{n,k,1} \frac{\partial}{\partial t}\big(I_{\alpha,\beta,\tau}^{n,k+1,2}-I_{\alpha,\beta,\tau}^{n,k,2}\big)
\]
and
\[
G_{\alpha,\beta,\tau}^{n,2}(k)=\big(I_{\alpha,\beta,\tau}^{n,k+1,2}-I_{\alpha,\beta,\tau}^{n,k,2}\big)\frac{\partial}{\partial t} I_{\alpha,\beta,\tau}^{n,k,1}.
\]
Due to Lemmas \ref{lmm_gau} and \ref{lmm_cvuk}, we can show that
\begin{align}
\sum_{n'=0}^{n-1}\sum_{(\alpha,\beta,\tau)\in \cJ_{n,n'}}\sup_{t\in [0,T]}\sup_{x\in\RR}\big|G_{\alpha,\beta,\tau}^{n,2}(k)\big|<c_1\frac{c_2^k}{\Gamma(\frac{1}{2}k+1)}.
\end{align}
On the other hand, using inequalities \eqref{for_piabtk} and \eqref{inq_j12'}, we get
\begin{align}\label{inq_gabt}
G_{\alpha,\beta,\tau}^{n,2}(k)\leq &C_{N,T}\Big[t^{-\frac{1}{2}}\sup_{(t,x)\in [0,T]\times\RR}|u^{(k)}(t,x)-u^{(k-1)}(t,x)|+H(n,k)\Big],
\end{align}
where
\begin{align*}
H(n,k)=&\int_{\widehat{\mathbb{T}}_{n'}^t}d\br_{n'}\int_{\RR^{n'}}d\bz_{n'}\prod_{i=1}^{n'}\Big(\int_{\RR}dz p_{t-r_i}(z_i-z)X_0(z)\Big)^{1-\beta_i}\prod_{i=1}^{|\alpha|}p(r_{\tau(i)},x-z_{\tau(i)})\\
&\ \times\prod_{i=|\alpha|+1}^{2n'}p(r_{\tau(i)}- r_{\iota_{\beta}(i-|\alpha|)}, z_{\tau(i)}-z_{\iota_{\beta}(i-|\alpha|)})\\
&\times \bigg|\frac{\partial}{\partial t}\Big(\prod_{i=1}^{n'}f(u^{(k)}(t-r_i,z_i))-\prod_{i=1}^{n'}f(u^{(k-1)}(t-r_i,z_i))\Big)\bigg|.
\end{align*}
By elementary calculus, we can show that
\begin{align*}
&\bigg|\frac{\partial}{\partial t}\Big(\prod_{i=1}^{n'}f(u^{(k)}(t-r_i,z_i))-\prod_{i=1}^{n'}f(u^{(k-1)}(t-r_i,z_i))\Big)\bigg|\\
\leq &\|f\|_{1,\infty}^{n'}\bigg(\sum_{l=1}^N\sup_{(r,x)\in [0,T]\times \RR}\big|u_l^{(k)}(r,x)-u_l^{(k-1)}(r,x)\big|\sum_{i=1}^{n'}\sum_{l=1}^N\sup_{x\in \RR}\Big|\frac{\partial}{\partial r}u^{(k)}_l(r_i,x)\Big|\\
&\qquad+\sum_{i=1}^{n'}\sum_{l=1}^{N}\sup_{x\in \RR}\Big|\frac{\partial}{\partial t}u_l^{(k)}(r_i,x)-\frac{\partial}{\partial t}u_l^{(k-1)}(r_i,x)\Big|\bigg).
\end{align*}
By inequality \eqref{lmm_duk1} and the asymptotic bound of Mittag-Leffler function, we know that for all $k=0,1,\dots$,
\[
\sup_{k\geq 0}\sum_{l=1}^N\sup_{x\in \RR}\Big|\frac{\partial}{\partial t}u^{(k)}_l(t,x)\Big|\leq c_1t^{-\frac{1}{2}}e^{c_2t}\leq C_{N,T}t^{-\frac{1}{2}}.
\]
Moreover, it follows from inequality \eqref{ine_pchk} that 
\begin{align*}
\sum_{n=1}^N\sup_{r\in [0,T]}\sup_{x\in \RR}|u_n^{(k)}(r,x)-u_n^{(k-1)}(r,x)|<c_1\frac{c_2^k}{\Gamma(\frac{1}{2}k+1)}.
\end{align*}
Therefore, using Lemma \ref{lmm_gau}
\begin{align}\label{lmm_udu2}
\sum_{n=1}^NH(n,k)\leq &c_1\Big(\frac{c_2^k}{\Gamma(\frac{1}{2}k+1)}\int_0^tdr(t-r)^{-\frac{1}{2}}r^{-\frac{1}{2}}\nonumber\\
&+\int_0^tdr(t-r)^{-\frac{1}{2}}\sum_{l=1}^{N}\sup_{x\in \RR}\frac{\partial}{\partial r}\big|u_l^{(k)}(r,x)-u_l^{(k-1)}(r,x)\big|\Big)\nonumber\\
\leq &\frac{c_1c_2^k}{\Gamma(\frac{1}{2}k+1)}+c_1\int_0^tdr(t-r)^{-\frac{1}{2}}\sum_{l=1}^{N}\sup_{x\in \RR}\frac{\partial}{\partial r}\big|u_l^{(k)}(r,x)-u_l^{(k-1)}(r,x)\big|.
\end{align}
Combining inequalities \eqref{lmm_udu1}-\eqref{lmm_udu2}, we can write
\begin{align*}
&\sum_{n=1}^{N}\sup_{x\in \RR}\Big|\frac{\partial}{\partial t}\big(u^{(k+1)}_n(t,x)-u^{(k)}_n(t,x)\big)\Big|\\
\leq &\frac{c_1c_2^kt^{-\frac{1}{2}}}{\Gamma(\frac{1}{2}k+1)}+c_1\int_0^tdr(t-r)^{-\frac{1}{2}}\sum_{n=1}^{N}\sup_{x\in \RR}\Big|\frac{\partial}{\partial r}\big(u_n^{(k)}(r,x)-u_n^{(k-1)}(r,x)\big)\Big|.
\end{align*}
Since 
\[
\Big|\frac{\partial}{\partial t}\big(u^{(1)}_n(t,x)-u^{(0)}_n(t,x)\big)\Big|\leq \Big|\frac{\partial}{\partial t}u^{(1)}_n(t,x)\Big|+\Big|\frac{\partial}{\partial t}u^{(0)}_n(t,x)\Big|
\]
is bounded uniformly in $(t,x)\in [s,T]\times \RR$,  we can deduce by iteration that
\begin{align*}
&\sum_{n=1}^{N}\sup_{x\in \RR}\Big|\frac{\partial}{\partial t}\big(u^{(k+1)}_n(t,x)-u^{(k)}_n(t,x)\big)\Big|\\
\leq &\frac{c_1c_2^kt^{-\frac{1}{2}}}{\Gamma(\frac{1}{2}k+1)}+\sum_{i=1}^{k-1}\frac{c_1^{i+1}c_2^{k-i}}{\Gamma(\frac{1}{2}(k-i)+1)}\int_{\widehat{\mathbb{T}}_i^t}d\br_i (t-r_1)^{-\frac{1}{2}}(r_1-r_2)^{-\frac{1}{2}}\cdots r_i^{-\frac{1}{2}}\\
\leq &c_1\frac{c_2^k}{\Gamma(\frac{1}{2}k+1)}t^{-\frac{1}{2}}+\sum_{i=1}^{k-1}\frac{c_1^{i+1}c_2^{k-i}}{\Gamma(\frac{1}{2}(k-i)+1)\Gamma(\frac{1}{2}i+1)}.
\end{align*}
By Stirling's formula, one can show that for all  $i,k\in\NN$ such that $1\leq i\leq k$,
\begin{align*}
\frac{\Gamma(\frac{1}{2}k+1)}{\Gamma(\frac{1}{2}(k-i)+1)\Gamma(\frac{1}{2}i+1)}\leq C^k
\end{align*}
for some universal constant $C$. Therefore, we can write
\begin{align*}
\sum_{n=1}^{N}\sup_{x\in \RR}\Big|\frac{\partial}{\partial t}\big(u^{(k+1)}_n(t,x)-u^{(k)}_n(t,x)\big)\Big|\leq c_1\frac{c_2^k}{\Gamma(\frac{1}{2}k+1)}\big(1+t^{-\frac{1}{2}}\big).
\end{align*}
Therefore, it follows from the asymptotic bound of Mittag-Leffler function that
\[
\sum_{k=1}^{\infty}\sum_{n=1}^{N}\sup_{x\in \RR}\Big|\frac{\partial}{\partial t}\big(u^{(k+1)}_n(t,x)-u^{(k)}_n(t,x)\big)\Big|\leq c_1e^{c_2} \big(1+t^{-\frac{1}{2}}\big).
\]
This proves that $\frac{\partial}{\partial t}u^{(k)}(t,x)$ is convergent uniformly on $[s,T]\times \RR$ for every $s\in (0,T]$. The proof of this lemma is complete.
\end{proof}
Combining Lemmas \ref{lmm_cvuk} and \ref{lmm_udu}, we get immediately the following theorem.
\begin{theorem}\label{thm_rgl}
Let $X_0\in H_{2,2}(\RR)\cap C_b^2(\RR)$. Assume Hypothesis \ref{hyp_fn} and assume that $\sigma(t,x,\mu)=\sigma(\mu)$ depends only on $\mu$. Let $X$ be a solution to \eqref{sbm} with initial condition $X_0$. Then, $\EE [X_t(x)^n]$ is differentiable at every $(t,x)\in (0,T]\times \RR$ for all $n\in \NN$.
\end{theorem}
\begin{proof}
Due to Lemmas \ref{lmm_cvuk} and \ref{lmm_udu}, we know that $\EE [X_t(x)^n]$ is differentiable for all $(t,x)\in (0,T]\times \RR$ and $n=1,\dots, N$. If $n> N$, then we apply Theorem \ref{coro_mnt}, and perform a changing 
of variable $u=t-s$. Then, the time differentiability of $\EE (X_t(x)^n)$ reduces to that of $\EE[(X_{t-u}(x))^n]$, $n=1,\dots, N$. This is already known. Thus we complete the proof of time regularity. The spatial regularity can be treated in a similar way. The proof of this theorem is complete.
\end{proof}

\end{document}